\documentclass{article}
\usepackage{amsmath}
\usepackage{amsfonts}
\usepackage{amsthm}
\usepackage{amssymb}
\usepackage{enumerate}

\newtheorem{thm}{Theorem}[section]
\newtheorem{prop}[thm]{Proposition}
\newtheorem{cor}[thm]{Corollary}
\newtheorem{lem}[thm]{Lemma}
\newtheorem{conj}[thm]{Conjecture}

\theoremstyle{definition}
\newtheorem{defn}[thm]{Definition}
\newtheorem{rmk}[thm]{Remark}

\begin{document}

\title{Generating Series of Key Polynomials and Bounded Ascending Sequences of Integers}

\author{Noah Cape and Shaul Zemel}

\maketitle

\section*{Introduction}

The Schubert polynomials are known to be the weighted counting functions for combinatorial objects called (reduced) RC-graphs, or (reduced) pipe dreams (see \cite{[BJS]}). One way of proving this is, as done in \cite{[FK]}, by defining their generating function in the appropriate Demazure algebra, and showing that the generating function of these weighted counting functions is determined by the same properties as the one of the Schubert polynomials (see also the appendix of \cite{[H]} for an extension of this proof to the $\beta$-Grothendieck polynomials, where the reduction assumption from the RC-graphs/pipe dreams is removed).

The key polynomials, also known as Demazure characters, were defined in \cite{[D1]} and \cite{[D2]} in order to describe characters of representations of the Borel subgroup inside the irreducible representations of general linear groups. They are also weighted counting functions of certain objects, like the skyline fillings from \cite{[Ma1]} and \cite{[Ma2]}, and are related to Gelfand--Tsetlin polytopes in \cite{[KST]}. They, like their building blocks called Demazure atoms (as defined in \cite{[LS]}), can be obtained as special cases of the non-symmetric Macdonald polynomials (see, e.g., \cite{[AS]}).

It is therefore interesting to see whether some of their properties are also visible using their generating function. Noting that these depend on two parameters, the coefficient of a fixing element of the Demazure algebra is still a generating function of key polynomials (rather than a single one). In this paper we show how these generating functions are closely related to bounded ascending sequences of integers. Similar investigations can be carried out for the Demazure atoms (though the polynomials in the numerators are much nicer for the key polynomials themselves, which is why we concentrate on them), as well as the generalization of the key polynomials called Lascoux polynomials, as defined in \cite{[L]} (they too decompose into Lascoux atoms---see \cite{[Mo]}). Note that \cite{[Y]} relates the latter to set-valued tableaux, and \cite{[BSW]} constructs colored lattice models associated to them. it will be interesting whether the results of these papers, when restricted to the key polynomials by setting $\beta=0$, can be related to ours.

We remark that all these families of polynomials are constructed, from some simple base cases, using modified divided difference operators that satisfy the braid relations---see \cite{[K]}, as well as the recent pre-print \cite{[Z]}, for more on this subject.

\smallskip

We now explain the results of our paper in more detail. A key polynomial, in some variables $\{x_{i}\}_{i=1}^{\infty}$, depends on a partition and a permutation (the latter takes the former to the weak composition, often used for indexing key polynomials in many references). For creating a generating function, the parameter indexing the partition $\lambda$ is the monomial $t^{\lambda}$ in another set of variables $\{t_{i}\}_{i=1}^{\infty}$, and if we restrict to permutations in $S_{n}$ and partitions of length at most $n$, then the multiplier associated with some word $w$ is the element $\varepsilon_{ww_{0}}$ in an appropriately modified Demazure algebra (with $w_{0} \in S_{n}$ the word of maximal length). Then this generating function satisfies a functional equation similar to that from \cite{[LS]}, of the Schubert polynomials, of the one of the $\beta$-Grothendieck polynomials as showing up in the appending of \cite{[H]} (see Proposition \ref{propgen} below).

Due to the double parametrization, by fixing $w$ we still get a generating function, as a sum over the partitions $\lambda$. Our object of investigation is these generating functions, denoted by $\mathcal{K}_{w}(x,t)$ and defined to be $\sum_{\lambda}K_{\lambda,w}(x)t^{\lambda}$, where for a partition $\lambda$ (and the permutation $w$), $K_{\lambda,w}(x)$ is the key polynomial associated with $w$ and $\lambda$. It is a rational function (or a generalized one, if we remove the restriction on the length of $\lambda$), which can be constructed, like the key polynomials themselves, from the basic case $w=\operatorname{Id}$ using the corresponding modified divided difference operators $\{\pi_{i}\}_{i\geq1}$. We prove that both the numerator and the denominator in the expression for $\mathcal{K}_{w}(x,t)$ carry interesting combinatorial meanings.

The numerator can be described explicitly quite easily as follows. For any integer $l$, there is a finite set $A_{l}(w)$ of increasing sequences of integers that are bounded by the first $l$ values of $w$ (ordered as well), as we prove in Theorem \ref{formofKw}. This follows quite easily by induction, after one determines the behavior of the sets $A_{l}(w)$ once one passes from $w$ to $s_{i}w$ with larger length. The more involved analysis is the one concerning the numerator, which is a polynomial $P_{w}(x,t)$. This polynomial is 1 up to quadratic terms (see Proposition \ref{expPwinT}), and the quadratic terms themselves are based on multi-sets having multiple presentations as sums from the $A_{l}(w)$'s (see Definition \ref{Bklwdef} and Theorem \ref{mposBkl}), the sets of which, denoted by $B_{k,l}(w)$, exhibit interesting combinatorial behavior.

In some cases we obtain an explicit formula for the multiplicity in which the term associated with $\eta \in B_{k,l}(w)$ shows up in the quadratic part of $P_{w}$, which is based on the poset type of the set of presentations of $\eta$ as such a sum (see Theorems \ref{diff1} and \ref{diff2} and Proposition \ref{lketa23} below). We conjecture that all the quadratic part can be expressed in such a way (see Conjecture \ref{poset} below). The cubic terms also seem to exhibit a combinatorial behavior of a similar flavor (as we suggest in Conjecture \ref{formPw3}). Comparing our formula for $\mathcal{K}_{w}(x,t)$ with the series obtained by replacing the denominator by 1 expresses the coefficient of a monomial in a key polynomial as a combination of numbers of integral points on polytopes, as explained in Proposition \ref{Fwcoeffs}. Generalizing to the Lascoux polynomials, we get a similar generating function, with the same denominator, and we prove in Proposition \ref{Pxiw1} a first result about the linear part of the resulting polynomial in this case.

\smallskip

The paper is divided into 6 sections. Section \ref{PropAlw} introduces the basic objects involving permutations and the sets $A_{l}(w)$ mentioned above, and establishes their required properties. Section \ref{PropBlw} considers sets of multi-sets that can be obtained as sums of two bounded ascending sequences, and proves results about them that are needed later. Section \ref{KeyPols} defines the key polynomials, and shows that their generating series has the desired form. Section \ref{Quadratic} focuses on the quadratic parts of the polynomials $P_{w}$, and proves that they are based on the sets $B_{k,l}(w)$, plus some auxiliary results. In Section \ref{Mults} we determine the multiplicities showing up in these quadratic terms in several cases. Finally, Section \ref{QuesRes} suggests several directions to pursue the investigation further, including a few conjectures and the results involving integral points on polytopes and the basic form arising from the Lascoux polynomials.

\section{Permutations, Lengths, and Bounded Ascending Sequences \label{PropAlw}}

Let $S_{n}$ be the symmetric group on the first $n$ positive integers. By viewing $S_{n}$ as the stabilizer in $S_{n+1}$ of the largest integer and taking the direct limit over $n$, we obtain the group $S_{\infty}$, and all the notions that involve a permutation $w$ that will be considered in this paper will be the same for $w$ as an element of the group $S_{n}$, of the next one $S_{n+1}$, or of the infinite limit group $S_{\infty}$. In our notation the set $\mathbb{N}$ of natural numbers does not contain 0.

Given a finite sequence $\alpha$ of distinct integers, we will denote by $\operatorname{ord}(\alpha)$ the tuple having the same elements as $\alpha$ but ordered increasingly. Moreover, if $\alpha$ is already ordered in such a way, and $\gamma$ is another increasing sequence of the same length as $\alpha$ (say $l$), we denote by $\alpha\leq\gamma$ the statement that $\alpha_{j}\leq\gamma_{j}$ for every $1 \leq j \leq l$. Using these notions we now make the following definition.
\begin{defn}
For $w$ in $S_{n}$ and an integer $l$, we denote by $w^{l}$ the ordering of the first $l$ values of $w$, namely $w^{l}:=\operatorname{ord}\big(w(1),\ldots,w(l)\big)$, and write $\{w^{l}_{j}\}_{j=1}^{l}$ for its elements. We also define \[A_{l}(w):=\{\alpha=(\alpha_{1},\ldots,\alpha_{l})\in\mathbb{N}^{l}|\alpha_{j}<\alpha_{j+1}\ \forall1 \leq j<l,\ \alpha \leq w^{l}\},\] namely $A_{l}(w)$ is the set of increasing sequences $(\alpha_{1},\ldots,\alpha_{l})$ of positive integers such that $\alpha_{j} \leq w^{l}_{j}$ for $1 \leq j \leq l$ (this is the set of ascending sequences that are bounded by $w^{l}$). We we also use the union $A(w):=\bigcup_{l=1}^{n}A_{l}(w)$. \label{Asetdef}
\end{defn}

We will usually write permutations in one-line notation. Hence the element $w \in S_{5}$ which takes 1 to 4, 2 to 2, 3 to 5, 4 to 3, and 5 to 1 will be written simply as 42531 (or 425316 when viewed as an element of $S_{6}$, etc.). We will also write ordered sets as concatenated sequence of digits, when no confusion may arise. Hence for our $w$ we have $w^{1}=4$, $w^{2}=24$, $w^{3}=245$, $w^{4}=2345$, and $w^{5}=12345$. In particular the set $A_{2}(w)$ is $\{12,13,14,23,24\}$, and we have
\begin{equation}
A_{3}(w)=\{123,124,125,134,135,145,234,235,245\}. \label{A3wex}
\end{equation}

\begin{rmk}
Note that $w^{l}$ and $A_{l}(w)$ from Definition \ref{Asetdef} depend only on the first $l$ values of $w$, so that the two sets in our example remain the same when $w$ is replaced by 42513, or 425361. It also follows from the definition that for $w \in S_{n}$ we have $w^{n}=(1,\ldots,n)$ and $A_{n}(w)$ consists of the single element $w^{n}$ (note that $w^{l} \in A_{l}(w)$ for every $w$ and $l$, by definition). In particular, for $w=\operatorname{Id}$ we have $A_{l}(\operatorname{Id})=\{(1,\ldots,l)\}$ for every $l$, and by considering any $w$ as an element of $S_{\infty}$, the union $A(w)$, which is now taken over all $l\geq1$, is now an infinite union, with $A_{l}(w)$ being the singleton $A_{l}(\operatorname{Id})$ for every large enough $l$ (and clearly finite for every $l\geq1$). \label{forId}
\end{rmk}

\smallskip

For $1 \leq i<n$, let $s_{i}$ be the simple transposition, interchanging $i$ and $i+1$ and leaving the remaining numbers invariant. Recall that the symmetric group $S_{n}$ is generated by $\{s_{i}\}_{i=1}^{n-1}$, and for a permutation $w \in S_{n}$ we define, as usual, its \emph{length} $\ell(w)$, to be the minimum number of simple transpositions in a product that yields $w$. We will often investigate our objects by induction, through constructing $w$ as a product of $\ell(w)$ simple transpositions. We now give some properties of $\ell(w)$ that follow directly from its definition, together with the fact that $s_{i}$ is an odd permutation and an involution for every $i$. The proofs (which are mostly simple and direct), as well as more details on $\ell(w)$ and some related notions, can be found in, e.g., Chapter 1 of \cite{[P]}.
\begin{lem}
The number $\ell(w)$ has the same parity as that $w$ as a permutation, and it equals the number of inversions $\{i<j|w(i)>w(j)\}$ of $w$. It also satisfies the triangle inequality $l(wu) \leq l(w)+l(u)$. \label{propoflw}
\end{lem}

Since concatenation of permutations corresponds to their compositions as functions acting on arguments showing up to the left of them, we obtain the following consequence.
\begin{lem}
For $w \in S_{n}$ and $1 \leq i<n$, the following are equivalent:
\begin{enumerate}[$(i)$]
\item We have the inequality $\ell(s_{i}w)>\ell(w)$.
\item The equality $\ell(s_{i}w)=\ell(w)+1$ holds.
\item The entry $i$ shows up before $i+1$ in the one-line notation for $w$.
\item The inequality $w^{-1}(i)<w^{-1}(i+1)$ holds.
\end{enumerate} \label{lsiwlw}
\end{lem}

Lemma \ref{lsiwlw} follows directly from Lemma \ref{propoflw}, with the triangle inequality required for showing that $(i)$ implies $(ii)$ there. The opposite situation for $w$ and $i$ is where $\ell(s_{i}w)<\ell(w)$, $\ell(s_{i}w)=\ell(w)-1$, $i$ appears after $i+1$ in the word $w$, and $w^{-1}(i)>w^{-1}(i+1)$.

Given any ascending sequence $\alpha$ of positive integers, we allow ourselves the abuse of notation of writing $s_{i}(\alpha)$ for $\operatorname{ord}\big(s_{i}(\alpha)\big)$, where the argument is the sequence obtained by applying $s_{i}$ to all the entries of $\alpha$. This will only matter when both $i$ and $i+1$ show up in $\alpha$, since in any other case the sequence $s_{i}(\alpha)$ (in the straightforward meaning) is already ordered increasingly. It is also clear from the definitions that $(s_{i}w)^{l}=s_{i}(w^{l})$ in these notations. In fact, one can define this for any permutation $u$ acting on $\alpha$, but we will only use it for $u=s_{i}$.

For seeing how the sets from Definition \ref{Asetdef} are affected when we move from $w$ to $s_{i}w$ in case the equivalent conditions from Lemma \ref{lsiwlw} are satisfied, we write $A_{l}(w)$ for every $l$ as the disjoint union $A_{l}^{(i)}(w) \cup A_{l,i}(w)$, where
\begin{equation}
A_{l}^{(i)}(w):=\{\alpha \in A_{l}(w)|s_{i}\alpha \in A_{l}(w)\}\mathrm{\ and\ }A_{l,i}(w):=\{\alpha \in A_{l}(w)|s_{i}\alpha \notin A_{l}(w)\}. \label{Alwdecom}
\end{equation}
For example, if we take $w=42531$ and $l=3$, then the set $A_{3}(w)$ is given in Equation \eqref{A3wex}, and with $i=2$ the subset $A_{3,2}(w)$ from Equation \eqref{Alwdecom} is just the singleton $\{245\}$ (as its $s_{2}$-image 345 is not in $A_{3}(w)$), while $A_{3}^{(2)}(w)$ is the complement $\{123,124,125,134,135,145,234,235\}$ (and indeed, $s_{2}$ leaves 123, 145, 234, and 235 invariant, while it interchanges 124 with 134 and 125 with 135, showing the $s_{2}$-invariance of this set).

The second set in Equation \eqref{Alwdecom}, which will soon be seen to be the more delicate one, is defined by the following properties.
\begin{lem}
The set $A_{l,i}(w)$ is non-empty if and only if $\ell(s_{i}w)>\ell(w)$ and $w^{-1}(i) \leq l<w^{-1}(i+1)$. For such a value of $l$ there exists some $1 \leq j \leq l$ with $w^{l}_{j}=i$, and then $A_{l,i}(w)$ is the set of $\alpha \in A_{l}(w)$ with $\alpha_{j}=i$ and in which $i+1$ does not show up. \label{nontrivAli}
\end{lem}

\begin{proof}
If an element $\alpha$ contains both $i$ and $i+1$, or contains neither, then $s_{i}(\alpha)=\alpha$ in our definitions and hence if $\alpha \in A_{l}(w)$ then $\alpha \in A_{l}^{(i)}(w)$. Similarly, if $\alpha$ contains $i+1$ but not $i$, then in $s_{i}(\alpha)$ we replace $i+1$ by $i$, and if $\alpha$ satisfies the bounds required for it to be in $A_{l}(w)$ then clearly so does $s_{i}(\alpha)$ and $\alpha$ is in $A_{l}^{(i)}(w)$ once again. It thus remains to consider elements $\alpha \in A_{l}(w)$ such that $\alpha_{j}=i$ for some index $j$ and $i+1$ does not show up in $\alpha$, where $s_{i}(\alpha)_{j}=i+1$ and the other entries of $s_{i}(\alpha)$ are like those of $\alpha$.

Now, we have $w^{l}_{j} \geq i$ since $\alpha \in A_{l}(w)$. If $w^{l}_{j}>i$ then $s_{i}(\alpha)$ also lies in $A_{l}(w)$ (since all the bounds are satisfied), and again $\alpha \in A_{l}^{(i)}(w)$. We thus consider the case where $w^{l}_{j}=i$, so in particular $w^{-1}(i) \leq l$ since $i$ shows up in $w^{l}$ and is thus the value $w(h)$ for some $h \leq l$, which is $w^{-1}(i)$ by definition. We also proves that that $A_{l,i}(w)$ is contained in the asserted set of $\alpha$'s, and we have to show the reverse inclusion and that this set is non-empty precisely when $\ell(s_{i}w)>\ell(w)$ and $l<w^{-1}(i+1)$.

But if $w^{-1}(i+1) \leq l$ as well, then $j<l$ (since $w^{l}_{l}$ is the maximal element of $w^{l}$, and it cannot be $i$ if $i+1$ is also there), and the ordering property of $w^{l}$ implies that $w^{l}_{j+1}=i+1$. Thus if $\alpha \in A_{l}(w)$ and $\alpha_{j}=i$, then we have the inequalities $i<\alpha_{j+1} \leq w^{l}_{j+1}=i+1$, meaning that $i+1$ also shows up in $\alpha$, which contradicts our assumption. Therefore in this case the asserted set of $\alpha$'s is empty, and with it also $A_{l,i}(w)=\emptyset$. Hence the only situation where $A_{l,i}(w)$ can be non-empty is if $w^{-1}(i) \leq l$ (as already noticed) but also $w^{-1}(i+1)>l$, meaning, through Lemma \ref{lsiwlw}, that $\ell(s_{i}w)>\ell(w)$.

Finally, assume that $\ell(s_{i}w)>\ell(w)$ and $l$ is in the desired range, so that $w^{l}_{j}=i$ and $i+1$ does not appear in $w^{l}$ (hence either $j=l$ or $w^{l}_{j+1}>i+1$). We consider a sequence $\alpha \in A_{l}(w)$ with $\alpha_{j}=i$, and such that if $j<l$ then $\alpha_{j+1}>i+1$. For every such sequence, $s_{i}(\alpha)$ is obtained by replacing the $j$th entry by $i+1$, which no longer satisfies the $w^{l}$ bounding condition, and is therefore not in $A_{l}(w)$. Therefore $A_{l,i}(w)$ is precisely the asserted set, and to see that it is not empty in this setting, note that $w^{l}$ itself, which is always in $A_{l}(w)$ (see Remark \ref{forId}), satisfies the conditions that we found for it to be in $A_{l,i}(w)$. This completes the proof of the lemma.
\end{proof}

We can now compare $A(s_{i}w)$ with $A(w)$ when the former permutation has greater length.
\begin{prop}
If $\ell(s_{i}w)>\ell(w)$ then for every $l$ the set $A_{l}(s_{i}w)$ is the disjoint union of $A_{l}(w)$ and $\{s_{i}(\alpha)|\alpha \in A_{l,i}(w)\}$, with the latter being non-empty for at least one value of $l$. In particular we have $A_{l}(w) \subseteq A_{l}(s_{i}w)$ for every $l$, $A_{l}(w)$ is preserved under $s_{i}$ for every $l$, and $A(w) \subsetneq A(s_{i}w)$. \label{Asiw}
\end{prop}

\begin{proof}
Lemma \ref{lsiwlw} implies that $w^{-1}(i)<w^{-1}(i+1)$. Now, when $l<w^{-1}(i)$ or when $l \geq w^{-1}(i+1)$ the sequence $w^{l}$ is invariant under applying $s_{i}$ in our terminology (indeed, in the former case neither $i$ nor $i+1$ show up in $w^{l}$, and in the latter case it contains both of them), and as applying $s_{i}$ yields $(s_{i}w)^{l}$, we deduce that in this case the latter coincides with $w^{l}$ and hence $A_{l}(s_{i}w)=A_{l}(w)$  for these values of $l$. Moreover, Lemma \ref{nontrivAli} shows that for such values of $l$ the set $A_{l,i}(w)$ is empty, so our presentation as a union trivially holds.

So assume now that $w^{-1}(i) \leq l<w^{-1}(i+1)$, so that $w^{l}_{j}=i$ for some $1 \leq j \leq l$, and if $j<l$ then $w^{l}_{j+1}>i+1$. Then $(s_{i}w)^{l}$ is obtained by replacing the $j$th entry by $i+1$, so in particular $w^{l}_{h}\leq(s_{i}w)^{l}_{h}$ for every $1 \leq h \leq l$. It follows that every element of $A_{l}(w)$ satisfies the (weaker) bounds of $(s_{i}w)^{l}$ as well, so that $A_{l}(w) \subseteq A_{l}(s_{i}w)$. Moreover, Lemma \ref{nontrivAli} shows that if $\alpha \in A_{l,i}(w)$ then $\alpha_{j}=i$ and if $j<l$ then $i+1<\alpha_{j+1} \leq w^{l}_{j+1}$, and as applying $s_{i}$ to $\alpha$ replaces $\alpha_{j}=i$ by $i+1$ and leaves the remaining entries of $\alpha$ invariant, we deduce that $s_{i}(\alpha)$ satisfies the bounds arising from $(s_{i}w)^{l}$. Hence $\{s_{i}(\alpha)|\alpha \in A_{l,i}(w)\}$, which is disjoint from $A_{l}(w)$ by the definition in Equation \eqref{Alwdecom}, is also contained in $A_{l}(s_{i}w)$. Note that Lemma \ref{nontrivAli} shows that this set is not empty for such value of $l$, and as there always is such a value (e.g., $l=w^{-1}(i)<w^{-1}(i+1)$), the final strict containment is also established.

It only remains to show that if $w^{-1}(i) \leq l<w^{-1}(i+1)$ and $\alpha \in A_{l}(s_{i}w)$ then it lies in the asserted union. Let again $1 \leq j \leq l$ be such that $w^{l}_{j}=i$ and hence $(s_{i}w)^{l}_{j}=i+1$, and then if $\alpha_{j} \leq i$ then we deduce, as in the previous paragraph, that $\alpha \in A_{l}(w)$. Otherwise $\alpha_{j}=i+1$, which means that $\alpha \not\in A_{l}(w)$, and since $\alpha_{h}>i+1$ for $h>j$ and in case $h<j$ (so that $j>1$) we have \[\alpha_{h}\leq\alpha_{j-1}\leq(s_{i}w)^{l}_{j-1}=w^{l}_{j-1}<w^{l}_{j}=i,\] we deduce that $i$ does not appear in $\alpha$. It follows that $s_{i}(\alpha)$ is obtained by replacing $\alpha_{j}=i+1$ by $i$, and our argument shows that $s_{i}(\alpha) \in A_{l}(w)$. Since $s_{i}\big(s_{i}(\alpha)\big)=\alpha \not\in A_{l}(w)$, we deduce that $s_{i}(\alpha) \in A_{l,i}(w)$, so that $\alpha$ is indeed the $s_{i}$-image of an element of $A_{l,i}(w)$, as desired. The fact that elements of $A_{l}^{(i)}$ come with their $s_{i}$-images by definition in Equation \eqref{Alwdecom}, and $A_{l,i}(w)$ and the complement $A_{l}(s_{i}w) \setminus A_{l}(w)$ are taken to one another by $s_{i}$, yields the asserted invariance as well. This proves the proposition.
\end{proof}
Proposition \ref{Asiw} can also be used to prove by induction the finiteness of $A_{l}(w)$ for all $l$ and $w$, as noted in Remark \ref{forId}.

\smallskip

Recall that two increasing sequences $\alpha$ and $\gamma$ of the same length $l$ satisfy $\alpha\leq\gamma$ in case $\alpha_{j}\leq\gamma_{j}$ for every $1 \leq j \leq l$, and we will write $\alpha<\gamma$ to denote $\alpha\leq\gamma$ and $\alpha\neq\gamma$ as usual (note that this does not imply $\alpha_{j}<\gamma_{j}$ for all $j$). This produces a partial order on the set $A_{l}(w)$ for any given $w$, in which there is always a unique maximal element, which is $w^{l}$ by definition (there is also the minimal element coming up from $A_{l}(\operatorname{Id})$, but we will not need it). For this we shall make the following definition.
\begin{defn}
Let $\alpha$ be an ascending sequence of $l$ positive integers, and fix an index $1 \leq j \leq l$ and some number $r$ that is not in $\alpha$. Then we define $R_{j,r}(\alpha)$ to be $\operatorname{ord}(\alpha_{j \leftarrow r})$, where $\alpha_{j \leftarrow r}$ is the sequence obtained by replacing $\alpha_{j}$ by $r$. \label{replace}
\end{defn}
The properties of the operation from Definition \ref{replace} are as follows.
\begin{lem}
We have $R_{j,r}(\alpha)<\alpha$ in case $r<\alpha_{j}$ and $R_{j,r}(\alpha)>\alpha$ when $r>\alpha_{j}$. Moreover, if $\alpha \in A_{l}(w)$ and $r \leq w^{l}_{j}$ for some $w$, then $R_{j,r}(\alpha)$ is also in $A_{l}(w)$. \label{Rjrprop}
\end{lem}

\begin{proof}
Since $r$ is not in $\alpha$, all the inequalities involving $r$ and elements of $\alpha$ are strict. We also note that if $\alpha_{j-1}<r<\alpha_{j+1}$ then $\alpha_{j \leftarrow r}$ is already ordered, the comparison with $\alpha$ is clear, and so is the fact that the assumptions for the second assertion imply that $\alpha_{j \leftarrow r} \in A_{l}(w)$. Assume next that $r>\alpha_{j+1}$, so that putting $\alpha_{j+1}$ in the $j$th position and $r$ in the entry number $j+1$ yields a sequence that is larger than or equal to $\alpha$ at every spot. Moreover, under the $w^{l}$-assumptions we have $\alpha_{j+1}<r \leq w^{l}_{j}$ and $r \leq w^{l}_{j}<w^{l}_{j+1}$, and this sequence is bounded by $w^{l}$ at every entry as well. We are now in the same situation as before, with the index $j$ replaced by $j+1$. Repeating the process until we reach the situation where either $r$ is in the $i$th position and $r<\alpha_{i+1}$ or $r$ lands in the $l$th position (and one of these will eventually occur, as the index $j$ only grows and $l$ is finite), we reach a situation where the new sequence is ordered and thus equals $R_{j,r}(\alpha)$, and the same argument proves both assertions in this case.

The remaining option with the initial index $j$ is that $r<\alpha_{j-1}$. We now put $\alpha_{j-1}$ in the $j$th position and $r$ in the one number $j-1$, yielding a sequence that is bounded by $\alpha$, and the $w^{l}$-conditions yield the inequalities $\alpha_{j-1}<\alpha_{j} \leq w^{l}_{j}$ and $r<\alpha_{j-1} \leq w^{l}_{j-1}$. This puts us in the same situation but with $j-1$ instead of $j$. Iterating the argument until either we have $r>\alpha_{i-1}$ when $r$ is in the $i$th position or $r$ reaches the first entry (one of which will again happen at some stage, by finiteness) produces the ordered sequence $R_{j,r}(\alpha)$ which indeed satisfies the bound of $\alpha$, as well as that of $w^{l}$ under the additional assumptions. This proves the lemma.
\end{proof}

\section{Sums of Two Sequences, and Multi-Sets \label{PropBlw}}

A set $\alpha$ in some $A_{l}(w)$ will correspond, in the polynomial ring in variables $\{x_{i}\}_{i=1}^{n}$ for large enough $n$ (or in infinitely many variables), to the monomial $x^{\alpha}:=\prod_{j=1}^{l}x_{\alpha_{j}}$. The next step involves considering products of two such monomials, which would produce $x^{\eta}$, where $\eta$ is now a multi-set of positive integers, in which some elements are allowed to show up with multiplicity, but the multiplicity is at most 2. We will write the union of multi-sets (increasing multiplicities) as addition (motivated by the fact that $x^{\alpha}x^{\beta}=x^{\alpha+\beta}$), meaning that we can write every multi-set with the above property as $\eta=\eta_{1}+2\eta_{2}$ for disjoint sets $\eta_{1}$ and $\eta_{2}$ of positive integers, representing the elements showing up in $\eta$ once and twice respectively. For ordinary sets of integers we will allow ourselves the abuse of notation of writing $\alpha$ both for a set of integers, as well as for the increasing sequence that was denoted by $\operatorname{ord}(\alpha)$ above.

So let now $l$ and $k$ be two integers, and we suppose, unless otherwise specified, that $k \leq l$. We then define the following sets.
\begin{defn}
Given $w \in S_{n}$ and such $k$ and $l$, set \[\tilde{B}_{k,l}(w):=\{\eta|\exists\alpha \in A_{l}(w)\mathrm{\ and\ }\beta \in A_{k}(w)\mathrm{\ such\ that\ }\eta=\alpha+\beta\},\] namely $\tilde{B}_{k,l}(w)$ is the set of multi-sets which can be presented, at least in one way, as the sum of an element of $A_{l}(w)$ and one from $A_{k}(w)$. We will be more interested, however, in the subset $B_{k,l}(w)$ consisting of those multi-sets which are thus obtained in more than one way in a non-trivial manner, namely $\eta$ is in $B_{k,l}$ if there are $\alpha$ and $\tilde{\alpha}$ in $A_{l}(w)$ as well as $\beta$ and $\tilde{\beta}$ in $A_{k}(w)$ such that \[\eta=\alpha+\beta=\tilde{\alpha}+\tilde{\beta}\quad\alpha\neq\tilde{\alpha},\alpha\neq\tilde{\beta},\beta\neq\tilde{\alpha},\beta\neq\tilde{\beta}.\] \label{Bklwdef}
\end{defn}
Note that if $k \neq l$ (i.e., $k<l$) then the conditions $\alpha\neq\tilde{\beta}$ and $\beta\neq\tilde{\alpha}$ in Definition \ref{Bklwdef} are immediate, but they do become meaningful when $l=k$. We will be more interested in $B_{k,l}(w)$, though it will be easier to work with $\tilde{B}_{k,l}(w)$. It is obvious from the fact that elements of $A_{l}(w)$ and $A_{k}(w)$ are ordinary sets (rather than multi-sets) of size $l$ and $k$ respectively that every element $\eta \in B_{k,l}(w)$ for some $w$ satisfies $|\eta_{2}| \leq k \leq l$ and $2|\eta_{2}|+|\eta_{1}|=k+l$, and every $\eta$ with these properties is in $B_{k,l}(w)$ for some $w \in S_{\infty}$.

Using the values of $A_{2}(w)$ and $A_{3}(w)$ for $w=42531$ in Equation \eqref{A3wex} above, we take $l=3$ and $k=2$, and evaluate $B_{2,3}(w)$ and $\tilde{B}_{2,3}(w)$. For example, the equalities $12+134=14+123=11234$ and $235+14=135+24=12345$ show that 11234 and 12345 are in $B_{2,3}(w)$, but 11334 lies in $\tilde{B}_{2,3}(w)$ as the sum 13+134, but is not in $B_{2,3}(w)$ since this is the only possible way to construct it as a sum of an element of $A_{2}(w)$ and one from $A_{3}(w)$. In total we obtain that
\begin{align*}
B_{2,3}(w)=&\{11234,11235,11245,11345,12234,12235,12245, \\ &~12334,12335,12344,12345,12445,22345\}
\end{align*}
and $\tilde{B}_{2,3}(w)$ is the union of $B_{2,3}(w)$ and
\begin{align}
&\{11223,11224,11225,11233,11244,11334,11335,11344, \nonumber \\ &~ 11445,12233,12244,22334,22335,22344,22445\}. \label{B23wex}
\end{align}
It is clearly much easier to determine when some multi-set lies in $\tilde{B}_{k,l}(w)$ (by finding a sum yielding it) than to verify whether it is in $B_{k,l}(w)$. Proposition \ref{Bkldiff} below gives a an alternative criterion for checking which elements of $\tilde{B}_{k,l}(w)$ are in $B_{k,l}(w)$ that is much easier to check.

\smallskip

Using Lemma \ref{Rjrprop} we will obtain the following argument, which we will use several times below. Recall that if $k \leq l$ then the set whose ordering produces $w^{k}$ is contained in that which yields $w^{l}$, so that for every $1 \leq j \leq k$ there exists an index $h$ such that $w^{k}_{j}=w^{l}_{h}$. This can be written in terms of an order-preserving injective map $\varphi_{k,l}$ such that $w^{k}_{j}=w^{l}_{\varphi_{k,l}(j)}$ for any $1 \leq j \leq k$ (this map is clearly the identity map in case $k=l$).
\begin{lem}
Given $k \leq l$, consider elements $\alpha \in A_{l}(w)$ and $\beta \in A_{k}(w)$, take some index $1 \leq j \leq k$, and assume that $\beta_{j}<\alpha_{\varphi_{k,l}(j)}$. Assume further that either $\beta_{j}$ is not in $\alpha$, or that $k=l$. Then there exist elements $\tilde{\alpha} \in A_{l}(w)$ and $\tilde{\beta} \in A_{k}(w)$ with $\tilde{\beta}>\beta$ and $\tilde{\alpha}+\tilde{\beta}=\alpha+\beta$. \label{interup}
\end{lem}

\begin{proof}
First we construct an element $\hat{\beta} \in A_{k}(w)$ with $\hat{\beta}>\beta$. If $\alpha_{\varphi_{k,l}(j)}$ is not in $\beta$, then $r=\alpha_{\varphi_{k,l}(j)}$ satisfies $\beta_{j}<r \leq w^{l}_{\varphi_{k,l}(j)}=w^{k}_{j}$, and Lemma \ref{Rjrprop} shows that $\hat{\beta}=R_{j,r}(\beta)$ is in $A_{k}(w)$ and satisfies $\hat{\beta}>\beta$. Otherwise set $j_{1}:=j$, and then $\alpha_{\varphi_{k,l}(j_{1})}$ equals some $\beta_{j_{2}}$, with $j_{2}>j_{1}$ since $r_{1}:=\alpha_{\varphi_{k,l}(j_{1})}>\beta_{j_{1}}$. But as $k \leq l$ we can apply $\varphi_{k,l}$, and as $\varphi_{k,l}(j_{2})>\varphi_{k,l}(j_{1})$ by order-preservation, we get $r_{2}:=\alpha_{\varphi_{k,l}(j_{2})}>\alpha_{\varphi_{k,l}(j_{1})}=\beta_{j_{2}}$, with $r_{1}=\beta_{j_{2}}$. But this is now the same situation, with the index $j_{1}$ replaced by the larger index $j_{2}$. We continue and construct an increasing sequence $j_{1}<j_{2}<\ldots$, where each time that the resulting element, say $r_{h-1}:=\alpha_{\varphi_{k,l}(j_{h-1})}$ still lies in $\beta$ we get a new index $j_{h}$ such that $r_{h-1}=\beta_{j_{h}}$, with the corresponding $r_{h}$.

Now, since $k$ is finite, this increasing sequence of indices must be finite, and let $h$ be the maximal one. We can now apply $R_{j_{h},r_{h}}$ to $\beta$, and note that the resulting element of $A_{k}(w)$, which is larger than $\beta$, no longer contains $\beta_{j_{h}}=r_{h-1}$. We can thus apply $R_{j_{h-1},r_{h-1}}$ to it, and go down with the index until the last application of $R_{j_{1},r_{1}}$, which will produce our $\hat{\beta}$. In fact, this $\hat{\beta}$ is just $R_{j,r_{h}}(\beta)$, but we needed this decomposition in order to determine $r_{h}$ as some element of $\alpha$, and obtain that $\hat{\beta}$ is in $A_{k}(w)$.

In the first case, where $\beta_{j}$ is not in $\alpha$, we set $\tilde{\beta}:=\hat{\beta}$ (which was seen to have the desired properties), and as $\beta_{j}<\alpha_{\varphi_{k,l}(j)}<\alpha_{\varphi_{k,l}(j_{h})}=r_{h}$, we can define $\tilde{\alpha}$ to be $R_{\varphi_{k,l}(j_{h}),\beta_{j}}(\alpha)$, which is in $A_{l}(w)$ by Lemma \ref{Rjrprop}. Since as sets we have $\tilde{\beta}=\beta\cup\{r_{h}\}\setminus\{\beta_{j}\}$ and $\tilde{\alpha}=\alpha\cup\{\beta_{j}\}\setminus\{r_{h}\}$, the assertion follows in this case.

We thus assume $k=l$, so that $\varphi_{k,l}$ is the identity map, and we know that $\beta_{j}<\alpha_{j}$, where the case that $\beta_{j}$ is not in $\alpha$ was already covered. By a similar argument we set $j^{(0)}:=j$ and find $j^{(1)}<j^{(0)}$ such that $r^{(1)}:=\beta_{j^{(0)}}=\alpha_{j^{(1)}}$ and this value is larger than $\beta_{j^{(1)}}$. We proceed in this manner until we get an index $p$ with $r^{(p)}:=\beta_{j^{(p-1)}}$ not lying in $\alpha$, with a decreasing sequence of indices $j=j^{(0)}>j^{(1)}>\ldots>j^{(p)}$ for some $p\geq1$. We let $R_{j^{(1)},r^{(1)}}$ act on $\hat{\beta}$, and then $R_{j^{(2)},r^{(2)}}$ operates on the result, with the application of $R_{j^{(p)},r^{(p)}}$ yielding our $\tilde{\beta}$, which can now be written as $R_{j^{(p)},r_{h}}(\beta)$ (but again with the decomposition yielding its required properties) hence as a set it equals $\beta\cup\{r_{h}\}\setminus\{r^{(p)}\}$. By setting $\tilde{\alpha}:=R_{j_{h},r^{(p)}}(\alpha) \in A_{l}(w)$, which as a set is $\beta\cup\{r^{(p)}\}\setminus\{r_{h}\}$, we obtain the result in this case as well. This completes the proof of the lemma.
\end{proof}
The element $\tilde{\beta}$ from Lemma \ref{interup} was seen to be of the form $R_{i,r}(\beta)>\beta$ for a single operator as given in Definition \ref{replace}, but $i$ and $r$ there do not satisfy the conditions of Lemma \ref{Rjrprop}. This is why we wrote this $R_{i,r}$ as a composition of several such operators, since now Lemma \ref{Rjrprop} applies to each of them and we could make sure that $\tilde{\beta} \in A_{k}(w)$. The other element $\tilde{\alpha}$ was obtained in a similar manner, but as $\tilde{\alpha}<\alpha$, the fact that it lies in $A_{l}(w)$ is immediate. By inverting all the inequalities in the proof, and using the same decomposition for $\alpha$ rather than $\beta$, we obtain the following complement, in the opposite direction.
\begin{lem}
Assume that $k<l$, $\alpha \in A_{l}(w)$, $\beta \in A_{k}(w)$, $\beta_{j}>\alpha_{\varphi_{k,l}(j)}$, and $\beta_{j}$ is not in $\alpha$. Then there are elements $\alpha<\tilde{\alpha} \in A_{l}(w)$ and $\tilde{\beta} \in A_{k}(w)$ for which $\tilde{\alpha}+\tilde{\beta}=\alpha+\beta$. \label{interdown}
\end{lem}
The opposite direction of the case $k=l$ in Lemma \ref{interup} is covered by the assertion itself, and need not be restated also in Lemma \ref{interdown}

\smallskip

Now consider $\eta=\eta_{1}+2\eta_{2}$ as above, and we will say that $\eta$ contains a set $\beta$ if $\beta\subseteq\eta_{1}\cup\eta_{2}$ as sets. We will repeatedly use the following simple observations.
\begin{lem}
If $\alpha$ and $\beta$ are increasing sequences that satisfy $\alpha+\beta=\eta$ then both $\alpha$ and $\beta$ contain $\eta_{2}$ and are contained in $\eta$ in the above sense. Moreover, let $\alpha$ and $\beta$ contain $\eta_{2}$ and be contained in $\eta$, and take another pair of sequences $\tilde{\alpha}$ and $\tilde{\beta}$ such that $\tilde{\alpha}+\tilde{\beta}=\alpha+\beta$. Then $\tilde{\alpha}$ and $\tilde{\beta}$ also contain $\eta_{2}$ and are contained in $\eta$. \label{conteta}
\end{lem}

\begin{proof}
The elements of $\eta_{2}$ show up with multiplicity 2 in the sum, meaning that they must be contained in both $\beta$ and $\alpha$. As every element showing up in $\alpha$ or $\beta$ must show up in their sum $\eta$, we deduce the other containment in the first assertion as well.

Take now $\alpha$ and $\beta$ satisfying the containment properties, set $\tilde{\eta}:=\alpha+\beta$, and express it as $\tilde{\eta}_{1}+2\tilde{\eta}_{2}$ as for $\eta$. Since $\eta_{2}$ is contained in both $\beta$ and $\alpha$, we deduce that $\eta_{2}\subseteq\tilde{\eta}_{2}$, while the fact that $\beta$ and $\alpha$ are contained in $\eta$ implies $\tilde{\eta}_{1}\subseteq\eta$ as well (and in fact $\tilde{\eta}_{1}\subseteq\eta_{1}$ by the previous containment and the fact that $\eta_{1}\cap\eta_{2}=\tilde{\eta}_{1}\cap\tilde{\eta}_{2}=\emptyset$). We thus apply the first assertion, with the sum $\tilde{\eta}$, and deduce that $\eta_{2}\subseteq\tilde{\eta}_{2}\subseteq\tilde{\beta}\subseteq\tilde{\eta}\subseteq\eta$ (and the same for $\tilde{\alpha}$). This completes the proof of the lemma.
\end{proof}

Based on Lemma \ref{conteta}, we make the following definition/notation.
\begin{defn}
For a multi-set $\eta=\eta_{1}+2\eta_{2}$, a permutation $w$, and an index $l$, we denote by $A_{l}^{\eta}(w)$ the set $\{\alpha \in A_{l}(w)|\eta_{2}\subseteq\alpha\subseteq\eta\}$. \label{setswitheta}
\end{defn}

The first property of the set from Definition \ref{setswitheta} is given in the following Proposition, in which we replaced in the index $l$ by $k$ in order to match the notation of Lemma \ref{interup} when we apply it.
\begin{prop}
Assume that $A_{k}^{\eta}(w)\neq\emptyset$. Then it contains a unique maximal element, that is larger than every other element in that set. \label{maxelt}
\end{prop}

\begin{proof}
Consider $\beta \in A_{k}^{\eta}(w)$, and assume that there is another element in this set, say $\alpha$, with $\beta\not\geq\alpha$. Hence there is an index $1 \leq j \leq k$ with $\beta_{j}<\alpha_{j}$. We are in the $k=l$ setting of Lemma \ref{interup}, which produces an element $\tilde{\beta} \in A_{k}(w)$, with $\tilde{\beta}>\beta$, and another element $\tilde{\alpha}$ with $\tilde{\alpha}+\tilde{\beta}=\alpha+\beta$. Then Lemma \ref{conteta} implies that $\tilde{\beta}$ is also in $A_{k}^{\eta}(w)$, so that if an element of our non-empty set does not contain all the other elements in that set, then it is not maximal there. But since $A_{k}(w)$ is finite (Remark \ref{forId}), so is $A_{k}^{\eta}(w)$, and as its is non-empty, it must contain a maximal element. By what we proved, this maximal element is the unique one, which is larger than all the others there. This proves the proposition.
\end{proof}

We can now make the following definition.
\begin{defn}
Denote, under the non-emptiness assumption, the unique maximal element from Proposition \ref{maxelt} by $\alpha_{\max}^{l,w}(\eta)$. If $k+l=2|\eta_{2}|+|\eta_{1}|$ then let $\beta_{\min}^{k,w}(\eta)$ be the unique set such that $\alpha_{\max}^{l,w}(\eta)+\beta_{\min}^{k,w}(\eta)$. By interchanging $k$ and $l$ we obtain the elements $\beta_{\max}^{k,w}(\eta)$ and $\alpha_{\min}^{l,w}(\eta)$. In case $k$, $l$, and $w$ are clear from the context we may write just $\alpha_{\max}(\eta)$, $\beta_{\max}(\eta)$, $\alpha_{\min}(\eta)$ and $\beta_{\min}(\eta)$, and we may also remove $\eta$ from the notation when no ambiguity may arise. \label{minmaxelts}
\end{defn}
The conditions in Proposition \ref{maxelt} implies that the complements from Definition \ref{minmaxelts} are sets (rather than multisets), which are then clearly contained in $\eta$ as well. In fact, $\beta_{\max}^{k,w}(\eta)$ means exactly the same as $\alpha_{\max}^{k,w}(\eta)$ there, and similarly $\alpha_{\min}^{l,w}(\eta)$ is just $\beta_{\min}^{l,w}(\eta)$. But based on the fact that when $k<l$ we use $\alpha$ (with decorations) for elements of $A_{l}(w)$ and $\beta$ (with indices) when considering elements of $A_{k}(w)$, we defined the multiple notation, which can now keep its meaning when some parameters are omitted.

Note that given $k$, $l$, $w$, and $\eta$, the elements $\beta_{\min}$ and $\alpha_{\min}$ need not be in $A_{k}(w)$ and $A_{l}(w)$ respectively. More precisely, we have the following result.
\begin{lem}
For such $k$, $l$, $w$, and $\eta$, we have $\beta_{\min} \in A_{k}(w)$ if and only if $\alpha_{\min} \in A_{l}(w)$ if and only if $\eta\in\tilde{B}_{k,l}(w)$. \label{relsminBkl}
\end{lem}

\begin{proof}
Since $\alpha_{\max} \in A_{l}(w)$, $\beta_{\max} \in A_{k}(w)$, and $\eta=\alpha_{\max}+\beta_{\min}=\alpha_{\min}+\beta_{\max}$ by Definition \ref{minmaxelts}, if either of the elements with index $\min$ is in $A(w)$ then $\eta\in\tilde{B}_{k,l}(w)$ directly from Definition \ref{Bklwdef}.

Conversely, assume that $\eta\in\tilde{B}_{k,l}(w)$, so that there are $\alpha \in A_{l}(w)$ and $\beta \in A_{k}(w)$ with $\eta=\alpha+\beta$ (Definition \ref{Bklwdef} again). Lemma \ref{conteta} implies that $\alpha$ is in $A_{l}^{\eta}(w)$ from Definition \ref{setswitheta} and $\beta$ is in $A_{k}^{\eta}(w)$. Now, if $\beta\neq\beta_{\max}$, then the proofs of Proposition \ref{maxelt} and Lemma \ref{interup} imply that there is another element $\tilde{\beta} \in A_{k}^{\eta}(w)$, which is of the form $R_{j,r}(\beta)$ for some $r>\beta_{j}$, which is larger than $\beta$ in that set. But this $r$ an element of $\eta$ that is not in $\beta$, meaning that it is in $\alpha$ (since $\alpha+\beta=\eta$), so that $r=\alpha_{i}$ for some $i$. We thus define $\tilde{\alpha}=R_{i,\beta_{j}}(\alpha)\leq\alpha$ (which is thus in $A_{l}(w)$), and we get $\tilde{\alpha}+\tilde{\beta}=\eta$ (so $\tilde{\alpha} \in A_{l}^{\eta}(w)$ by Lemma \ref{conteta}).

We can repeat this process, and remain with $\eta$ as the sum of an element from $A_{l}^{\eta}(w)$ and one from $A_{k}^{\eta}(w)$, until we replace $\beta$ by $\beta_{\max}$. But then $\alpha$ becomes $\alpha_{\min}$ (by Definition \ref{minmaxelts}), and it is in $A_{l}(w)$ as desired. The same argument but with increasing $\alpha$ to $\alpha_{\max}$ shows that $\beta_{\min}$ must be in $A_{k}(w)$ as well. This proves the lemma.
\end{proof}

We now deduce the following useful consequence.
\begin{cor}
Assume that $\eta\in\tilde{B}_{k,l}(w)$. Then we have $\alpha_{\max}\geq\alpha_{\min}$ and $\beta_{\max}\geq\beta_{\min}$. Moreover, given an increasing sequence $\beta$ of $k$ integers that contains $\eta_{2}$ and is contained in $\eta$, we have $\beta_{\min}\leq\beta\leq\beta_{\max}$ if and only if $\beta \in A_{k}(w)$ and there is $\alpha \in A_{l}(w)$ such that $\alpha+\beta=\eta$. The same assertion holds for increasing sequences $\alpha$ of $l$ integers with $\eta_{2}\subseteq\alpha\subseteq\eta$. \label{setofpairs}
\end{cor}

\begin{proof}
The first assertion follows directly from Lemma \ref{relsminBkl} and Proposition \ref{maxelt}. For the second one, if $\alpha+\beta=\eta$ for such $\alpha$ and $\beta$ then $\alpha \in A_{l}^{\eta}(w)$ and $\beta \in A_{k}^{\eta}(w)$ (by Lemma \ref{conteta}) and thus satisfy $\alpha\leq\alpha_{\max}$ and $\beta\leq\beta_{\max}$ (via Proposition \ref{maxelt}), hence also $\beta_{\min}\leq\beta$ and $\alpha_{\min}\leq\alpha$ by taking the complements from Definition \ref{minmaxelts}. Conversely, if $\eta_{2}\subseteq\beta\subseteq\eta$ then there is a sequence (rather than a multi-set) $\alpha$ with $\alpha+\beta=\eta$, and from $\beta\leq\beta_{\max}$ we deduce that $\beta \in A_{k}(w)$ (since $\beta_{\max} \in A_{k}(w)$ by definition), while $\beta_{\min}\leq\beta$ implies $\alpha\leq\alpha_{\max}$ and hence $\alpha \in A_{l}(w)$ by the same argument. The proof for $\eta_{2}\subseteq\alpha\subseteq\eta$ as analogous. This proves the corollary.
\end{proof}

The next result relates, for fixed $k$, $l$, $w$, and $\eta$ (with $k \leq l$ as always), the elements $\alpha_{\max} \in A_{l}(w)$ and $\beta_{\max}$, when they exist.
\begin{prop}
Take any $\eta=\eta_{1}+2\eta_{2}$ for which $A_{l}^{\eta}(w)$ and $A_{k}^{\eta}(w)$ are non-empty, so that $\alpha_{\max}$ and $\beta_{\max}$ from Definition \ref{minmaxelts} both exist. Then we have $\beta_{\max}\subseteq\alpha_{\max}$ and also $\beta_{\min}\cap\beta_{\max}=\eta_{2}$. \label{contmax}
\end{prop}

\begin{proof}
The first assertion is trivial when $k=l$, so assume $k<l$. Suppose, to the contrary, that this is not the case, denote the elements of $\beta_{\max}$ as $\{\beta_{j}\}_{j=1}^{k}$ in increasing order as usual, and by assumption there exists some $1 \leq j \leq k$ such that $\beta_{j}$ is not in $\alpha_{\max}$. In particular it does not equal $\alpha_{\varphi_{k,l}(j)}$, and if $\beta_{j}$ is smaller than this element, the we are in the setting of Lemma \ref{interup}, while when $\beta_{j}$ is the larger element we can apply Lemma \ref{interdown}.

Let $\tilde{\alpha} \in A_{l}(w)$ and $\tilde{\beta} \in A_{k}(w)$ be the resulting elements. Since $\alpha_{\max} \in A_{l}^{\eta}(w)$ and $\beta_{\max} \in A_{k}^{\eta}(w)$, the equality $\tilde{\alpha}+\tilde{\beta}=\alpha_{\max}+\beta_{\max}$ implies, via Lemma \ref{conteta}, that $\tilde{\alpha} \in A_{l}^{\eta}(w)$ and $\tilde{\beta} \in A_{k}^{\eta}(w)$ as well. But the construction from Lemma \ref{interup} implies that $\tilde{\beta}>\beta_{\max}$, and when we applied Lemma \ref{interdown} we obtain $\tilde{\alpha}>\alpha$. This means that in both situations we got to a contradiction of the definition of one of these elements via Proposition \ref{maxelt} and Definition \ref{minmaxelts}, so that the situation in which $\beta_{\max}\nsubseteq\alpha_{\max}$ cannot occur. This proves the asserted containment.

Now, as $\beta_{\min}+\alpha_{\max}=\eta$ (also when $k=l$), we deduce that $\beta_{\min}$ is the union of $\eta_{2}$, which is contained in $\beta_{\max}$, with the complement $(\eta_{1}\cup\eta_{2})\setminus\alpha_{\max}$, which is disjoint from $\beta_{\max}$ by the containment just established. The intersection of such a with $\beta_{\max}$ is thus $\eta_{2}$ as desired. This proves the proposition.
\end{proof}
Note that the condition that $\eta \in B_{k,l}(w)$, which is stronger than what we assumed in Proposition \ref{contmax}, is not necessary, and the result holds whenever all the objects are defined.

We deduce, as a consequence, the distinction between the two sets from Definition \ref{Bklwdef}.
\begin{prop}
Assume that $\eta$ is in $\tilde{B}_{k,l}(w)$, and write $\eta$ as $\alpha+\beta$ with elements $\alpha \in A_{l}(w)$ and $\beta \in A_{k}(w)$. Then for $k<l$ we get $\eta \not\in B_{k,l}(w)$ if and only if $\beta\subseteq\alpha$ if and only if $k=|\eta_{2}|$. When $l=k$ the condition $\eta \not\in B_{k,l}(w)$ occurs either in the same situation, or when $l=k\leq|\eta_{2}|+1$, and in no other setting. \label{Bkldiff}
\end{prop}

\begin{proof}
Recall from Lemma \ref{conteta} that in any presentation of $\eta$ we have $\alpha \in A_{l}^{\eta}(w)$ and $\beta \in A_{k}^{\eta}(w)$. Hence if $k=|\eta_{2}|$ then in any such presentation with one of the sets being of size $k$, that set must be precisely $\eta_{2}$, so in particular we have $\beta=\eta_{2}$, it is contained in $\alpha=\eta_{1}\cup\eta_{2}$, and this is the only presentation for $\eta$ and hence $\eta \not\in B_{k,l}(w)$. Conversely if $\beta\subseteq\alpha$ and $\alpha+\beta=\eta$ then the elements of $\beta$ are in $\eta_{2}$ and we have $k=|\eta_{2}|$. Next we observe that when $l=k=|\eta_{2}|+1$ we get $|\eta_{1}|=2$, so that when $\alpha+\beta=\eta$ both sets consist of the elements of $\eta_{2}$ plus a single element from $\eta_{1}$ (Lemma \ref{conteta} again). This means that up to interchanging the sets there is only one presentation for $\eta$ as such a sum, and $\eta \not\in B_{k,l}(w)$ once again.

It remains to prove that these are the only situations for $\eta$ to be in $\tilde{B}_{k,l}(w)$ but not in $B_{k,l}(w)$. Note that Lemma \ref{relsminBkl} implies that $\alpha_{\min} \in A_{l}(w)$ and $\beta_{\min} \in A_{k}(w)$, so that the presentations $\eta=\alpha_{\max}+\beta_{\min}=\alpha_{\min}+\beta_{\max}$ from Definition \ref{minmaxelts} are both with summands from the required sets. If $k<l$ then for $\eta \not\in B_{k,l}(w)$ both presentations must the same, with the letter ($\alpha$ or $\beta$) indicating the index ($l$ or $k$), so that both become $\eta=\alpha_{\max}+\beta_{\max}$. The containment from Proposition \ref{contmax} yields the assertion in this case.

When $k=l$ we write just $\eta=\beta_{\max}+\beta_{\min}$, where we recall that $\beta_{\min} \in A_{k}(w)$ by Lemma \ref{relsminBkl} and hence $\beta_{\min}\leq\beta_{\max}$ by Proposition \ref{maxelt}. If the sets are the same then we are again in the situation where $l=k=|\eta_{2}|$, so assume otherwise. Then is an entry $j$ whose value in $\beta_{\min}$ is smaller than the corresponding entry of $\beta_{\max}$, and we apply Lemma \ref{interup} to obtain an element $\tilde{\beta}>\beta_{\min}$ in $A_{k}(w)$, which is of the form $R_{j,r}(\beta_{\min})$ for some $j$ and $r$, and another element $\tilde{\alpha} \in A_{k}(w)$ such that $\tilde{\alpha}+\tilde{\beta}=\beta_{\max}+\beta_{\min}=\eta$. But if we assume that $\eta \not\in B_{k,l}(w)$ then this must be the same presentation or its interchange, and the fact that $\tilde{\beta}\neq\beta_{\min}$ implies that this is indeed the interchange. Hence $\beta_{\max}=\tilde{\beta}=R_{j,r}(\beta_{\min})$, so that $\beta_{\min}\cap\beta_{\max}$, which equals $\eta_{2}$ by Proposition \ref{contmax}, contains all the elements of $\beta_{\min}$ except for the $j$th one, and is indeed of size $k-1$ as desired. This completes the proof of the proposition.
\end{proof}
It follows from Proposition \ref{Bkldiff} that for determining whether a particular multi-set $\eta=\eta_{1}+2\eta_{2}$ lies in $B_{k,l}(w)$ for a given $w$, it suffices to check the existence of a single presentation of $\eta$ as a sum as in Definition \ref{Bklwdef} (i.e., to see if $\eta\in\tilde{B}_{k,l}(w)$), as well as verify that $|\eta_{2}|$ is smaller than $k$ if $k<l$, and smaller than $k-1$ in case $k=l$. Indeed, this is not the case for the elements of $\tilde{B}_{2,3}(w)$ given in Equation \eqref{B23wex} for $w=42531$, while the condition from Proposition \ref{Bkldiff} is satisfied for the elements of $B_{2,3}(w)$ showing up right before it. The latter condition can be written in a uniform manner as $|\eta|_{2}<k-\delta_{k,l}$, where the latter symbol is the Kronecker $\delta$-symbol, yielding 1 when $k=l$ and 0 otherwise.

\smallskip

We will also need to see how the sets $B_{k,l}(w)$ changes when we replace $w$ by $s_{i}w$ for some index $i$ with $\ell(s_{i}w)>\ell(w)$.
\begin{prop}
If $\ell(s_{i}w)>\ell(w)$ then the set $\tilde{B}_{k,l}(s_{i}w)$ consists of elements of $\tilde{B}_{k,l}(w)$, their $s_{i}$-images, and the sums of elements from $A_{l,i}(w)$ with $s_{i}$-images of those from $A_{k,i}(w)$. For $B_{k,l}(s_{i}w)$ we take $B_{k,l}(w)$, its $s_{i}$-images, and the sums, except for the sums $\alpha+s_{i}\alpha$ for $\alpha \in A_{l,i}$ in case $k=l$. Both sets are preserved under the action of $s_{i}$. \label{Bsiw}
\end{prop}

\begin{proof}
If $\eta\in\tilde{B}_{k,l}(w)$ then $\eta=\alpha+\beta$ for $\alpha \in A_{l}(w)$ and $\beta \in A_{k}(w)$ via Definition \ref{Bklwdef}. Since Proposition \ref{Asiw} shows that $\alpha \in A_{l}(s_{i}w)$ and $\beta \in A_{k}(s_{i}w)$, we get that $\eta\in\tilde{B}_{k,l}(s_{i}w)$, as desired. Moreover, the invariance in Proposition \ref{Asiw} implies that $\alpha \in A_{l}(s_{i}w)$ and $\beta \in A_{k}(s_{i}w)$ as well, meaning that $s_{i}\eta=s_{i}\alpha+s_{i}\beta$ also lies in $\tilde{B}_{k,l}(s_{i}w)$. Proposition \ref{Bkldiff} yields that for being in $B_{k,l}(s_{i}w)$ we need to omit those elements $\eta$ in which $|\eta_{2}|=k$, as well as those with $|\eta_{2}|=k-1$ if $l=k$. Moreover, the action of $s_{i}$ clearly preserves the union of these sets by definition, yielding the invariance of this part.

The only other way in which $\eta$ can be in $\tilde{B}_{k,l}(s_{i}w)$ is if we can write $\eta=\alpha+\beta$ with at least one of $\alpha$ and $\beta$ is not in $A(w)$, and at least one of their $s_{i}$-images is not in $A(w)$. Since $s_{i}$ takes $A(s_{i}w) \setminus A(w)$ into $A(w)$, we either have $s_{i}\alpha \not\in A_{l}(w)$ (hence $\alpha \in A_{l,i}(w)$) and $\beta \not\in A_{k}(w)$ (and thus $\beta$ is the $s_{i}$-image of an element of $A_{k,i}(w)$), or the other way around. But in the latter case $\alpha$ contains $i+1$ but not $i$, $\beta$ contains $i$ but not $i+1$, and both $i$ and $i+1$ are in $\eta_{1}$, meaning that $\eta=s_{i}\eta$ and we can apply $s_{i}$ to both $\alpha$ and $\beta$ and return to the previous case. This establishes the result for $\tilde{B}_{k,l}(s_{i}w)$, and all these sums are thus invariant under $s_{i}$.

For $B_{k,l}(s_{i}w)$, when $k<l$ the case where $\eta=\alpha+\beta$ satisfies $|\eta_{2}|=k$ arises only when $\beta\subseteq\alpha$, which cannot occur when either one of $i$ or $i+1$ is in $\beta$ but not in $\alpha$. If $l=k$ then we also exclude the case with $|\eta_{2}|=k-1$ and $|\eta_{1}|=2$, which implies that $\eta_{1}=\{i,i+1\}$, $\alpha=\eta_{2}\cup\{i\} \in A_{l,i}(w)$, with complement $\eta_{2}\cup\{i+1\}$ to $\eta$. This completes the proof of the proposition.
\end{proof}
Note that the proof of Proposition \ref{Bsiw} shows that in sums of $s_{i}$-images of elements from $A_{l,i}(w)$ with elements of $A_{k,i}(w)$ we can take the $s_{i}$ to the second summand without affecting the sum, and we can use both presentations in what follows.

\section{Key Polynomials and Generating Functions \label{KeyPols}}

Consider the ring of polynomials in $n$ variables $\{x_{i}\}_{i=1}^{n}$ (or infinitely many variables), say over $\mathbb{Q}$ or over $\mathbb{C}$, elements of which we will write as just $f(x)$. Then for every $1 \leq i<n$ we can let the simple transposition $s_{i}$ act on each polynomial by interchanging $x_{i}$ and $x_{i+1}$, and then recall the \emph{divided difference operator} $\partial_{i}$ and its extension $\pi_{i}$, which are defined, on a polynomial $f$, by \[\partial_{i}(f):=\frac{f-s_{i}f}{x_{i}-x_{i+1}}\qquad\mathrm{and}\qquad\pi_{i}(f):=\partial_{i}(x_{i}f).\] The operator $\pi_{i}$ is an idempotent, namely satisfy $\pi_{i}^{2}=\pi_{i}$, and the family of operators satisfy the \emph{braid relations}, namely $\pi_{i}$ and $\pi_{j}$ commute when $|i-j|>1$ and the cubic relation $\pi_{i}\circ\pi_{i+1}\circ\pi_{i}=\pi_{i+1}\circ\pi_{i}\circ\pi_{i+1}$ holds for any $1 \leq i<n-1$.

The braid relations are the defining relations for $S_{n}$ as generated by $\{s_{i}\}_{i=1}^{n-1}$. The fact that the $\pi_{i}$'s satisfy them implies that given a permutation $w$, we can take any presentation of it as a composition $s_{i_{1}}\circ\ldots \circ s_{i_{\ell}}$ with $\ell=\ell(\sigma)$, and then the composition $\pi_{w}:=\pi_{i_{1}}\circ\ldots\circ\pi_{i_{\ell}}$ is independent of the decomposition chosen.

Recall that a \emph{partition} is a non-increasing sequence $\lambda=\{\lambda_{i}\}_{i=1}^{\infty}$ of non-negative integers, which vanish for large enough $i$. We say that $\lambda$ has \emph{length} $n$, written $\ell(\lambda)=n$, in case $\lambda_{n}\neq0$ and $\lambda_{n+1}=0$, and make the following definition.
\begin{defn}
Let $\lambda$ be a partition with $\ell(\lambda) \leq n$, and take a permutation $w \in S_{n}$. Then we define $x^{\lambda}$ to be the monomial $\prod_{i=1}^{n}x_{i}^{\lambda_{i}}$ in the ring of (at least) $n$ variables, and we define the \emph{key polynomial} $K_{\lambda,w}(x)$ to be $\pi_{w}(x^{\lambda})$. \label{keypoldef}
\end{defn}

\begin{rmk}
The key polynomials from Definition \ref{keypoldef} are also known as \emph{Demazure characters}, as they are obtained as the values at diagonal matrices of characters of representations of the Borel subgroup of $\operatorname{GL}_{n}$. It is clear that by considering $\lambda$ as a partition of length at most $n+1$, and by identifying $w$ with its image in $S_{n+1}$, the key polynomial $K_{\lambda,w}(x)$ remains the same, and can therefore be defined for an arbitrary partition $\lambda$ (with no restriction on the length) and with $w \in S_{\infty}$. \label{Kindepn}
\end{rmk}

We remark that in many references, the key polynomials are indexed by \emph{weak compositions}, namely (finite, or essentially finite) sequences $\nu=\{\nu_{i}\}_{i=1}^{\infty}$ of non-negative integers. The \emph{length} $\ell(\nu)$ of a weak composition $\nu$ is $n$ if $\nu_{n}\neq0$ and $\nu_{i}=0$ for all $i>n$. To any weak composition one can attach the partition $\lambda$ obtained by ordering the $\nu_{i}$'s in a decreasing order, and then some permutation $w$ on the indices takes this partition to $\nu$. It is clear that $\ell(\nu) \leq n$ then $\ell(\lambda) \leq n$, and we can take $w$ to be in $S_{n}$. Conversely, applying any permutation $w$ to any partition $\lambda$ produces a weak composition $\nu$, and if $w \in S_{n}$ and $\ell(\lambda) \leq n$ then $\ell(\nu) \leq n$ as well. Then the polynomial $K_{\lambda,w}$ is also denoted by $K_{\nu}$. Note that when some of the numbers can appear in $\nu$ and $\lambda$ with multiplicities, so $w$ is not unique, but in this situation the idempotent property of the $\pi_{i}$'s gives that $K_{\lambda,w}$ gives the same polynomial for every $w$ yielding the same $\nu$, making these notions well-defined.

\smallskip

The total generating function $\mathcal{K}$ of the key polynomials, say in $n$ variables, will have to take both partitions and permutations into account. We thus consider a \emph{twisted Demazure algebra}, or the \emph{Hecke algebra} with indices $-1$ and 0, which is generated by $\varepsilon_{i}=\varepsilon_{s_{i}}$ with $1 \leq i<n$, which satisfy the equality $\varepsilon_{i}^{2}=-\varepsilon_{i}$ for every $i$ as well as the braid relations in products. It is free with a basis $\{\varepsilon_{v}\}_{v \in S_{n}}$ (defined in the obvious manner), and the product $\varepsilon_{i}\varepsilon_{v}$ thus equals $\varepsilon_{s_{i}v}$ in case $\ell(s_{i}v)>\ell(v)$, and just $-\varepsilon_{v}$ when $\ell(s_{i}v)<\ell(v)$. We denote by $w_{0}$ the permutation of maximum length in $S_{n}$, whose one-line notation is the numbers from 1 to $n$ in decreasing order, and for partitions we simply add another set of variables $\{t_{i}\}_{i=1}^{n}$, hence obtaining functions of both $x$ and $t$ representing both variables, with $t^{\lambda}$ defined similarly to $x^{\lambda}$ from Definition \ref{keypoldef} for any partition $\lambda$ with $\ell(\lambda) \leq n$. We then define
\begin{equation}
\mathcal{K}^{(n)}(x,t):=\sum_{\ell(\lambda) \leq n}\sum_{w \in S_{n}}K_{\lambda,w}(x)t^{\lambda}\varepsilon_{ww_{0}}=\sum_{w \in S_{n}}\mathcal{K}_{w}^{(n)}(x,t)\varepsilon_{ww_{0}}, \label{genser}
\end{equation}
introducing the components $\mathcal{K}_{w}^{(n)}(x,t):=\sum_{\ell(\lambda) \leq n}K_{\lambda,w}(x)t^{\lambda}$ associated with every $w \in S_{n}$, which are related by the following result.
\begin{lem}
If $\pi_{i}$ acts on the $x$-variables, then $\pi_{i}\big(\mathcal{K}_{w}^{(n)}(x,t)\big)$ equals $\mathcal{K}_{s_{i}w}^{(n)}(x,t)$ in case $\ell(s_{i}w)>\ell(w)$, and just $\mathcal{K}_{w}^{(n)}(x,t)$ back again when $\ell(s_{i}w)>\ell(w)$. \label{piiKw}
\end{lem}

\begin{proof}
By definition, when $\ell(s_{i}w)>\ell(w)$ we can write $\pi_{s_{i}w}=\pi_{i}\circ\pi_{w}$, meaning that $K_{\lambda,s_{i}w}(x)=\pi_{i}\big(K_{\lambda,w}(x)\big)$ for every partition $\lambda$, by Definition \ref{keypoldef}. When $\ell(s_{i}w)<\ell(w)$ we have $w=s_{i} \cdot s_{i}w$ and hence $\pi_{w}=\pi_{i}\circ\pi_{s_{i}w}$ as above, and the same argument gives $K_{\lambda,w}(x)=\pi_{i}\big(K_{\lambda,s_{i}w}(x)\big)$. But then \[\pi_{i}\big(K_{\lambda,w}(x)\big)=\pi_{i}^{2}\big(K_{\lambda,s_{i}w}(x)\big)=\pi_{i}\big(K_{\lambda,s_{i}w}(x)\big)=K_{\lambda,w}(x)\] by the idempotent property of $\pi_{i}$. Multiplying by $t^{\lambda}$ and taking the sum over partitions $\lambda$ with $\ell(\lambda) \leq n$ yields the desired result. This proves the lemma.
\end{proof}

The inductive construction of each $K_{\lambda,w}$ by the length of $w$ through applications of the $\pi_{i}$'s transfers to a functional equation for the large generating series from Equation \eqref{genser}, as is given by the following result, which imitates the one for the generating series of Schubert and Grothendieck polynomials in, e.g., \cite{[FK]} or the appendix of \cite{[H]}.
\begin{prop}
The series $\mathcal{K}^{(n)}(x,t)$ from Equation \eqref{genser} satisfies the functional equation $\pi_{i}\big(\mathcal{K}^{(n)}(x,t))=(\varepsilon_{i}+1)\mathcal{K}^{(n)}(x,t)$ for every $1 \leq i<n$, where $\pi_{i}$ acts on the $x$-variables and the multiplication by $\varepsilon_{i}$ is done in the Demazure algebra part. \label{propgen}
\end{prop}

\begin{proof}
We will suppress the variables $x$ and $t$ as well as the superscript $n$, and recall that every element of $S_{n}$ can be written either as $w$ with $\ell(s_{i}w)>\ell(w)$, or as $s_{i}w$ with the same inequality $\ell(s_{i}w)>\ell(w)$. The definition in Equation \eqref{genser} and Lemma \ref{piiKw} yield \[\mathcal{K}=\!\!\!\sum_{\ell(s_{i}w)>\ell(w)}\!\!\!(\mathcal{K}_{w}\varepsilon_{ww_{0}}+\mathcal{K}_{s_{i}w}\varepsilon_{s_{i}ww_{0}})\mathrm{\ and\ }\pi_{i}(\mathcal{K})=\!\!\!\sum_{\ell(s_{i}w)>\ell(w)}\!\!\!\mathcal{K}_{s_{i}w}(\varepsilon_{ww_{0}}+\varepsilon_{s_{i}ww_{0}}).\]

We now recall (see, e.g., \cite{[P]}) that since the product of $w^{-1}$ and $ww_{0}$ is the longest word $w_{0}$, we have $\ell(ww_{0})=\ell(w_{0})-\ell(w^{-1})=\ell(w_{0})-\ell(w)$ (as $w$ and $w^{-1}$ have the same length). Applying the same argument for $s_{i}w$, we deduce that $\ell(s_{i}w)>\ell(w)$ if and only if $\ell(s_{i}ww_{0})<\ell(ww_{0})$. It follows, by the definition of our Demazure algebra, that $\varepsilon_{i}\varepsilon_{s_{i}ww_{0}}=\varepsilon_{ww_{0}}$ but $\varepsilon_{i}\varepsilon_{ww_{0}}=-\varepsilon_{ww_{0}}$, so that right multiplication by $\varepsilon_{i}+1$ annihilates $\varepsilon_{ww_{0}}$ and takes $\varepsilon_{s_{i}ww_{0}}$ to $\varepsilon_{s_{i}ww_{0}}+\varepsilon_{ww_{0}}$. Thus applying it to the sum $\mathcal{K}$ yields the same expression as $\pi_{i}(\mathcal{K})$, as asserted. This proves the proposition.
\end{proof}

\begin{rmk}
The generating series of the Schubert or Grothendieck polynomials from \cite{[FK]} or \cite{[H]} does not depend on $n$, in the sense that identifying each $w \in S_{n}$ with its image in $S_{n+1}$ produces the same Schubert polynomial and the same multiplier from the relevant Demazure algebra. In Equation \eqref{genser} the appearance of $w_{0}$, which depends on $n$ makes that no longer the case. It is required in Proposition \ref{propgen}, since in Definition \ref{keypoldef} the operation of $\pi_{i}$ increases the length of the word $w$ in the index, with $\operatorname{Id}$ being the basic case, while for the Schubert and Grothendieck polynomials $w_{0}$ is the basic case and the action of $\partial_{i}$ (or the modification required for the Grothendieck polynomials) reduces the length of the word. As for the components $\mathcal{K}_{w}^{(n)}$, it follows from Remark \ref{Kindepn} that if $w \in S_{n}$ is viewed as an element of $S_{n+1}$ then the summands in $\mathcal{K}_{w}^{(n+1)}$ for partitions $\lambda$ with $\ell(\lambda) \leq n$ yield the original $\mathcal{K}_{w}^{(n)}$. We can therefore look at $\mathcal{K}_{w}^{(n)}$ as the expression obtained by setting $t_{n+1}=0$, or equivalently $x_{n+1}=t_{n+1}=0$, in $\mathcal{K}_{w}^{(n+1)}$. This allows us to consider a generating series $\mathcal{K}_{w}$ in infinitely many variables $x_{i}$ and infinitely many $t_{i}$'s, with no restriction on the length of the summation index $\lambda$ either, and substituting $x_{i}=t_{i}=0$, or just $t_{i}=0$, for every $i>n$ in this $\mathcal{K}_{w}$ will produce the generating functions $\mathcal{K}_{w}^{(n)}$ in finitely many variables. Lemma \ref{piiKw} holds for $\mathcal{K}_{w}$ as well by the same argument. \label{deponn}
\end{rmk}

\smallskip

We now begin to obtain the properties of $\mathcal{K}_{w}$ and the $\mathcal{K}_{w}^{(n)}$'s. We will first introduce a shortened notation, that will show up in all our formulae. For any $l\geq1$ we define $T_{l}:=\prod_{i=1}^{l}t_{i}$, and similarly $X_{l}:=\prod_{i=1}^{l}x_{i}$, and obtain the following expression.
\begin{prop}
The sum $\mathcal{K}_{\operatorname{Id}}^{(n)}(x,t)$ is given by $1\big/\prod_{l=1}^{n}(1-X_{l}T_{l})$. In infinitely many variables we have $\mathcal{K}_{\operatorname{Id}}(x,t)=1\big/\prod_{l=1}^{\infty}(1-X_{l}T_{l})$. \label{KId}
\end{prop}

\begin{proof}
Since we have $K_{\lambda,\operatorname{Id}}=x^{\lambda}$ for every partition $\lambda$, the expression that we need is $\sum_{\ell(\lambda) \leq n}x^{\lambda}t^{\lambda}$. Now, a partition $\lambda$ with $\ell(\lambda) \leq n$ is determined by the set of non-negative integers $h_{l}:=\lambda_{l}-\lambda_{l+1}$ (with $\lambda_{n+1}=0$), and any choice of such $n$ numbers is associated with a unique partition $\lambda$, which is given explicitly by $\lambda_{i}=\sum_{l=i}^{n}h_{l}$. It follows that if $\lambda$ is associated with these $h_{l}$'s then \[x^{\lambda}t^{\lambda}=\prod_{i=1}^{n}x_{i}^{\lambda_{i}}t_{i}^{\lambda_{i}}=\prod_{i=1}^{n}\prod_{l=i}^{n}x_{i}^{h_{l}}t_{i}^{h_{l}}= \prod_{l=1}^{n}\prod_{i=1}^{l}x_{i}^{h_{l}}t_{i}^{h_{l}}=\prod_{l=1}^{n}X_{l}^{h_{l}}T_{l}^{h_{l}},\] and as taking the sum over $\lambda$ amounts to taking the sum over all possible choices of the $h_{l}$'s independently, we deduce from the formula for geometric series that \[\mathcal{K}_{\operatorname{Id}}^{(n)}(x,t)=\sum_{h_{1}=0}^{\infty}\ldots\sum_{h_{n}=0}^{\infty}\prod_{l=1}^{n}X_{l}^{h_{l}}T_{l}^{h_{l}}= \prod_{l=1}^{n}\sum_{h_{l}=0}^{\infty}X_{l}^{h_{l}}T_{l}^{h_{l}}=\prod_{l=1}^{n}\frac{1}{1-X_{l}T_{l}},\] as desired. In infinitely many variables we either take the inverse limit over $n$, or apply the same considerations but now with $h_{l}$ for all $l\geq1$, and obtain the asserted formula. This proves the proposition.
\end{proof}

\begin{rmk}
The generating function $\mathcal{K}_{w}$ for $w \in S_{n}$ is more natural than $\mathcal{K}_{w}^{(n)}$ since the dependence on $n$, as presented in Remark \ref{deponn}, plays no role in all the calculations to follow. However, note that $\mathcal{K}_{\operatorname{Id}}^{(n)}$ from Proposition \ref{KId}, as a generating series, is the series of a rational function, while $\mathcal{K}_{\operatorname{Id}}$ itself is not because its denominator is the product of infinitely many expressions. This will be the case for any permutation $w$, as we see in Theorem \ref{formofKw} below. \label{ratfunc}
\end{rmk}

For stating and prove the basic form of $\mathcal{K}_{w}^{(n)}$ and $\mathcal{K}_{w}$ for any permutation $w$, we recall the sets $A_{l}(w)$ from Definition \ref{Asetdef}. Take some index $i$ such that $\ell(s_{i}w)>\ell(w)$, and so that Proposition \ref{Asiw} gives that $A(w) \subsetneq A(s_{i}w)$ and $A(s_{i}w) \setminus A(w)=\bigcup_{l}\{s_{i}(\alpha)|\alpha \in A_{l,i}(w)\}$. We now make the following definition.
\begin{defn}
For every $w \in S_{\infty}$ let $P_{w}(x,t)$ be the polynomial defined inductively on $\ell(w)$ as follows. We set $P_{\operatorname{Id}}:=1$, and if $P_{w}$ is given and $\ell(s_{i}w)>\ell(w)$, then let $N_{w,i}(x,t):=\prod_{l=1}^{\infty}\prod_{\alpha \in A_{l,i}(w)}(1-x^{s_{i}\alpha}T_{l})$, and set $P_{s_{i}w}:=\pi_{i}(P_{w}N_{w,i})$, with $\pi_{i}$ acting on the $x$-variables as always. \label{Pwdef}
\end{defn}
The finiteness from Lemma \ref{nontrivAli} shows that the union yielding $A(s_{i}w) \setminus A(w)$ is essentially over finitely many values of $l$, so that $N_{w,i}$ from Definition \ref{Pwdef} is indeed a polynomial, hence by induction so is $P_{w}$ since the operators $\pi_{i}$ take polynomials to polynomials. Note that if $w \in S_{n}$ and $i<n$ then this difference is given by the $s_{i}$-images of $\bigcup_{l=1}^{n}A_{l,i}(w)$ (including some empty sets), making the finiteness more visible.

Our first main result is as follows.
\begin{thm}
Assume that $w \in S_{n}$. Then we have the expressions \[\mathcal{K}_{w}^{(n)}(x,t)=\frac{P_{w}(x,t)}{\prod_{l=1}^{n}\prod_{\alpha \in A_{l}(w)}(1-x^{\alpha}T_{l})},\ \mathcal{K}_{w}(x,t)=\frac{P_{w}(x,t)}{\prod_{l=1}^{\infty}\prod_{\alpha \in A_{l}(w)}(1-x^{\alpha}T_{l})}.\] \label{formofKw}
\end{thm}

\begin{proof}
We argue by induction on $\ell(w)$. When $\ell(w)=0$ and $w=\operatorname{Id}$, Remark \ref{forId} implies that $A_{l}(w)$ is the single element $\alpha$ for which $x^{\alpha}=X_{l}$ for every $l$, so that with $P_{\operatorname{Id}}=1$ by Definition \ref{Pwdef}, Proposition \ref{KId} yields the asserted result.

So assume that the formula is valid for some $w \in S_{n}$, and take $1 \leq i<n$ with $\ell(s_{i}w)>\ell(w)$. The denominator of $P_{w}(x,t)$ is then the product of the part $\prod_{l}\prod_{\alpha \in A_{l}^{(i)}(w)}(1-x^{\alpha}T_{l})$ from Equation \eqref{Alwdecom} (which is invariant under $s_{i}$ because each multiplier there is either invariant or comes multiplied by its $s_{i}$-image), and $\prod_{l}\prod_{\alpha \in A_{l,i}(w)}(1-x^{\alpha}T_{l})$. Therefore the denominator of $s_{i}P_{w}(x,t)$ is the product of the first expression and of $\prod_{l}\prod_{\alpha \in A_{l,i}(w)}(1-x^{s_{i}\alpha}T_{l})$, which is $N_{w,i}(x,t)$ from Definition \ref{Pwdef}.

Proposition \ref{Asiw} then implies that the common denominator of $\mathcal{K}_{w}^{(n)}$ (or $\mathcal{K}_{w}$) and its $s_{i}$-image is $\prod_{l}\prod_{\alpha \in A_{l}(s_{i}w)}(1-x^{\alpha}T_{l})$ (with $l$ going up to $n$ or up to $\infty$), which is invariant under $s_{i}$. More precisely, these considerations express $\mathcal{K}_{w}^{(n)}$ (or $\mathcal{K}_{w}$) as \[\frac{P_{w}(x,t)N_{w,i}(x,t)}{\prod_{l}\prod_{\alpha \in A_{l}(s_{i}w)}(1-x^{\alpha}T_{l})},\] with the $s_{i}$-invariant denominator, and Lemma \ref{piiKw} expresses the series associated with $s_{i}w$ as its $\pi_{i}$-image, where the denominator can be taken out of the action by $s_{i}$-invariance. But this denominator is the desired one for $\mathcal{K}_{s_{i}w}^{(n)}$ or $\mathcal{K}_{s_{i}w}$, and the resulting numerator is precisely $P_{s_{i}w}(x,t)$ by Definition \ref{Pwdef}. This proves the theorem.
\end{proof}
Note that the construction of $P_{w}$ in Definition \ref{Pwdef} may a priori depend on the choice of the construction of $w$ as a product of simple transpositions, but the fact that Theorem \ref{formofKw} expresses it as the well-defined generating series $\mathcal{K}_{w}^{(n)}$ (for large enough $n$) or $\mathcal{K}_{w}$ times an explicit denominator implies that $P_{w}$ is indeed well-defined for every $w$. As in Remark \ref{ratfunc}, $\mathcal{K}_{w}^{(n)}$ has the advantage of being a rational function while $\mathcal{K}_{w}$ is canonical but is a more general type of formal power series. However, the polynomial $P_{w}$ is independent of $n$ by definition.

\smallskip

As essentially follows from the construction in Definition \ref{Pwdef} and Theorem \ref{formofKw}, the dependence of $P_{w}(x,t)$ on the variables $\{t_{i}\}_{i}$ is through the products $\{T_{l}\}_{l}$ (note that as $t_{1}$ divides every $T_{l}$ precisely once, we can consider this expansion according to the degree in that single variable). We will therefore expand $P_{w}$ according to the degrees in the latter variables. For investigating this, we apply the Leibniz Rule for $\pi_{i}$ (which arises from that for $\partial_{i}$), and deduce that if $\ell(s_{i}w)>\ell(w)$ then
\begin{equation}
P_{s_{i}w}=\pi_{i}(P_{w}N_{w,i})=P_{w}\pi_{i}(N_{w,i})+x_{i+1}s_{i}N_{w,i} \cdot \partial_{i}(P_{w}). \label{Leibniz}
\end{equation}
We can thus obtain the description for the terms having low degrees in the $T_{l}$'s, as follows.
\begin{prop}
For every $w \in S_{n}$ we have $P_{w}=1-P_{w,2}+O(T^{3})$, where $O(T^{3})$ involves a (possibly empty) sum of terms which in the product of at least three $T$ terms. The term $P_{w,2}(x,t)$ is a (possible empty) sum of expressions of the form $m^{k,l}_{\eta}(w)x^{\eta}T_{k}T_{l}$, with $k \leq l$, where $\eta=\eta_{1}+2\eta_{2}$ is a multiset with $|\eta_{2}| \leq k$ and $|\eta_{1}|+2|\eta_{2}|=k+l$, and $m_{\eta}$ is an integer. \label{expPwinT}
\end{prop}
Proposition \ref{expPwinT} can be understood as the statement that if $P_{w,d}$ is the part of $P_{w}$ that has degree $d$ in the $T_{l}$'s, then $P_{w,0}$ is always 1, $P_{w,1}=0$, and $P_{w,2}$ is of the asserted form.

\begin{proof}
We again work by induction, with the case $w=\operatorname{Id}$ being clear, so that we assume that $P_{w}$ has the asserted form, and consider $P_{s_{i}w}$ for some $i$ with $\ell(s_{i}w)>\ell(w)$. We expand the latter as in Equation \eqref{Leibniz}, and note that since all the operations are in the $x$-variables, we can replace $P_{w}$ by $1-P_{w,2}$ and $N_{w,i}$ by its constant, linear, and quadratic terms in the $T_{l}$'s only, and the rest goes into the unspecified $O(T^{3})$ term.

We begin with the second term, where $\partial_{i}$ annihilates 1, so that only the part arising from the constant term 1 of $N_{w,i}$ contributes terms which are not $O(T^{3})$. This means that up to $O(T^{3})$, this term produces $-x_{i+1}\partial_{i}(P_{w,2})$, which is the same as $-\pi_{i}(P_{w,2})+P_{w,2}$.

For the first term, $N_{w,i}$ is a product over $A_{l,i}(w)$, where we recall from Lemma \ref{nontrivAli} that every element of $A_{l,i}(w)$ contains $i$ but does not contain $i+1$, so that the opposite holds for its $s_{i}$-image. Therefore every multiplier $x^{s_{i}\alpha}T_{l}$ in its the product defining $N_{w,i}$ is divisible by $x_{i+1}$ (once), but not by $x_{i}$. Hence when we apply $\pi_{i}$, the constant term produces 1, which when multiplied by $1-P_{w,2}$ combines with the previous expression to $1-\pi_{i}(P_{w,2})$. Moreover, as $\pi_{i}$ preserves the degree in $x$, we get that $\pi_{i}(P_{w,2})$ has the asserted property in case $P_{w,2}$ does.

Now, up to $O(T^{3})$, the remaining terms of $\pi_{i}(N_{w,i})$ are simply multiplied by 1, and we saw that each term that is linear in $T$ is $x_{i+1}$ times a monomial $x^{\tilde{\alpha}}T_{l}$ with $\tilde{\alpha}$ not involving $i$ or $i+1$. But then its $\pi_{i}$-image vanishes, as $x_{i}x_{i+1}x^{\tilde{\alpha}}T_{l}$ is invariant under $s_{i}$ and is thus annihilated by $\partial_{i}$. Finally, any quadratic term from $N_{w,i}$ is, by the same argument, of the form $x_{i+1}^{2}x^{\tilde{\alpha}+\tilde{\beta}}T_{k}T_{l}$ (with a positive sign), so that applying $\pi_{i}$ takes it to $-x_{i}x_{i+1}x^{\tilde{\alpha}+\tilde{\beta}}T_{k}T_{l}$, which is of the desired form. This proves the proposition.
\end{proof}

As examples, we consider some permutations of small length, except for the base case with $w=\operatorname{Id}$ and $\ell(w)=0$. If $\ell(w)=1$, so that $w=s_{i}$ for some $i$, then $N_{\operatorname{Id},i}$ is the single multiplier $1-X_{i-1}x_{i+1}T_{i}$, and $P_{s_{i}}=1$. When $w$ is the product of two simple transpositions $s_{i}$ and $s_{j}$ with $|i-j|\geq2$, a similar argument also produces $P_{w}=1$. In some situations $P_{w}$ is $s_{i}$-invariant and $N_{w,i}$ contains a single multiplier, yielding that $P_{s_{i}w}=P_{w}$. The first non-trivial values show up when $\ell(w)=2$ and $w$ is the product of two simple transpositions having adjacent indices. These considerations give $P_{w}=1$ for the permutations \[w\in\{12345,21345,13245,12435,12354,21435,21354,13254\} \subseteq S_{5}.\] We also have \[P_{w}=1-x_{1}x_{2}x_{3}T_{1}T_{2}\mathrm{\ for\ }w\in\{23145,23154,31245,31254,32145,32154\},\] with the permutations 13425, 14235, and 14325 yielding $P_{w}=1-x_{1}^{2}x_{2}x_{3}x_{4}T_{2}T_{3}$, and the value \[P_{w}=1-x_{1}^{2}x_{2}^{2}x_{3}x_{4}x_{5}T_{3}T_{4}\mathrm{\ if\ }w\in\{23145,23154,31245,31254,32145,32154\}.\] Among the remaining permutations $w \in S_{5}$ with $\ell(w)=3$, the ones with the simplest polynomials are $P_{31425}$, which equals
\begin{equation}
1-x_{1}x_{2}x_{3}T_{1}T_{2}-x_{1}x_{2}x_{3}x_{4}T_{1}T_{3}-x_{1}^{2}x_{2}x_{3}x_{4}T_{2}T_{3}+(x_{1}^{2}x_{2}^{2}x_{3}x_{4}+x_{1}^{2}x_{2}x_{3}^{2}x_{4})T_{1}T_{2}T_{3}, \label{P31425}
\end{equation}
and $P_{14253}$, whose value is \[1-x_{1}^{2}x_{2}x_{3}x_{4}(T_{2}T_{3}+x_{5}T_{2}T_{4}+x_{2}x_{5}T_{3}T_{4})+x_{1}^{3}x_{2}^{2}x_{5}(x_{2}^{2}x_{3}+x_{2}x_{3}^{2})T_{2}T_{3}T_{4}.\]

\section{The Quadratic Terms \label{Quadratic}}

Proposition \ref{expPwinT} implies that the first non-trivial terms in the polynomial $P_{w}$ from Definition \ref{Pwdef} is $P_{w,2}$, which is a sum of expressions of the form $m_{\eta}x^{\eta}T_{k}T_{l}$, with a negative sign. In this section we investigate, for each $w$, $k$, and $l$ (with $k \leq l$), which multi-sets $\eta$ show up with a non-zero coefficient $m_{\eta}$, and prove that they are positive and which values can they take.

As usual, we begin with the known value 0 of $P_{w,2}$ for $w=\operatorname{Id}$, assume that $P_{w,2}$ is known, and consider $P_{s_{i}w,2}$ for some $i$ with $\ell(s_{i}w)>\ell(w)$. So fix such $i$ and $w$, and if we denote by $N_{w,i}^{(2)}$ the parts of $N_{w,i}$ that are quadratic in the $T_{l}$'s (or in $t_{1}$), then it follows directly from the proof of Proposition \ref{expPwinT} that
\begin{equation}
P_{s_{i}w,2}(x,t)=\pi_{i}\big(P_{w,2}(x,t)\big)-\pi_{i}\big(N_{w,i}^{(2)}(x,t)\big). \label{Pw2ind}
\end{equation}
In order to do so, recall that the multi-sets $\eta$ from that lemma can only contain indices with multiplicity at most 2. By going over the possible degrees of each monomial in both $i$ and $i+1$, we obtain the following decomposition, in which we again suppress the variables $x$ and $t$ from all arguments.
\[P_{w,2}=P_{w,2}^{00}+x_{i}P_{w,2}^{10}+(x_{i}+x_{i+1})P_{w,2}^{01}+x_{i}^{2}x_{i+1}P_{w,2}^{21}+(x_{i}^{2}x_{i+1}+x_{i}x_{i+1}^{2})P_{w,2}^{12}\]
\begin{equation}
+x_{i}^{2}x_{i+1}^{2}P_{w,2}^{22}+x_{i}^{2}P_{w,2}^{20}+x_{i}x_{i+1}P_{w,2}^{11}+(x_{i}^{2}+x_{i}x_{i+1}+x_{i+1}^{2})P_{w,2}^{02}, \label{decomPw2}
\end{equation}
where all the polynomials $P_{w,2}^{ab}$ depend on neither $x_{i}$ not $x_{i+1}$. Note that $P_{w,2}^{00}$, $P_{w,2}^{01}$, $P_{w,2}^{02}$, $P_{w,2}^{12}$, and $P_{w,2}^{22}$ are indeed based on the terms of $P_{w,2}$ in which $x_{i}$ and $x_{i+1}$ show up with the corresponding exponents, but in $P_{w,2}^{10}$, $P_{w,2}^{21}$, $P_{w,2}^{11}$, and $P_{w,2}^{20}$ they represent the difference between the terms having $x_{i}$ and $x_{i+1}$ with these exponents and some other terms, namely those showing up in $P_{w,2}^{01}$, $P_{w,2}^{12}$, $P_{w,2}^{02}$, and $P_{w,2}^{02}$ respectively.

The main tool for applying the induction step in all of the arguments below, which also forms an explanation for the decomposition in Equation \eqref{decomPw2}, is the following calculation, for which we recall from the proof of Proposition \ref{expPwinT} that $N_{w,i}^{(2)}$ can be written as $x_{i+1}^{2}\tilde{N}_{w,i}^{(2)}$, with the latter multiplier not depending on $x_{i}$ and on $x_{i+1}$.
\begin{lem}
When we decompose $P_{s_{i}w,2}$ as in Equation \eqref{decomPw2}, we have \[P_{s_{i}w,2}^{00}=P_{w,2}^{00},\ P_{s_{i}w,2}^{22}=P_{w,2}^{22},\ P_{s_{i}w,2}^{10}=P_{s_{i}w,2}^{20}=P_{s_{i}w,2}^{21}=0,\] \[P_{s_{i}w,2}^{01}=P_{w,2}^{01}+P_{w,2}^{10},\ P_{s_{i}w,2}^{02}=P_{w,2}^{02}+P_{w,2}^{20},\ P_{s_{i}w,2}^{12}=P_{w,2}^{12}+P_{w,2}^{21},\] and $P_{s_{i}w,2}^{11}=P_{w,2}^{11}+\tilde{N}_{w,i}^{(2)}$. \label{typesind}
\end{lem}

\begin{proof}
When we express $P_{w,2}$ as in Equation \eqref{decomPw2} and apply $\pi_{i}$, the terms involving $P_{w,2}^{00}$, $P_{w,2}^{01}$, $P_{w,2}^{02}$, $P_{w,2}^{12}$, $P_{w,2}^{22}$, and $P_{w,2}^{11}$, which are invariant under $s_{i}$, remain the same after $\pi_{i}$, and for the other three terms we use the fact that $\pi_{i}$ takes $x_{i}$ to $x_{i}+x_{i+1}$, $x_{i}^{2}x_{i+1}$ to $x_{i}^{2}x_{i+1}+x_{i}x_{i+1}^{2}$, and $x_{i}^{2}$ to $x_{i}^{2}+x_{i}x_{i+1}+x_{i+1}^{2}$. Since $\pi_{i}(x_{i+1}^{2})=-x_{i}x_{i+1}$, gathering the expressions as in Equation \eqref{decomPw2} yields the desired result. This proves the lemma.
\end{proof}

Given $w$, $k$, $l$ (with $k \leq l$), and a multi-set $\eta$, we recall from Proposition \ref{expPwinT} that when the term $x^{\eta}T_{k}T_{l}$ shows up in $P_{w,2}$, then the associated multiplicity is denoted by $m^{k,l}_{\eta}(w)$. When $k$ and $l$ are clear from the context, we will omit them from the notation, as we sometimes will also with $w$. Of course, $m^{k,l}_{\eta}(w)=0$ in case the term in question does not appear in $P_{w,2}$.

In order to investigate these multiplicities, we will need some notation. The setting is again that $w$ is a permutation and $i$ satisfies $\ell(s_{i}w)>\ell(w)$. We will write $\eta=\eta^{ab}$ in case $x_{i}$ shows up with multiplicity $a$ in $\eta$, and $x_{i+1}$ is there with multiplicity $b$. In this situation, if $c+d=a+b$ we will write $\eta^{cd}$ for the multi-set resulting from replacing the multiplicities of $x_{i}$ and $x_{i+1}$ by $c$ and $d$ respectively. In particular we have $s_{i}\eta^{ab}=\eta^{ba}$.

We will thus need the result of Lemma \ref{typesind} but for these multiplicities.
\begin{lem}
When passing from $w$ to $s_{i}w$, the multiplicities of the various $\eta^{ab}$'s are given by \[m_{\eta^{00}}(s_{i}w)\!=\!m_{\eta^{00}}(w),\ m_{\eta^{22}}(s_{i}w)\!=\!m_{\eta^{22}}(w),\ m_{\eta^{10}}(s_{i}w)\!=\!m_{\eta^{01}}(s_{i}w)\!=\!m_{\eta^{10}}(w),\] \[m_{\eta^{20}}(s_{i}w)=m_{\eta^{02}}(s_{i}w)=m_{\eta^{20}}(w),\ m_{\eta^{21}}(s_{i}w)=m_{\eta^{12}}(s_{i}w)=m_{\eta^{21}}(w),\] and $m_{\eta^{11}}(s_{i}w)=m_{\eta^{11}}(w)+\big(m_{\eta^{20}}(w)-m_{\eta^{02}}(w)\big)+m_{\eta^{02}}(N_{w,i}^{(2)})$, where the last term is the multiplicity in which $\eta^{02}$ shows up in $N_{w,i}^{(2)}$. \label{multsiw}
\end{lem}

\begin{proof}
In the decomposition from Equation \eqref{decomPw2}, the multiplicities $m_{\eta^{00}}$, $m_{\eta^{01}}$, $m_{\eta^{02}}$, $m_{\eta^{12}}$, and $m_{\eta^{22}}$ (all for $w$) only get contributions from the terms involving $P_{w,2}^{00}$, $P_{w,2}^{01}$, $P_{w,2}^{02}$, $P_{w,2}^{12}$, and $P_{w,2}^{22}$ respectively. Applying Lemma \ref{typesind} and using the same argument for $s_{i}w$ gives the result for $\eta^{00}$ and $\eta^{22}$, and shows that the multiplicities of the others are determined by the expressions $P_{w,2}^{01}+P_{w,2}^{10}$, $P_{w,2}^{02}+P_{w,2}^{20}$, and $P_{w,2}^{12}+P_{w,2}^{21}$ respectively.

However, the respective appearance of the terms $x_{i}$, $x_{i}^{2}$, and $x_{i}^{2}x_{i+1}$ in front of $P_{w,2}^{01}$, $P_{w,2}^{02}$ and $P_{w,2}^{12}$ implies that the corresponding multiplicities also contribute to $m_{\eta^{01}}$, $m_{\eta^{02}}$ and $m_{\eta^{12}}$ respectively, so that the terms involving $P_{w,2}^{10}$, $P_{w,2}^{20}$ and $P_{w,2}^{21}$ contain the expressions associated with $\eta^{10}$, $\eta^{20}$, and $\eta^{21}$ with the respective multiplicities $m_{\eta^{10}}-m_{\eta^{01}}$, $m_{\eta^{20}}-m_{\eta^{02}}$, and $m_{\eta^{12}}-m_{\eta^{21}}$. When these considerations are applied for $s_{i}w$, the fact that the corresponding parts of $P_{s_{i}w,2}$ vanish (by Lemma \ref{typesind}) implies that these differences vanish for $s_{i}w$. Moreover, the multiplicities in the sums from the previous paragraph are thus $m_{\eta^{10}}$, $m_{\eta^{20}}$, and $m_{\eta^{21}}$, yielding the desired expressions for these six multiplicities.

Finally, the multiplier $x_{i}x_{i+1}$ appears in Equation \eqref{decomPw2} in front of $P_{w,2}^{11}$ and $P_{w,2}^{02}$, meaning, by the first paragraph, that the former contains the term associated with $\eta^{11}$ with the multiplicity $m_{\eta^{11}}-m_{\eta^{02}}$. Since Lemma \ref{typesind} shows that for $s_{i}w$ we have to add $\tilde{N}_{w,i}^{(2)}$ to $P_{w,2}^{11}$, with the contribution of $-\pi_{i}(x_{i+1}^{2}\tilde{N}_{w,i}^{(2)})$ to the term in question being $m_{\eta^{02}}(N_{w,i}^{(2)})$ by the proof of that lemma, applying the same considerations with $s_{i}w$ yield \[m_{\eta^{11}}(w)-m_{\eta^{02}}(w)+m_{\eta^{20}}(N_{w,i}^{(2)})\!=\!m_{\eta^{11}}(s_{i}w)-m_{\eta^{02}}(s_{i}w)\!=\!m_{\eta^{11}}(s_{i}w)-m_{\eta^{20}}(w),\]
where we used the already established value of $m_{\eta^{02}}(s_{i}w)$, and the asserted value of $m_{\eta^{11}}(s_{i}w)$ follows as well. This proves the lemma.
\end{proof}

\smallskip

In order to understand the last relation from Lemma \ref{multsiw} better, we prove a lemma, which involves three multi-sets $\eta^{20}$, $\eta^{11}$, and $\eta^{02}$ related as above.
\begin{rmk}
If $\eta^{20}=\eta^{20}_{1}+2\eta^{20}_{2}$, $\eta^{11}=\eta^{11}_{1}+2\eta^{11}_{2}$, and $\eta^{02}=\eta^{02}_{1}+2\eta^{02}_{2}$, then one easily verifies the relations $\eta^{20}_{1}=\eta^{02}_{1}=\eta^{11}_{1}\setminus\{i,i+1\}$, $\eta^{20}_{2}=\eta^{11}_{2}\cup\{i\}$, and $\eta^{02}_{2}=\eta^{11}_{2}\cup\{i+1\}$. \label{inddecom}
\end{rmk}

\begin{lem}
For $w$, $i$, and the multi-sets $\eta^{20}$, $\eta^{11}$, and $\eta^{02}$ related as above, the following are equivalent:
\begin{enumerate}[$(i)$]
\item The multiplicity $m^{k,l}_{\eta^{02}}(N_{w,i}^{(2)})$ is positive, or equivalently non-zero.
\item We have $\eta^{20}=\alpha+\beta$ for $\alpha \in A_{l,i}(w)$ and $\alpha\neq\beta \in A_{k,i}(w)$.
\item $\eta^{20}$ is in $\tilde{B}_{k,l}(w)$ and with $\eta^{20}_{1}$ non-empty, and every presentation of $\eta^{20}$ as a sum as in Definition \ref{Bklwdef} is with elements from $A_{l,i}(w)$ and $A_{k,i}(w)$.
\item We have $\eta^{20}\in\tilde{B}_{k,l}(w)$ with a non-empty $\eta^{20}_{1}$, but $\eta^{11}$, and with it $\eta^{02}$, are not in that set.
\item $\eta^{11} \in B_{k,l}(s_{i}w) \setminus B_{k,l}(w)$.
\end{enumerate} \label{Nwipos}
\end{lem}
Note that the condition $\alpha\neq\beta$ in part $(ii)$ is immediate when $k<l$, but it does make a difference when $k=l$. Of course, the expression $\eta^{20}_{1}$ is defined by the usual formula $\eta^{20}=\eta^{20}_{1}+2\eta^{20}_{2}$, and similarly for $\eta^{02}$ (in which $i\in\eta^{20}_{2}$ is replaced by $i+1$ and the $\eta_{1}$ set remains the same) and $\eta^{11}$ (for which we take $i$ away from $\eta_{2}$ and add it and $i+1$ into $\eta_{1}$).

\begin{proof}
Any contribution to $N_{w,i}^{(2)}$ arises as the product of two distinct terms $x^{s_{i}\alpha}T_{l}$ and $x^{s_{i}\beta}T_{k}$ with $\alpha \in A_{l,i}(w)$ and $\beta \in A_{k,i}(w)$ by Definition \ref{Pwdef} (with a change of notation) and the quadratic assumption. Hence if $(i)$ holds then one such product yields $\eta^{02}$, so that $\eta^{02}=s_{i}\alpha+s_{i}\beta$ with such $\alpha$ and $\beta$, and applying $s_{i}$ to get $\eta^{20}$ establishes $(ii)$. The converse implication is deduced by the same argument. Moreover, the proof of Lemma \ref{typesind} implies that every such contribution to the multiplicity is positive.

Next, if $(ii)$ holds then $\eta^{20}\in\tilde{B}_{k,l}(w)$ by Definition \ref{Bklwdef}, and $\eta^{20}_{1}\neq\emptyset$ since $\beta\neq\alpha$. Now, Lemma \ref{nontrivAli} implies that $w^{-1}(i) \leq k \leq l<w^{-1}(i+1)$, and we have $\alpha_{j}=\beta_{h}=i$ for the indices $j$ and $h$ for which $w^{l}_{j}=w^{k}_{h}=i$. Recalling that $\alpha$ and $\beta$ are ordered and sum to $\eta^{20}$, and comparing the number of integers that are smaller than $i$, we obtain \[|\{g\in\eta_{1}|g<i\}|+2|\{g\in\eta_{2}|g<i\}|=j+h-2\] (note that Remark \ref{inddecom} implies that these sets are the same for $\eta^{20}$, $\eta^{11}$, and $\eta^{02}$, so there is no ambiguity with the notation there). We also have $i\in\alpha\cap\beta$, and since the sum is $\eta^{20}$, we deduce, with the same lack of ambiguity, that
\begin{equation}
|\{g\in\eta_{1}|g>i+1\}|+2|\{g\in\eta_{2}|g>i+1\}|=k+l-j-h. \label{bigentries}
\end{equation}

Consider thus any pair of elements $\tilde{\alpha} \in A_{l}(w)$ and $\tilde{\beta} \in A_{k}(w)$ for which $\tilde{\alpha}+\tilde{\beta}=\eta^{20}$. There are only $l-j$ bounds in $w^{l}$ and only $k-h$ bounds in $w^{k}$ that allow for entries to be larger than $i+1$, and the corresponding entries of $\tilde{\alpha}$ and $\tilde{\beta}$ are the only ones that can give the elements of $\eta$ that are larger than $i+1$. But Equation \eqref{bigentries} implies that we need all of these parts of $\tilde{\alpha}$ and $\tilde{\beta}$ to be larger than $i+1$ in order to give the corresponding part of their sum $\eta^{20}$. This leaves only $j$ entries of $\tilde{\alpha}$ and $h$ entries of $\tilde{\beta}$ to contain the elements of $\eta$ that equal $i$ or less, and since both contain $i$, we must have, due to the ordering, that $\tilde{\alpha}_{j}=\tilde{\beta}_{h}=i$. As we already know that $i+1$ does not show up in a sum that yields $\eta^{20}$, Lemma \ref{nontrivAli} implies that $\tilde{\alpha} \in A_{l,i}(w)$ and $\tilde{\beta} \in A_{k,i}(w)$, and $(iii)$ follows. The fact that $(iii)$ implies $(ii)$ is immediate.

Now, if $(iii)$ holds then the first condition in $(iv)$ also does (by definition). Write $\eta^{02}$ as a sum of two elements, one of size $l$ and one of size $k$. The $\eta^{02}$ condition then allows us to write $\eta^{02}=s_{i}\alpha+s_{i}\beta$ for $\alpha$ and $\beta$ containing $i$ but not $i+1$. Since $s_{i}\alpha \in A_{l}(w)$ implies $\alpha \in A_{l}(w)$ and $s_{i}\beta \in A_{k}(w)$ yields $\beta \in A_{k}(w)$, and we then have $\eta^{20}=\alpha+\beta$, this implies $\alpha \in A_{l,i}(w)$ and $\beta \in A_{k,i}(w)$ by $(iii)$. But then $s_{i}\alpha \not\in A_{l}(w)$ and $s_{i}\beta \not\in A_{k}(w)$ (by the definition in Equation \eqref{Alwdecom}), which proves that $\eta^{02}\not\in\tilde{B}_{k,l}$. The same proof shows that if we write $\eta^{11}$ as a sum of an element of length $l$ and one of length $k$ such that $i$ lies in one of them and $i+1$ is in the other, then the second summand is not in $A(w)$.

Write now $\eta^{11}=\alpha+\beta$ such that one of them contains both $i$ and $i+1$. If $\beta \in A_{k}(w)$ then the only entries of $\beta$ that can be larger than $i$ have indices larger than $h$, and when $\alpha \in A_{l}(w)$ those of $\alpha$ that may be larger than $i$ must be larger than $j$, and there are in total $k+l-h-j$ such entries. But these entries must then give, to the sum $\eta^{11}$, all the elements from Equation \eqref{bigentries}, and the occurrence of $i+1$ in $\eta^{11}$ together. As this is impossible, we deduce that $\eta^{11}\not\in\tilde{B}_{k,l}$ s well, which proves $(iv)$.

Conversely, if $(iv)$ holds then the first condition of $(iii)$ follows directly, and we can write $\eta^{20}=\alpha+\beta$ with $\alpha \in A_{l}(w)$ and $\beta \in A_{k}(w)$. In case $\alpha \not\in A_{l,i}(w)$ (resp. $\beta \not\in A_{k,i}(w)$), we deduce from Equation \eqref{Alwdecom} that $s_{i}\alpha \in A_{l}(w)$ (resp. $s_{i}\beta \in A_{k}(w)$), and the sum $s_{i}\alpha+\beta$ (resp. $\alpha+s_{i}\beta$) is a presentation of $\eta^{11}$ as an element of $\tilde{B}_{k,l}(w)$, contradicting $(iv)$. This contradiction yields the desired converse implication.

Assuming $(iv)$ again, write $\eta^{20}=\alpha+\beta$ with $\alpha \in A_{l}(w)$ and $\beta \in A_{k}(w)$ once more, and by applying $s_{i}$ to either one of $\alpha$ and $\beta$ we get a presentation of $\eta^{11}$ as the sum of two elements from $A(s_{i}w)$ (Proposition \ref{Asiw}) and get $\eta^{11}\in\tilde{B}_{k,l}(s_{i}w)$. But from the $\tilde{B}_{k,l}$ condition on $\eta^{20}$ we get $|\eta^{20}_{2}| \leq k$ (by Lemma \ref{conteta}), and the condition on $\eta^{20}_{1}$ implies that if $k=l$ then the latter inequality is strict. Remark \ref{inddecom} translates these inequalities to $|\eta^{11}_{2}|<k-\delta_{k,l}$, so that $\eta^{11} \in B_{k,l}(s_{i}w)$ by Proposition \ref{Bkldiff}. Since the condition on $\eta^{11}$ in $(iv)$ excludes it from being in $B_{k,l}(w)$, this proves that $(iv)$ implies $(v)$.

Finally, $(v)$ implies, by Proposition \ref{Bkldiff}, that $\eta^{11}\in\tilde{B}_{k,l}(s_{i}w)$ and also $|\eta^{11}_{2}|<k-\delta_{k,l}$, and since this inequality holds and we assume that $\eta^{11} \not\in B_{k,l}(w)$, we get $\eta^{11}\not\in\tilde{B}_{k,l}(w)$ by the same proposition. This expresses $\eta^{11}$ as a sum $\alpha+\beta$ of elements from $A(s_{i}w)$, with at least one of them not being in $A(w)$. But then Proposition \ref{Asiw} implies that this summand is the $s_{i}$-image of an element of $A_{l,i}(w)$ or $A_{k,i}(w)$, so by Lemma \ref{nontrivAli} it contains $i+1$ but not $i$. Hence the other summand contains $i$ (by the $\eta^{11}$ sum condition) and thus lies in $A(w)$ (Lemma \ref{nontrivAli} again), and replacing $i+1$ by $i$ gives a sum of $\eta^{20}$ that puts it in $\tilde{B}_{k,l}(w)$. As in the previous paragraph, Remark \ref{inddecom} takes the inequality $|\eta^{11}_{2}|<k-\delta_{k,l}$ to $\eta_{20}^{1}$ being non-empty, which establishes the parts of $(iv)$ that imply $(iii)$ and hence the remaining statement about $\eta^{02}$. This completes the proof of the lemma.
\end{proof}

\smallskip

Our first result states when non-zero multiplicities may show up.
\begin{thm}
For $k$, $l$, $w$, and $\eta$ with the usual properties, we can have $m^{k,l}_{\eta}(w)\neq0$ only when $\eta \in B_{k,l}(w)$. \label{mposBkl}
\end{thm}

\begin{proof}
We fix $k$ and $l$, and work by the usual induction on $w$. When $w=\operatorname{Id}$ we have $P_{\operatorname{Id},2}=0$, so that the assertion holds in an empty manner. We thus take some permutation $w$ for which we know that the result holds, choose some $i$ for which $\ell(s_{i}w)>\ell(w)$, and consider a multi-set $\eta$ for which $m_{\eta}(s_{i}w)\neq0$.

When $\eta$ is $\eta^{00}$, $\eta^{10}$, $\eta^{20}$, $\eta^{21}$, or $\eta^{22}$, Lemma \ref{multsiw} yields $m_{\eta}(s_{i}w)=m_{\eta}(w)$, so that the induction hypothesis implies that $\eta$ is in $B_{k,l}(w)$ and therefore also in $B_{k,l}(s_{i}w)$ by Proposition \ref{Bsiw}. Next, take $\eta$ which is either $\eta^{01}$, $\eta^{02}$, or $\eta^{12}$, so that Lemma \ref{multsiw} gives $m_{\eta}(s_{i}w)=m_{s_{i}\eta}(w)$. Then the induction hypothesis implies that $s_{i}\eta \in B_{k,l}(w)$, and Proposition \ref{Bsiw} once again gives $\eta \in B_{k,l}(s_{i}w)$.

It remains to consider the case where $\eta=\eta^{11}$, where the non-zero expression $m_{\eta}(s_{i}w)$ is given by Lemma \ref{multsiw}, and prove that $\eta^{11} \in B_{k,l}(s_{i}w)$. If $m_{\eta^{11}}(w)\neq0$ then we are done as in the previous paragraph, and Lemma \ref{Nwipos} gives the result in case the last term is non-zero. When $m_{\eta^{02}}(w)$ or $m_{\eta^{20}}(w)$ are non-zero, the induction hypothesis puts $\eta^{02}$ or $\eta^{20}$ in $B_{k,l}$, so that in particular $|\eta^{02}_{2}|$ and $|\eta^{20}_{2}|$ are smaller than $k-\delta_{k,l}$ by Proposition \ref{Bkldiff}. Since Remark \ref{inddecom} implies that the same holds for $|\eta^{11}_{2}|$, it suffices to show that $\eta^{11}\in\tilde{B}_{k,l}(s_{i}w)$ if either $\eta^{02}$ or $\eta^{20}$ are in $\tilde{B}_{k,l}(w)$.

Now, if $\eta^{02}\in\tilde{B}_{k,l}(w)$ then the proof of $(iii)$ implying $(iv)$ in Lemma \ref{Nwipos} implies, \emph{mutatis mutandis}, that $\eta^{11}\in\tilde{B}_{k,l}(w)$ as desired. Finally, in the case where $\eta^{20}\in\tilde{B}_{k,l}(w)$, we write it as a sum $\alpha+\beta$ from $A_{l}(w)$ and $A_{k}(w)$, apply $s_{i}$ to either $\alpha$ or $\beta$, and obtain, via Propositions \ref{Asiw} and \ref{Bsiw}, the asserted claim. This proves the theorem.
\end{proof}
Note that when $w=s_{i}$, the set $A(w)$ is the union of $A(\operatorname{Id})$ and a single element, so that the set $B_{k,l}(s_{i})$ is empty for every $k$ and $l$. This is in correspondence, via Theorem \ref{mposBkl}, with the fact that $P_{s_{i}}=1$, hence $P_{s_{i},2}=0$. We get the same polynomials for $w=s_{i}s_{j}$ when $|i-j|\geq2$, which matches with the fact that also for this $w$ all the $B_{k,l}$ sets are empty. One can also verify that all the $P_{w,2}$ terms in all the expressions around Equation \eqref{P31425} are indeed of the form $x^{\eta}T_{k}T_{l}$ for $\eta \in B_{k,l}(w)$.

We draw an immediate consequence, simplifying Lemma \ref{multsiw}.
\begin{cor}
Given $k$, $l$, $w$, and $i$ as above, for $\eta=\eta^{11}$ the multiplicity $m_{\eta^{11}}(s_{i}w)$ equals $m_{\eta^{20}}(w)+m_{\eta^{02}}(N_{w,i}^{(2)})$ in case $\eta^{11} \in B_{k,l}(s_{i}w) \setminus B_{k,l}(w)$, and $m_{\eta^{11}}(w)+\big(m_{\eta^{20}}(w)-m_{\eta^{02}}(w)\big)$ otherwise. \label{meta11siw}
\end{cor}

\begin{proof}
Lemma \ref{Nwipos} implies that unless the condition $\eta^{11} \in B_{k,l}(s_{i}w) \setminus B_{k,l}(w)$ is satisfied, we have $m_{\eta^{02}}(N_{w,i}^{(2)})=0$ and this summand may be removed from the expression in Lemma \ref{multsiw}. When this condition is satisfied, the we also get from Lemma \ref{multsiw} that $\eta^{02} \not\in B_{k,l}(w)$ as well, and since Theorem \ref{mposBkl} gives $m_{\eta^{11}}(w)=m_{\eta^{02}}(w)=0$ the formula from Lemma \ref{multsiw} reduces to the desired expression. This proves the corollary.
\end{proof}

\smallskip

For a sequence $\beta$, write $\beta_{<i}$, $\beta_{\leq i}$, $\beta_{>i+1}$, and $\beta_{\geq i+1}$ for the set of elements of $\beta$ that satisfy the inequality in the subscript. Using this notation, we prove two additional lemmas, which will be helpful for establishing properties of the multiplicities in the next section.
\begin{lem}
Take $w$, $i$, and $\eta=\eta^{11}$ as above, with the related element $\eta^{20}$, and assume that $\eta\in\tilde{B}_{k,l}(w)$. Then the following are equivalent:
\begin{enumerate}[$(i)$]
\item The sequence $\eta^{20}$ is not in $\tilde{B}_{k,l}(w)$.
\item In every presentation $\eta=\alpha+\beta$ with $\alpha \in A_{l}(w)$ and $\beta \in A_{k}(w)$, both $i$ and $i+1$ are in $\alpha$.
\item Neither $\beta_{\min}$ nor $\beta_{\max}$, for our $\eta$, contain $i$ or $i+1$, and we have the equalities $|\beta_{\min,<i}|=|\beta_{\max,<i}|$ and $|\beta_{\min,>i+1}|=|\beta_{\max,>i+1}|$ as well as $|\alpha_{\min,<i}|=|\alpha_{\max,<i}|$ and $|\alpha_{\min,\geq i+1}|=|\alpha_{\max,\geq i+1}|$.
\end{enumerate} \label{eta20Bkl}
\end{lem}

\begin{proof}
First we prove that if the converse of $(ii)$ holds, then there is an expression $\eta=\alpha+\beta$ with $\alpha \in A_{l}(w)$ and $\beta \in A_{k}(w)$ such that $i$ and $i+1$ are not in the same summand. Indeed, there is, by assumption, such a sum in which $\beta$ contains either $i$ or $i+1$. We need to show that if it contains both then we can replace $\alpha$ and $\beta$ by $\tilde{\alpha} \in A_{l}(w)$ and $\tilde{\beta} \in A_{k}(w)$ with $\tilde{\alpha}+\tilde{\beta}=\eta$ such that the property involving $i$ and $i+1$ holds.

Indeed, recall from Proposition \ref{contmax} that neither $i$ nor $i+1$ lie in both $\beta_{\min}$ or $\beta_{\max}$ (since they are in $\eta_{1}$), and Corollary \ref{setofpairs} implies that $\beta_{\min}\leq\beta\leq\beta_{\max}$. Thus, there is an operation of the sort $R_{j,r}$ from Definition \ref{replace} that either replaces $i+1$ by a larger number from $\eta_{1}$ but still with $R_{j,r}(\beta)\leq\beta_{\max}$, or puts a smaller number from $\eta_{1}$ in the place of $i$ and still satisfies $R_{j,r}(\beta)\geq\beta_{\min}$. Since $\tilde{\beta}:=R_{j,r}(\beta)$ will also be in $A_{k}^{\eta}(w)$ and with $\beta_{\min}\leq\tilde{\beta}\leq\beta_{\max}$, Corollary \ref{setofpairs} shows that we have $\eta=\tilde{\alpha}+\tilde{\beta}$ for the appropriate $\tilde{\alpha} \in A_{l}(w)$. Since by construction $\tilde{\beta}$ contains precisely one of $i$ and $i+1$, so does $\tilde{\alpha}$, and our claim follows.

Now, if $(ii)$ does not hold then write $\eta=\alpha+\beta$ with $\alpha$ and $\beta$ as in the claim, and replacing $i+1$ by $i$ in the corresponding summand then implies that $\eta^{20}\in\tilde{B}_{k,l}(w)$, hence contradicting $(i)$. Conversely, if $(i)$ does not hold, then there are $\alpha \in A_{l}(w)$ and $\beta \in A_{k}(w)$ such that $\eta^{20}=\alpha+\beta$, with $i$ lying in both. If $\alpha \in A_{l,i}(w)$ and $\beta \in A_{l,i}(w)$, then Lemma \ref{Nwipos} implies that $\eta\not\in\tilde{B}_{k,l}(w)$, contradicting our assumption (the condition $\alpha\neq\beta$ there was only relevant for $\eta^{20}_{1}$ to be non-empty). Thus either $s_{i}\alpha$ or $s_{i}\beta$ are in $A(w)$, and we get a presentation of $\eta$ that contradicts $(ii)$. Hence $(i)$ and $(ii)$ are equivalent.

Next, assume $(iii)$, and write $\eta=\alpha+\beta$ with $\alpha \in A_{l}(w)$ and $\beta \in A_{k}(w)$. From $\beta\leq\beta_{\max}$ we deduce that $|\beta_{<i}|\geq|\beta_{\max,<i}|$ while the inequality $\beta\geq\beta_{\min}$ implies that $|\beta_{>i+1}|\geq|\beta_{\min,>i+1}|$. Since $\beta_{<i}$ and $\beta_{>i+1}$ are disjoint subsets of $\beta$, their sizes sum to at most $k$, with the sum being $k$ if and only if $i$ and $i+1$ are not in $\beta$. But if $\beta$ is $\beta_{\max}$ or $\beta_{\min}$ then our two inequalities are equalities (by our assumption), and they sum to $k$ by what we just saw. Hence $|\beta_{<i}|$ and $|\beta_{>i+1}|$ sum to $k$ for every such $\beta$, meaning that in every presentation as in $(ii)$ neither $i$ nor $i+1$ can be in $\beta$ and they are thus in $\alpha$. This proves that $(iii)$ implies $(ii)$.

Conversely, since $\beta_{\min} \in A_{k}(w)$ and $\alpha_{\min} \in A_{l}(w)$ by Lemma \ref{relsminBkl}, the definition of these elements in Definition \ref{minmaxelts} implies that if $\beta_{\min}$ or $\beta_{\max}$ contains either $i$ or $i+1$ but not both then $(ii)$ does not hold. Moreover, the proof of our claim shows in particular that if one of these sets contains both $i$ and $i+1$ then again $(i)$ and $(ii)$ do not hold. Moreover, since when $i$ and $i+1$ are neither in $\beta_{\min}$ nor in $\beta_{\max}$ the sums $|\beta_{\min,>i+1}|+|\alpha_{\max,\geq i+1}|$ and $|\beta_{\max,>i+1}|+|\alpha_{\min,\geq i+1}|$ both give the size from Equation \eqref{bigentries}, and $|\beta_{\min,<i}|+|\alpha_{\max,<i}|$ and $|\beta_{\max,<i}|+|\alpha_{\min,<i}|$ produce the size from the equation preceding it, it suffices to establish the equalities involving the $\beta$'s.

So assume that $\beta_{\min}$ nor $\beta_{\max}$ contain neither $i$ nor $i+1$, so that the sums $|\beta_{\min,<i}|+|\beta_{\min,>i+1}|$ and $|\beta_{\max,<i}|+|\beta_{\max,>i+1}|$ both equal $k$ as above, and we always have the inequalities $|\beta_{\min,<i}|\geq|\beta_{\max,<i}|$ and $|\beta_{\max,>i+1}|\geq|\beta_{\min,>i+1}|$, which are strict together. But when they are strict, we can replace the maximal entry of $\beta_{\min}\cap\eta_{1}$ that is smaller than $i$, or the minimal one of $\beta_{\max}\cap\eta_{1}$ that is larger than $i+1$, by either $i$ or $i+1$ using the operator $R_{j,i}$ or $R_{j,i+1}$ for the appropriate $j$, and obtain an element of $A_{k}^{\eta}(w)$ that lies between $\beta_{\min}$ and $\beta_{\max}$ and one of $i$ and $i+1$, meaning that $(ii)$ again does not hold. Therefore when $(ii)$ holds, $(iii)$ does too. This completes the proof of the lemma.
\end{proof}
Condition $(iii)$ of Lemma \ref{eta20Bkl} also gives the equalities
$|\alpha_{\min,\leq i}|=|\alpha_{\max,\leq i}|$ and $|\alpha_{\min,<i+1}|=|\alpha_{\max,<i+1}|$, because condition $(ii)$ there implies that $i$ and $i+1$ are in both $\alpha_{\min}$ and $\alpha_{\max}$. For the $\beta$'s changing to the non-strict inequalities does not affect the sets, since neither $i$ nor $i+1$ show up in these sets.

\begin{lem}
With the usual $w$, $i$, and $\eta=\eta^{11}$, which we assume to be in $\tilde{B}_{k,l}(w)$, consider now the related element and $\eta^{02}$, and the following are equivalent:
\begin{enumerate}[$(i)$]
\item The element $\eta^{02}$ does not lie in $\tilde{B}_{k,l}(w)$.
\item If $\alpha \in A_{l}(w)$ and $\beta \in A_{k}(w)$ satisfy $\eta=\alpha+\beta$ and $i$ and $i+1$ are in different summands, then the summand containing $i$ is in $A_{l,i}(w)$ or in $A_{k,i}(w)$.
\item We have $|\beta_{\min,\leq i}|=|\beta_{\max,\leq i}|$ and $|\beta_{\min,\leq i+1}|=|\beta_{\max,\leq i+1}|$, as well as the equalities $|\alpha_{\min,\leq i}|=|\alpha_{\max,\leq i}|$ and $|\alpha_{\min,\geq i+1}|=|\alpha_{\max,\geq i+1}|$.
\end{enumerate} \label{eta02Bkl}
\end{lem}

\begin{proof}
When $(i)$ does not occur, we have $\eta^{02}$ as a sum of an element from $A_{l}(w)$ and one from $A_{k}(w)$, both of which containing $i+1$ and not containing $i$. Replacing $i+1$ by $i$ in any of them produces a presentation $\eta=\alpha+\beta$ from $A_{l}(w)$ and $A_{k}(w)$, with $i+1$ being in one summand and $i$ in the other, and the latter is in $A_{l}^{(i)}(w)$ or $A_{k}^{(i)}(w)$ by our assumption and the definition in Equation \eqref{Alwdecom}. But this contradicts $(ii)$, so that $(ii)$ implies $(i)$. Conversely, if $(ii)$ does not hold then we can write $\eta=\alpha+\beta$ with $i$ and $i+1$ in different summands, such that $\alpha \in A_{l}(w)$, $\beta \in A_{k}(w)$, and the $s_{i}$-image of the summand containing $i$ also lying in $A(w)$. But replacing this summand by its $s_{i}$-image yields an equation implying that $\eta^{20}\in\tilde{B}_{k,l}(w)$, contradicting $(i)$. conditions $(i)$ and $(ii)$ are thus equivalent.

Now, as in the proof of Lemma \ref{eta20Bkl}, the inequalities $\beta_{\min}\leq\beta\leq\beta_{\max}$ arising as usual wherever $\eta=\alpha+\beta$ with $\alpha \in A_{l}(w)$ and $\beta \in A_{k}(w)$ imply that $|\beta_{\min,\leq i}|\geq|\beta_{\leq i}|\geq|\beta_{\max,\leq i}|$ and $|\beta_{\min,\geq i+1}|\geq|\beta_{\geq i+1}|\geq|\beta_{\min,\geq i+1}|$, and for these sets we have $|\beta_{\leq i}|+|\beta_{\geq i+1}|=k$ for every such $\beta$. Assuming $(iii)$, we get that the values of $|\beta_{\leq i}|$ and $|\beta_{\geq i+1}|$ are the same for every such $\beta$, and consider a sum $\eta=\alpha+\beta$ for which $i$ and $i+1$ are in different summands. Since $s_{i}\eta=\eta$, we get $\eta=s_{i}\alpha+s_{i}\beta$, and we deduce that if $i\in\beta$ and $\beta \in A_{k}^{(i)}$, or if $i\in\alpha$ and $\alpha \in A_{l}^{(i)}$, then $s_{i}\alpha \in A_{l}(w)$ and $s_{i}\beta \in A_{k}(w)$. But in the former case going from $\beta$ to $s_{i}\beta$ transfers one element from $\beta_{\leq i}$ to $\beta_{\geq i+1}$ and in the other an element is taken in the other direction, so that the equalities cannot hold. Therefore $(iii)$ implies $(ii)$.

Finally, we recall from the proof of Lemma \ref{eta20Bkl} that $|\beta_{\min,<i}|\geq|\beta_{\max,<i}|$ and $|\beta_{\max,>i+1}|\geq|\beta_{\min,>i+1}|$ and that the inequalities are strict together, and assume, by contradiction to $(iii)$, that they are indeed strict. Note that if $i+1\in\beta_{\min}$, this strict inequality implies that the corresponding entry of $\beta_{\max}$ is larger than $i+1$ (since there is an entry of $\beta_{\max}$ that equals at least $i+1$ and shows up already before). Lemma \ref{interup} produces an element $\tilde{\beta}$ with the same sizes of the sets in the partition as $\beta_{\min}$ and which does not contain $i+1$, and if $i+1$ was not in $\beta_{\min}$ to begin with, we set $\tilde{\beta}$ to be just $\beta_{\min}$. Moreover, if $i\not\in\tilde{\beta}$ then once again the appropriate operator $R_{j,i}$ replaces the maximal entry of $\beta_{\min}\cap\eta_{1}$ that is smaller than $i+1$ by $i$ (again without changing the sizes of the sets in the partition), and we let $\beta$ be the resulting element, or just $\beta=\tilde{\beta}$ in case $i$ was already there.

Since $\beta \in A_{k}^{\eta}(w)$ and satisfies $\beta\geq\beta_{\min}$, Corollary \ref{setofpairs} implies that $\eta=\alpha+\beta$ for $\alpha \in A_{l}(w)$, with $\alpha$ containing $i+1$ and not $i$ by the fact that $\eta=\eta^{11}$ and how we constructed $\beta$. Moreover, the assumption on the sizes of the sets implies that entry of $\beta_{\max}$ that is in the same location as $i$ in $\beta$ is strictly larger than $i$, which means that $s_{i}\beta \in A_{k}(w)$ as well, contradicting $(ii)$. An argument involving complements to $\eta$, as in the proof of Lemma \ref{eta20Bkl}, shows that the equalities for the $\beta$'s in $(iii)$ are equivalent to those for the $\alpha$'s. This proves that $(ii)$ implies $(iii)$, and with it, the lemma.
\end{proof}
A simple application of $s_{i}$, which reduces each $i+1$ to $i$, shows that if $\eta^{20}\in\tilde{B}_{k,l}(w)$ then $\eta^{02}\in\tilde{B}_{k,l}(w)$ (as was also mentioned in the proof of Lemma \ref{Nwipos}). Hence condition $(i)$ of Lemma \ref{eta20Bkl} is stronger than condition $(i)$ of Lemma \ref{eta02Bkl}. To see this with the other conditions, note that if condition $(ii)$ of Lemma \ref{eta20Bkl} holds then so does that of Lemma \ref{eta02Bkl}, as an empty statement. Similarly, when $i$ and $i+1$ are neither in $\beta_{\min}$ nor in $\beta_{\max}$, the sets with the strict inequalities coincide with that having the non-strict ones, and indeed condition $(iii)$ of Lemma \ref{eta20Bkl} implies condition $(iii)$ of Lemma \ref{eta02Bkl}.

\section{Some Explicit Values of Multiplicities \label{Mults}}

Fix $k$ and $l$ with $k \leq l$, and consider a multi-set $\eta=\eta_{1}+2\eta_{2}$, for which we assume as always that $|\eta_{1}|+2|\eta_{2}|=k+l$. We are interested in the multiplicities $m^{k,l}_{\eta}(w)$, as defined via Proposition \ref{expPwinT}, for various $w$. Lemma \ref{conteta} implies that if $\eta$ is the sum of some $\alpha \in A_{l}(w)$ and some $\beta \in A_{k}(w)$ for any $w$ then $|\eta_{2}| \leq k$. Since Theorem \ref{mposBkl} implies that non-zero multiplicities can be found only in elements of the sets $B_{k,l}(w)$ from Definition \ref{Bklwdef}, the latter argument shows that when $|\eta_{2}|>k$ we have $m^{k,l}_{\eta}(w)=0$ for all $w$. Moreover, Proposition \ref{Bkldiff} allows us to restrict attention to the case where $|\eta_{2}|<k-\delta_{k,l}$.

It indeed turns out that the resulting parameter $r:=k-\delta_{k,l}-|\eta_{2}|$, which needs to be positive for giving non-zero multiplicities, plays a more important role in determining the behavior of these multiplicities. For investigating the behavior of these multiplicities using our inductive argument, we shall now check how the sets from Definition \ref{setswitheta} are affected during the induction step.
\begin{lem}
Consider a permutation $w$, and index $i$ with $\ell(s_{i}w)>\ell(w)$, and a multi-set $\eta$ as usual.
\begin{enumerate}[$(i)$]
\item If $\eta$ is $\eta^{00}$, $\eta^{10}$, $\eta^{20}$, $\eta^{21}$, or $\eta^{22}$, then $A_{l}^{\eta}(s_{i}w)=A_{l}^{\eta}(w)$ for every $l$. In particular we have $\alpha_{\max}^{l,s_{i}w}(\eta)=\alpha_{\max}^{l,w}(\eta)$ and $\beta_{\min}^{k,s_{i}w}(\eta)=\beta_{\min}^{k,w}(\eta)$ when this set is not empty and $|\eta_{1}|+2|\eta_{2}|=k+l$.
\item In case $\eta$ is $\eta^{01}$, $\eta^{02}$, or $\eta^{12}$, we have $A_{l}^{\eta}(s_{i}w)=\{s_{i}\alpha|\alpha \in A_{l}^{s_{i}\eta}(w)\}$ for any $l$. Therefore if these sets are not empty and $|\eta_{1}|+2|\eta_{2}|=k+l$ then $\alpha_{\max}^{l,s_{i}w}(\eta)=s_{i}\alpha_{\max}^{l,w}(s_{i}\eta)$ and $\beta_{\min}^{k,s_{i}w}(\eta)=s_{i}\beta_{\min}^{k,w}(s_{i}\eta)$.
\item Moving from $\eta$ to $s_{i}\eta$ does not affect $\eta_{1}$ when $\eta=\eta^{02}$, replaces $i+1$ by $i$ in the same position in case $\eta=\eta^{01}$, and changes $i$ to $i+1$ in the same location if $\eta=\eta^{12}$.
\item When $\eta=\eta^{11}$ and $\alpha_{\max}^{l,w}(\eta) \in A_{l}^{(i)}(w)$ we have $A_{l}^{\eta}(s_{i}w)=A_{l}^{\eta}(w)$, and in particular $\alpha_{\max}^{l,s_{i}w}(\eta)=\alpha_{\max}^{l,w}(\eta)$, as well as $\beta_{\min}^{k,s_{i}w}(\eta)=\beta_{\min}^{k,w}(\eta)$ when the usual equality $|\eta_{1}|+2|\eta_{2}|=k+l$ holds. In this case if $i\in\alpha_{\max}^{l,w}(\eta)$ then $i+1$ is also there.
\item If $\eta=\eta^{11}$ and $\alpha_{\max}^{l,w}(\eta) \in A_{l,i}(w)$, then $\alpha_{\max}^{l,s_{i}w}(\eta)=s_{i}\alpha_{\max}^{l,w}(\eta)>\alpha_{\max}^{l,w}(\eta)$.
\end{enumerate} \label{siwsets}
\end{lem}
Note that the statements about the minimal elements in Lemma \ref{siwsets} do not assume that $k \leq l$, and therefore hold also when we interchange the roles of $\alpha$ and $l$ with those of $\beta$ and $k$ to get this inequality when required.

\begin{proof}
For part $(i)$, note that $A_{l}^{\eta}(w) \subseteq A_{l}^{\eta}(s_{i}w)$, so that if the latter is empty then so is the former and the result holds trivially. Assuming now that it is not empty, recall from Proposition \ref{Asiw} that $A_{l}(s_{i}w)$ is the union of $A_{l}(w)$ and $\{s_{i}(\alpha)|\alpha \in A_{l,i}(w)\}$, and that elements of the latter set contain $i+1$ but not $i$ (by Lemma \ref{nontrivAli} and applying $s_{i}$). But if $\alpha \in A_{l}^{\eta}(s_{i}w)$ then the type of $\eta$ implies that if $i+1$ is in that set then so is $i$. Hence such $\alpha$ is not an $s_{i}$-image of an element from $A_{l,i}(w)$, and we thus have $\alpha \in A_{l}(w)$ and hence $\alpha \in A_{l}^{\eta}(w)$, yielding the reverse inclusion as well. Considering the maximal element in both sets yields the result about $\alpha_{\max}$, and the one for $\beta_{\min}$ now follows directly from Definition \ref{minmaxelts}. Part $(i)$ is thus established.

Turning to part $(ii)$, Proposition \ref{Asiw} implies that $s_{i}$ preserves $A(s_{i}w)$, and it is thus clear from Definition \ref{setswitheta} that $A_{l}^{\eta}(s_{i}w)=\{s_{i}\alpha|\alpha \in A_{l}^{s_{i}\eta}(s_{i}w)\}$. But if $\eta$ is of one of the types from part $(ii)$ then the type of $s_{i}\eta$ shows up in part $(i)$, and by applying this part we can replace $A_{l}^{s_{i}\eta}(s_{i}w)$ by $A_{l}^{s_{i}\eta}(w)$ and obtain the assertion about the sets. Now, when $\eta$ is $\eta^{01}$ and $\eta^{02}$, $i$ does not show up and any occurrence of $i+1$ satisfies the same inequalities as $i$ there, so that $s_{i}$ preserves the orders when it acts on $A_{l}^{s_{i}\eta}(w)$ and $A_{l}^{\eta}(s_{i}w)$, in both directions. If $\eta=\eta^{12}$ then the same assertion is clear when both $i$ and $i+1$ are present or when only $i+1$ shows up in the two compared sequences (or only $i$ appears, when considering $s_{i}\eta=\eta^{12}$). It is also clear in case $\alpha$ contains only $i+1$ (thus only $i$ is in $s_{i}\alpha$) while $\tilde{\alpha}$ contains both $i$ and $i+1$ if $\alpha\leq\tilde{\alpha}$ (resp. $s_{i}\alpha\geq\tilde{\alpha}$), since $s_{i}\alpha\leq\alpha$ and $\tilde{\alpha}$ is $s_{i}$-invariant. When $\alpha\geq\tilde{\alpha}$, the position of $i$ in $\tilde{\alpha}$ must be compared with an entry of $\alpha$ that is at least $i+1$ (and similarly in case $s_{i}\alpha\leq\tilde{\alpha}$, the position of $i+1$ in $\tilde{\alpha}$ is compared with one from $s_{i}\alpha$ that is at most $i$), so that $s_{i}\alpha\geq\tilde{\alpha}$ (resp. $\alpha\leq\tilde{\alpha}$) as well. Therefore considering the maximal elements yields the formula for $\alpha_{\max}$ also here, and with Definition \ref{minmaxelts} the one about $\beta_{\min}$ again follows. This proves part $(ii)$.

Next, when $\eta=\eta^{02}$ neither $i$ nor $i+1$ lie in $\eta_{1}$, while if $\eta=\eta^{01}$ then $i+1\in\eta_{1}$ with $i$ not being in that set, and in case $\eta=\eta^{12}$ we have $i\in\eta_{1}$ but $i+1$ is not there. Part $(iii)$ now easily follows.

Now, when $\eta=\eta^{11}$ the $\eta$-condition from Definition \ref{setswitheta} is $s_{i}$-invariant, and Proposition \ref{Asiw} presents $A_{l}(s_{i}w)$ as the union of $A_{l}(w)$ and its $s_{i}$-images. Therefore the set $A_{l}^{\eta}(s_{i}w)$ contains the elements of $A_{l}^{\eta}(w)$ and their $s_{i}$-images, and in particular $\alpha_{\max}=\alpha_{\max}^{l,w}(\eta)$ and $s_{i}\alpha_{\max}$ are there. If $\alpha_{\max} \in A_{l}^{(i)}(w)$ then its $s_{i}$-image is also in $A_{l}^{\eta}(w)$, hence we have $s_{i}\alpha_{\max}\leq\alpha_{\max}$, meaning that $\alpha_{\max}$ either contains both $i$ and $i+1$, or contains neither, or contains only $i+1$ and not $i$ (this proves the last assertion in part $(iv)$). When $\alpha_{\max} \in A_{l,i}(w)$, Lemma \ref{nontrivAli} implies that $s_{i}\alpha_{\max}$ contains $i+1$ and not $i$, and we have $s_{i}\alpha_{\max}>\alpha_{\max}$.

Given any element $\alpha \in A_{l}^{\eta}(s_{i}w)$, if $\alpha \in A_{l}^{\eta}(w)$ then $\alpha\leq\alpha_{\max}$ by Proposition \ref{maxelt}, and otherwise $\alpha=s_{i}\tilde{\alpha}$ with $\tilde{\alpha} \in A_{l}^{\eta}(w)$ and $\tilde{\alpha}\leq\alpha_{\max}$. When $\alpha_{\max}$ contains neither $i$ nor $i+1$, this bound remains invariant under $s_{i}$ and we again have $\alpha\leq\alpha_{\max}$, and if $\alpha_{\max}$ contains only $i+1$ then again $\tilde{\alpha}\leq\alpha_{\max}$ implies $\alpha=s_{i}\tilde{\alpha}\leq\alpha_{\max}$ because the bounds from $\alpha_{\max}$ involve $i+1$ and not $i$. If both $i$ and $i+1$ are in $\alpha_{\max}$ then the proof of part $(ii)$ shows that again $\tilde{\alpha}\leq\alpha_{\max}$ implies $\alpha\leq\alpha_{\max}$, and since $\alpha_{\max} \in A_{l}(w)$ we get that $\alpha \in A_{l}(w)$ as well. Hence when $\alpha_{\max} \in A_{l}^{(i)}(w)$ we get the inclusion $A_{l}^{\eta}(s_{i}w) \subseteq A_{l}^{\eta}(w)$ as well, yielding the equality of sets from part $(iv)$ and with it the equality of the maximal elements hence also of the minimal ones via Definition \ref{minmaxelts}.

When $\alpha_{\max}^{l,w}(\eta) \in A_{l,i}(w)$ we combine all the inequalities involving $\alpha$ or $\tilde{\alpha}$ from $A_{l}^{\eta}(w)$ and $\alpha$ with the inequality $s_{i}\alpha_{\max}>\alpha_{\max}$. Hence for $\alpha \in A_{l}^{\eta}(w)$ we get $\alpha<s_{i}\alpha_{\max}$, while if $\alpha=s_{i}\tilde{\alpha}$ for $\tilde{\alpha} \in A_{l}^{\eta}(w)$ then $\tilde{\alpha}<s_{i}\alpha_{\max}$, and the fact that $s_{i}\alpha_{\max}$ contains $i+1$ and not $i$ yields $\alpha \leq s_{i}\alpha_{\max}$ as well before. This proves part $(iv)$, and with it, the lemma.
\end{proof}

We can now prove the result about the multiplicities in the case $r=1$.
\begin{thm}
Take a permutation $w$, and assume that $l$, $k$, and $\eta=\eta_{1}+2\eta_{2}$ are such that either $l>k=|\eta_{2}|+1$ or $l=k=|\eta_{2}|+2$. Then the multiplicity $m^{k,l}_{\eta}(w)$ equals 1 when $\eta \in B_{k,l}(w)$, and 0 otherwise. \label{diff1}
\end{thm}

\begin{proof}
The fact that $m_{\eta}(w)=0$ when $\eta \not\in B_{k,l}(w)$ is covered in Theorem \ref{mposBkl}, so we apply the same induction as usual to prove that if $\eta \in B_{k,l}(w)$ then $m_{\eta}(w)=1$. For the base case $w=\operatorname{Id}$, Remark \ref{forId} shows that $A_{l}(w)$ and $A_{k}(w)$ are singletons, with the unique element of the latter contained in that of the former as sets. It follows that every element $\eta$ of the set $\tilde{B}_{k,l}(\operatorname{Id})$ satisfies $|\eta_{2}|=k$ and thus $B_{k,l}(\operatorname{Id})=\emptyset$ by Proposition \ref{Bkldiff}. Since $P_{\operatorname{Id},2}=0$, the result holds for this case.

So assume that the assertion holds for some permutation $w$, and as always we take an index $i$ with $\ell(s_{i}w)>\ell(w)$. Consider an appropriate multi-set $\eta$, and note that if $\eta$ is $\eta^{00}$, $\eta^{10}$, $\eta^{20}$, $\eta^{21}$, or $\eta^{22}$ then $m_{\eta}(s_{i}w)=m_{\eta}(w)$ by Lemma \ref{multsiw}. Now, if such $\eta$ is in $B_{k,l}(s_{i}w)$, or equivalently $\tilde{B}_{k,l}(s_{i}w)$ because the inequality from Proposition \ref{Bkldiff} holds by our initial assumption on $\eta$, then consider any presentation of it as the sum of $\alpha \in A_{l}(s_{i}w)$ and $\beta \in A_{k}(s_{i}w)$. We get $\alpha \in A_{l}^{\eta}(s_{i}w)$ and $\beta \in A_{k}^{\eta}(s_{i}w)$ by Lemma \ref{conteta} as usual, and thus part $(i)$ of Lemma \ref{siwsets} implies that $\alpha \in A_{l}^{\eta}(w) \subseteq A_{l}(w)$ and $\beta \in A_{k}^{\eta}(w) \subseteq A_{k}(w)$, and we get $\eta\in\tilde{B}_{k,l}(w)$ and thus $\eta \in B_{k,l}(w)$ via Proposition \ref{Bkldiff}. It follows from our induction hypothesis that $m_{\eta}(s_{i}w)=m_{\eta}(w)=1$, and the result is established for these types of $\eta$.

We now take $\eta$ that is either $\eta^{01}$, $\eta^{02}$, or $\eta^{12}$, so that Lemma \ref{multsiw} gives $m_{\eta}(s_{i}w)=m_{s_{i}\eta}(w)$. When $\eta$ is in $B_{k,l}(s_{i}w)$, Proposition \ref{Bsiw} shows that $s_{i}\eta$ is also in that set, and since it is of one of the types mentioned in the previous paragraph, the argument there implies that $s_{i}\eta \in B_{k,l}(w)$ and thus $m_{\eta}(s_{i}w)=m_{s_{i}\eta}(w)=1$ by the induction hypothesis once again.

Finally we consider the case where $\eta=\eta^{11}$, and then since the relations from Remark \ref{inddecom} imply that $\eta^{20}$ and $\eta^{02}$ no longer satisfy the inequality from Proposition \ref{Bkldiff}, they are not in $B_{k,l}(w)$ and hence Theorem \ref{mposBkl} reduces the expression for $m_{\eta^{11}}(s_{i}w)$ in Lemma \ref{multsiw} to $m_{\eta^{11}}(w)+m_{\eta^{02}}(N_{w,i}^{(2)})$. As the induction hypothesis implies that the first term is 1 when $\eta^{11} \in B_{k,l}(w)$ and 0 when $\eta^{11} \not\in B_{k,l}(w)$, and Lemma \ref{Nwipos} yields that the second term is positive when $\eta^{11} \in B_{k,l}(s_{i}w) \setminus B_{k,l}(w)$ and vanishes otherwise, it only remains to verify that when this term is positive, it equals 1.

But the proof of the equivalence of $(i)$ and $(ii)$ in Lemma \ref{Nwipos} implies that this multiplicity is the number of pairs, which for $k=l$ we take to be unordered, of $\alpha \in A_{l,i}(w)$ and $\beta \in A_{k,i}(w)$, or equivalently $\alpha \in A_{l}(w)$ and $\beta \in A_{k}(w)$ in general, such that $\alpha+\beta=\eta^{20}$. Moreover, as $\eta^{20}\in\tilde{B}_{k,l}(w)$ by Lemma \ref{Nwipos}, we know that at least one such pair exists. If $k<l$ then Remark \ref{inddecom} and our condition on $|\eta_{2}|=|\eta^{11}_{2}|$ implies that $|\eta^{20}_{2}|=k$, so that the proof of Proposition \ref{Bkldiff} shows that there is only one such possible pair, namely $\beta=\eta^{20}_{2}$ and $\alpha=\eta^{20}_{1}\cup\eta^{20}_{2}$, and the desired number is indeed 1. When $k=l$ the same considerations give $|\eta^{20}_{2}|=k-1$ and $|\eta^{20}_{1}|=2$, and the proof of Proposition \ref{Bkldiff} again gives only a single pair up to ordering, and the number is 1 also in this case. This completes the proof of the theorem.
\end{proof}
Note that in all the expressions around Equation \eqref{P31425}, every element $\eta$ in some set $B_{k,l}$ satisfies $|\eta_{2}|=k-1$ (and $k<l$), and they all indeed show up with multiplicity 1, as Theorem \ref{diff1} predicts.

\smallskip

Before we turn to higher cases of $r$, we prove a complement of Lemma \ref{siwsets}, involving the more delicate case where $\eta=\eta^{11}$.
\begin{lem}
Take $\eta=\eta^{11}$, with $k$, $l$, $w$, and $i$ as above, set $\tilde{\eta}=\eta^{20}$, and assume that $A_{l}^{\tilde{\eta}}(w)$ contains an element of $A_{l,i}(w)$. The the following assertions hold:
\begin{enumerate}[$(i)$]
\item The sequence $\alpha_{\max}^{l,w}(\tilde{\eta})$ is also in $A_{l,i}(w)$.
\item We have the equality $\alpha_{\max}^{l,w}(\eta)=\alpha_{\max}^{l,w}(\tilde{\eta})$.
\item The equality $\alpha_{\max}^{l,s_{i}w}(\eta)=s_{i}\alpha_{\max}^{l,w}(\eta)$ also holds.
\item When $|\eta_{1}|+2|\eta_{2}|=k+l$, we have the equalities $\beta_{\min}^{l,w}(\eta)=s_{i}\beta_{\min}^{l,w}(\tilde{\eta})$ and $\beta_{\min}^{k,s_{i}w}(\eta)=\beta_{\min}^{k,w}(\tilde{\eta})$, with the latter element containing $i$ and not $i+1$.
\end{enumerate} \label{eta11siw}
\end{lem}
Also in part $(iv)$ of Lemma \ref{eta11siw} we do not assume that $k \leq l$, so that we can do the same interchanges when applying that lemma as well.

\begin{proof}
Since $A_{l,i}(w)\neq\emptyset$, we have $w^{-1}(i) \leq l<w^{-1}(i+1)$ by Lemma \ref{nontrivAli}, so that there is $1 \leq j \leq l$ such that $w^{l}_{j}=i$ and $i+1$ is not in $w^{l}$. Our assumption yields, through that lemma, that some element $\alpha \in A_{l}^{\tilde{\eta}}(w)$ satisfies $\alpha_{j}=i$, and $\alpha$ does not contain $i+1$. Since Proposition \ref{maxelt} and the fact that $\alpha_{\max}^{l,w}(\tilde{\eta}) \in A_{l}^{\tilde{\eta}}(w) \subseteq A_{l}(w)$ give, via Definition \ref{Asetdef}, that $\alpha\leq\alpha_{\max}^{l,w}(\tilde{\eta}) \leq w^{l}$, the $j$th entry of $\alpha_{\max}^{l,w}(\tilde{\eta})$ must also be $i$ and the next one strictly larger than $i+1$. Part $(i)$ thus follows from another application of Lemma \ref{nontrivAli}.

Now, as Remark \ref{inddecom} yields $\eta_{2}\subseteq\tilde{\eta}_{2}$ but $\tilde{\eta}_{1}\cup\tilde{\eta}_{2}\subseteq\eta_{1}\cup\eta_{2}$, with the respective complements $\{i\}$ and $\{i+1\}$, we deduce from Definition \ref{setswitheta} that $A_{l}^{\tilde{\eta}}(w) \subseteq A_{l}^{\eta}(w)$. More precisely, an element of $A_{l}^{\eta}(w)$ lies in $A_{l}^{\tilde{\eta}}(w)$ if and only if it contains $i$ but does not contain $i+1$. The inclusion of sets yields, via Proposition \ref{maxelt}, the inequality $\alpha_{\max}^{l,w}(\tilde{\eta})\leq\alpha_{\max}^{l,w}(\eta)$, and the previous paragraph, which did not depend on the type of $\tilde{\eta}$ hence holds also for $\eta$, implies that $\alpha_{\max}^{l,w}(\eta)$ is also in $A_{l,i}(w)$. But then it contains $i$ and not $i+1$ (Lemma \ref{nontrivAli} once more), so that it lies $A_{l}^{\tilde{\eta}}(w)$ as we saw, and satisfies the reverse inequality. This proves the equality from part $(ii)$, and part $(iii)$ then follows directly from part $(v)$ of Lemma \ref{siwsets}.

Finally, the joint element from part $(ii)$ contains $i$ and not $i+1$, and $\beta_{\min}^{k,w}(\tilde{\eta})$ is its complement to $\tilde{\eta}=\eta^{20}$, which thus also contains $i$ and not $i+1$. Hence the sequence $\beta_{\min}^{k,w}(\eta)$ that we have to add in order to get $\eta=\eta^{11}$ is indeed its $s_{i}$-image, yielding the first equality in part $(iv)$. Part $(iii)$ though shows that for $s_{i}w$ and $\eta=\eta^{11}$ we moved $i$ to $i+1$ in $\alpha_{\max}$, and thus its sum with $\beta_{\min}^{k,w}(\tilde{\eta})$ is indeed $\eta^{11}$, and the second equality follows as well. This proves the lemma.
\end{proof}

We will also be using the following properties of these elements.
\begin{lem}
With the usual $w$, $i$, $l$, $k$, and $\eta=\eta^{11}\in\tilde{B}_{k,l}(w)$, let $j$ and $h$ be the indices for which $w^{l}_{j}=w^{k}_{h}=i$ when they exist, and then the following results hold:
\begin{enumerate}[$(i)$]
\item If $\eta^{20}\not\in\tilde{B}_{k,l}(w)$, then when $\eta=\alpha+\beta$ with $\alpha \in A_{l}(s_{i}w)$ and $\beta \in A_{k}(s_{i}w)$, we already have $\alpha \in A_{l}(w)$ and $\beta \in A_{k}(w)$. The elements $\alpha_{\max}$, $\alpha_{\min}$, $\beta_{\max}$, and $\beta_{\min}$ remain unchanged when going from $w$ to $s_{i}w$ in this case.
\item In case $\eta^{20}\in\tilde{B}_{k,l}(w)$ but $\eta^{02}\not\in\tilde{B}_{k,l}(w)$, we either have $\alpha_{\max} \in A_{l,i}(w)$ and the minimal entry of $\alpha_{\min}\cap\eta_{1}$ that lies after the $j$th entry is larger than $i$, or $\beta_{\max} \in A_{k,i}(w)$ and the all the entries of $\beta_{\min}\cap\eta_{1}$ with indices larger than $h$ being larger than $i$, or both.
\item When $\eta^{02}\in\tilde{B}_{k,l}(w)$, if $\beta_{\max} \in A_{k,i}(w)$, then $i+1\not\in\beta_{\max}$ and some entry of $\beta_{\min}\cap\eta_{1}$ is smaller than $i$ and has index larger than $h$. In case $\alpha_{\max} \in A_{l,i}(w)$ we have $i\not\in\beta_{\min}$ but $i+1\in\beta_{\min}$, lying in some entry $p$ say, and there is an entry of $\beta_{\max}\cap\eta_{1}$ that is larger than $i+1$ and has index smaller that $p$.
\end{enumerate} \label{siweta11}
\end{lem}

\begin{proof}
For part $(i)$, Lemma \ref{eta20Bkl} gives that $\beta_{\max}$ contains neither $i$ nor $i+1$, and that $\alpha_{\max}$ contains them both. Since in every sum $\eta=\alpha+\beta$ we have $\alpha \in A_{l}^{\eta}(s_{i}w)$ and $\beta \in A_{k}^{\eta}(s_{i}w)$ by Lemma \ref{conteta} as usual, part $(i)$ follows directly from part $(iv)$ of Lemma \ref{siwsets}.

Lemma \ref{eta02Bkl} implies that in part $(ii)$ the equalities on the sizes of the sets from condition $(iii)$ there must hold, while from Lemma \ref{eta20Bkl} we deduce that at least one of $\beta_{\max}$ and $\beta_{\min}$ must contain either $i$ or $i+1$. But if $i+1\in\beta_{\max}$ then it is not in $\beta_{\min}$, hence must be replaced by a smaller number, which contradicts condition $(iii)$ of Lemma \ref{eta02Bkl}. Similarly, when $i\in\beta_{\min}$ the corresponding entry of $\beta_{\max}$ is larger, again contradicting this equality. It follows that $i+1\in\alpha_{\min}$ and $i\in\alpha_{\max}$, and we have either $i\in\beta_{\max}$ or $i+1\in\beta_{\min}$ (or both). But in the former case we get, by the maximality of $\beta_{\max}$, that $s_{i}\beta_{\max} \not\in A_{k}(w)$ and thus $\beta_{\max} \in A_{k,i}(w)$ via Equation \eqref{Alwdecom}, and in the latter $\alpha_{\max} \in A_{l,i}(w)$ by the same manner. Part $(ii)$ is thus established, as the assertions about the entries of $\beta_{\min}$ and $\alpha_{\min}$ follow from condition $(iii)$ of Lemma \ref{eta02Bkl}.

For part $(iii)$ we invoke Lemma \ref{nontrivAli}, which implies that $i+1\not\in\beta_{\max}$ provides us the location of $i$ in $\beta_{\max}$ in case it is in $A_{k,i}(w)$, and also gives, via Definition \ref{minmaxelts}, that $i\not\in\beta_{\min}$ and $i+1\in\beta_{\min}$ when $\alpha_{\max} \in A_{l,i}(w)$. The contraposition of condition $(iii)$ of Lemma \ref{eta02Bkl} yields the remaining assertions of part $(iii)$. This proves the lemma.
\end{proof}

\smallskip

When $r\geq2$ the multiplicities $m_{\eta}(w)$ can be different for different elements $\eta \in B_{k,l}(w)$. To investigate them we shall need some notation. Recall from Proposition \ref{contmax} that $\beta_{\min}\cap\beta_{\max}=\eta_{2}$ with these $k$, $w$, and $\eta$, meaning that if $\eta\in\tilde{B}_{k,l}(w)$ then each of these elements of $A_{k}(w)$ is the union of $\eta_{2}$ with a set of $r+\delta_{k,l}$ elements from $\eta_{1}$, and these sets of $r+\delta_{k,l}$ elements are disjoint. Moreover, as $\beta_{\min}\leq\beta_{\max}$ (by Corollary \ref{setofpairs}), the corresponding subsets of $\eta_{1}$ satisfy the same inequality. It is also clear that if $|\eta_{1}|+2|\eta_{2}|=k+l$ and $|\eta_{2}|=k-\delta_{k,l}-r$ then $|\eta_{1}|=l-k+2r+2\delta_{k,l}$. In what follows we will need a notation for the indices of elements of $\eta_{1}$, so we write the entries of $\operatorname{ord}(\eta_{1})$ as $(\gamma_{1},\ldots,\gamma_{l-k+2r+2\delta_{k,l}})$.

We wish to obtain an exact formula for $m^{k,l}_{\eta}(w)$ in case $|\eta_{2}|=k-\delta_{k,l}-2$. Assuming that $k<l$, $\eta \in B_{k,l}$, and $|\eta_{2}|=k-2$, we deduce from the previous paragraph that $\beta_{\min}=\eta_{2}\cup\{\gamma_{a},\gamma_{b}\}$ and $\beta_{\max}=\eta_{2}\cup\{\gamma_{c},\gamma_{d}\}$, where we assume that $a<b$ and $c<d$. The fact that $\beta_{\min}\leq\beta_{\max}$ and $\beta_{\min}\cap\beta_{\max}=\eta_{2}$ implies that $a<c$ and $b<d$ (hence $a<d$) as well as $b \neq c$, but both $b>c$ and $b<c$ are possible. We will also set $\varepsilon$ to be 1 in case $b>c$ and 0 when $b<c$.
\begin{thm}
Assume that $\eta \in B_{k,l}(w)$ satisfies $l>k=|\eta_{2}|+2$, and write the entries of $\operatorname{ord}(\eta_{1})$ as $(\gamma_{1},\ldots,\gamma_{l-k+4})$. Then, in the notation above, we have $m^{k,l}_{\eta}(w)=d-a+1-\varepsilon(b-c)$. \label{diff2}
\end{thm}
The multiplicity from Theorem \ref{diff2} is thus $c+d-a-b+1$ when $b>c$, and just $d-a+1$ in case $b<c$.

\begin{proof}
We follow the proof of Theorem \ref{diff1}. When $w=\operatorname{Id}$ the set $B_{k,l}(w)$ is empty for every $k$ and $l$, so that the result holds in an empty manner. Assume that we have a permutation $w$ for which it holds for every $\eta$, choose $i$ such that $\ell(s_{i}w)>\ell(w)$, take $\eta \in B_{k,l}(s_{i}w)$, and we need to compare the sequences $\beta_{\max}^{k,s_{i}w}(\eta)$ and $\beta_{\min}^{k,s_{i}w}(\eta)$ with the appropriate sequences for $w$.

When $\eta$ is one of $\eta^{00}$, $\eta^{10}$, $\eta^{20}$, $\eta^{21}$, or $\eta^{22}$, the proof of Theorem \ref{diff1} implies that $\eta \in B_{k,l}(w)$. Moveover, we have the equalities $\beta_{\max}^{k,s_{i}w}(\eta)=\beta_{\max}^{k,w}(\eta)$ and $\beta_{\min}^{k,s_{i}w}(\eta)=\beta_{\min}^{k,w}(\eta)$ by part $(i)$ of Lemma \ref{siwsets} in this case, so that the parameters $a$, $b$, $c$, and $d$, hence also $\varepsilon$, that we get for $s_{i}w$ and $\eta$ are the same as those for $w$ and $\eta$. Since Lemma \ref{multsiw} and the induction hypothesis give the equalities $m_{\eta}(s_{i}w)=m_{\eta}=d-a+1-\varepsilon(b-c)$, the result follows in this case.

Next we turn our attention to $\eta$ that is $\eta^{01}$, $\eta^{02}$, or $\eta^{12}$. In this case the proof of Theorem \ref{diff1} produces $s_{i}\eta \in B_{k,l}(w)$, and part $(ii)$ of Lemma \ref{siwsets} gives $\beta_{\max}^{k,s_{i}w}(\eta)=s_{i}\beta_{\max}^{k,w}(s_{i}\eta)$ and $\beta_{\min}^{k,s_{i}w}(\eta)=s_{i}\beta_{\min}^{k,w}(s_{i}\eta)$. As part $(iii)$ of the latter lemma shows that the application of $s_{i}$ does not change the indices from the ordering of $\eta_{1}$ in these cases, the values of $a$, $b$, $c$, and $d$, as well as the resulting one of $\varepsilon$, that arise from $s_{i}w$ and $\eta$ coincide with those coming from $w$ and $s_{i}\eta$. The equalities $m_{\eta}(s_{i}w)=m_{s_{i}\eta}(w)=d-a+1-\varepsilon(b-c)$ resulting from Lemma \ref{multsiw} and the induction hypothesis establish the result in this setting.

So now take $\eta=\eta^{11}$, which is assumed to be in $B_{k,l}(s_{i}w)$, and consider first the case where $\eta^{11} \not\in B_{k,l}(w)$, so that Corollary \ref{meta11siw} expresses $m_{\eta^{11}}(s_{i}w)$ as $m_{\eta^{20}}(w)+m_{\eta^{02}}(N_{w,i}^{(2)})$. Then Remark \ref{inddecom}, Proposition \ref{contmax}, and Corollary \ref{setofpairs} imply that both $\beta_{\max}^{k,w}(\eta^{20})$ and $\beta_{\min}^{k,w}(\eta^{20})$ contain $\eta^{20}_{2}=\eta^{11}_{2}\cup\{i\}$ and another element from $\eta^{20}_{1}=\eta^{11}_{2}\setminus\{i,i+1\}$, with the latter element being strictly larger for $\beta_{\max}^{k,w}(\eta^{20})$ than for $\beta_{\min}^{k,w}(\eta^{20})$. Moreover, Lemma \ref{Nwipos} implies that $\eta^{20}\in\tilde{B}_{k,l}(w)$, so that it is also in $B_{k,l}(w)$ via Remark \ref{inddecom}, the condition on $|\eta_{2}|=|\eta^{11}_{2}|$, and Proposition \ref{Bkldiff}, and hence $m_{\eta^{20}}(w)=1$ by Theorem \ref{diff1}.

We denote by $e$ and $f$ the indices with which we have $i\neq\gamma_{f}\in\beta_{\max}^{k,w}(\eta^{20})$ and $i\neq\gamma_{e}\in\beta_{\min}^{k,w}(\eta^{20})$, so that $e<f$ and both these sequences are in $A_{k,i}(w)$ by Lemma \ref{Nwipos}. Moreover, as Lemma \ref{nontrivAli} implies that $i$ shows up in the same entry of both $\beta_{\max}^{k,w}(\eta^{20})$ and $\beta_{\min}^{k,w}(\eta^{20})$, we deduce that $\gamma_{e}$ and $\gamma_{f}$ are either both smaller than $i$ or both larger than $i+1$. Hence if $i$ is $\gamma_{t}$ for some $t$, so that $i+1=\gamma_{t+1}$ (as both are in $\eta^{11}_{1}$), then either $e$ and $f$ are both smaller that $t$ or are both larger than $t+1$.

The number of pairs is now given, via Corollary \ref{setofpairs} as the number of choices of $\beta \in A_{k}(w)$ that satisfy $\eta^{20}_{2}\subseteq\beta\subseteq\eta^{20}$ and $\beta_{\min}^{k,w}(\eta^{20})\leq\beta\leq\beta_{\max}^{k,w}(\eta^{20})$. But there is only one free element in $\beta$, which has to be taken from $\eta^{20}_{1}$ hence from $\eta^{11}_{1}$ and be between $\gamma_{e}$ and $\gamma_{f}$, and there are precisely $f-e+1$ such options in $\eta^{11}_{1}$ in our notation. Moreover, recalling that both $e$ and $f$ are on the same side of $t$ and $t+1$, all these $f-e+1$ options are from $\eta^{20}_{1}$ as desired. It follows that $m_{\eta^{11}}(s_{i}w)=f-e+2$ in this case.

But Lemma \ref{Nwipos} implies that $\beta_{\max}^{k,w}(\eta^{20}) \in A_{k,i}(w)$ and $\alpha_{\max}^{l,w}(\eta^{20}) \in A_{l,i}(w)$ in this case. We thus invoke Part $(iii)$ of Lemma \ref{eta11siw}, with $\beta$ and $k$, which gives that $\beta_{\max}^{k,s_{i}w}(\eta^{11})=s_{i}\beta_{\max}^{k,w}(\eta^{20})$, whose explicit description is $\eta_{2}\cup\{\gamma_{f},i+1\}$. Part $(iv)$ of that lemma, as is stated there, yields $\beta_{\min}^{k,s_{i}w}(\eta^{11})=\beta_{\min}^{k,w}(\eta^{20})$, which explicitly equals $\eta_{2}\cup\{\gamma_{e},i\}$. This means that if $e$ and $f$ are smaller than $t$ then $d=t+1$, $c=f$, $b=t$, $a=e$, $\varepsilon=1$, and we get a value of $f-e+2$, and in case $e$ and $f$ are larger that $t+1$ we get $d=f$, $c=t+1$, $b=e$, $a=t$, $\varepsilon=1$, and again the value is $f-e+2$. This establishes the result in this case.

The remaining case is where $\eta=\eta^{11} \in B_{k,l}(w)$, where Corollary \ref{meta11siw} gives $m_{\eta^{11}}(s_{i}w)=m_{\eta^{11}}(w)+\big(m_{\eta^{20}}(w)-m_{\eta^{02}}(w)\big)$. In addition, since if $\eta^{02}\in\tilde{B}_{k,l}(w)$ then $\eta^{20}\in\tilde{B}_{k,l}(w)$ (by applying $s_{i}$ to the sum), and in our setting this is the same for $B_{k,l}(w)$, we deduce, via Remark \ref{inddecom}, our condition on $|\eta^{11}_{2}|$, and Theorem \ref{diff1} again, that $m_{\eta^{11}}(s_{i}w)$ is $m_{\eta^{11}}(w)+1$ in case $\eta^{20}$ is in $B_{k,l}(w)$ and $\eta^{02}$ is not, and just $m_{\eta^{11}}(w)$ if either both are in $B_{k,l}(w)$ or both are not. Part $(i)$ of Lemma \ref{siweta11} now implies that when neither $\eta^{20}$ nor $\eta^{02}$ are in $B_{k,l}(w)$, we have $\beta_{\max}^{k,s_{i}w}(\eta)=\beta_{\max}^{k,w}(\eta)$ and $\beta_{\min}^{k,s_{i}w}(\eta)=\beta_{\min}^{k,w}(\eta)$, so that $a$, $b$, $c$, and $d$ (hence $\varepsilon$) remain the same when going from $w$ to $s_{i}w$, and indeed the value of $m_{\eta^{11}}(w)$ from the induction hypothesis yields the asserted one for $m_{\eta^{11}}(s_{i}w)$.

Next, part $(ii)$ of Lemma \ref{siweta11} divides the case in which $\eta^{20} \in B_{k,l}(w)$ and $\eta^{02} \not\in B_{k,l}(w)$ into three sub-cases, in which we draw conclusions using Lemma \ref{nontrivAli}, Definition \ref{minmaxelts}, Proposition \ref{contmax}, and part $(iv)$ of Lemma \ref{siwsets}. One possibility is with $\beta_{\max} \in A_{k,i}(w)$ and $\alpha_{\max} \not\in A_{l,i}(w)$, where $i\in\beta_{\max}$, $i\not\in\beta_{\min}$, $i\in\alpha_{\max}$, and $i+1\in\alpha_{\max}$ thus $i+1\not\in\beta_{\min}$. Another is with $\alpha_{\max} \in A_{l,i}(w)$ and $\beta_{\max} \not\in A_{k,i}(w)$, so that $i+1\in\beta_{\min}$, $i+1\not\in\beta_{\max}$, and $i\not\in\beta_{\max}$. Finally when $\alpha_{\max} \in A_{l,i}(w)$ and $\beta_{\max} \in A_{k,i}(w)$ we get $i\in\beta_{\max}$, $i+1\not\in\beta_{\max}$, $i\not\in\beta_{\min}$, and $i+1\in\beta_{\min}$.

Combining these with the cardinality conditions from part $(ii)$ of Lemma \ref{siweta11}, the first sub-case produces the setting where $d=t$ and that in which $c=t$ and $b>t+1$ (because $d>t+1$), and applying $s_{i}$ increased $d$ by 1 in the former case, while in the latter we have $\varepsilon=1$ and $s_{i}$ increases $c$ by 1 (still with $\varepsilon=1$). From the second one we extract the situation with $a=t+1$ and that where $b=t+1$ and $c<t$ (since $a<t$), and in the last one we have $c=t$ and $b=t+1$, where applying $s_{i}$ subtracts 1 from $a$ in the former case, while in the latter we again have $\varepsilon=1$, $s_{i}$ subtracts 1 from $b$, and leaves the value of $\varepsilon$. In the remaining case we have $\varepsilon(b-c)=1$, but $s_{i}$ interchanges $c=t$ and $b=t+1$ hence takes $\varepsilon=1$ to 0. In total, in all these situations the asserted value of $m_{\eta^{11}}(s_{i}w)$ is one more than the one of $m_{\eta^{11}}(w)$ from the induction hypothesis, in correspondence with the value from Corollary \ref{meta11siw} and Theorem \ref{diff1}.

Finally, when $\eta^{02}\in\tilde{B}_{k,l}(w)$, if $\alpha_{\max} \in A_{l}^{(i)}(w)$ and $\beta_{\max} \in A_{k}^{(i)}(w)$ then once again $\beta_{\max}^{k,s_{i}w}(\eta)=\beta_{\max}^{k,w}(\eta)$ and $\beta_{\min}^{k,s_{i}w}(\eta)=\beta_{\min}^{k,w}(\eta)$ via part $(i)$ of Lemma \ref{siweta11}, so that $a$, $b$, $c$, and $d$ (and $\varepsilon$) do not change when we replace $w$ by $s_{i}w$, and the induction hypothesis yields that $m_{\eta^{11}}(s_{i}w)$ has the asserted value. When $\beta_{\max}$ is in $A_{k,i}(w)$ or $\alpha_{\max} \in A_{l,i}(w)$ we invoke part $(iii)$ of Lemma \ref{siweta11}, which shows, since $k-|\eta_{2}|=2$, that in the former case our indices satisfy the inequalities $a<b<t=c<t+1<d$ and in the latter we have $a<t<t+1=b<c<d$. Since in these situations we have $\varepsilon=0$, which remains such also after applying $s_{i}$, we obtain that the desired value of $m_{\eta^{11}}(s_{i}w)$ coincides with the known one $m_{\eta^{11}}(w)$ also in this case. This completes the proof of the theorem.
\end{proof}

We now consider the complementary case for $r=2$, in which $l=k=|\eta_{2}|+3$ and thus $|\eta_{1}|=6$. Proposition \ref{contmax} implies that for each $w$ the entries of $\eta_{1}$ are divided between $\beta_{\min}$ and $\beta_{\max}$, and Corollary \ref{setofpairs} and Theorem \ref{mposBkl} show that for positive multiplicities we may restrict attention to the situation where $\beta_{\min}\leq\beta_{\max}$. The complementary result to Theorem \ref{diff2} for this case, which is proved in a similar manner, is as follows.
\begin{prop}
Given such $k$,$l$, $\eta$, and $w$, there are 5 options for the values of $\beta_{\min}$ and $\beta_{\max}$ satisfying the inequality from Corollary \ref{setofpairs}, yielding the following respective multiplicities.
\begin{enumerate}[$(i)$]
\item If $\beta_{\min}\cap\eta_{1}=\{\gamma_{1},\gamma_{3},\gamma_{5}\}$ and $\beta_{\max}\cap\eta_{1}=\{\gamma_{2},\gamma_{4},\gamma_{6}\}$, then $m_{\eta}=3$.
\item When $\beta_{\min}\cap\eta_{1}=\{\gamma_{1},\gamma_{3},\gamma_{4}\}$ and $\beta_{\max}\cap\eta_{1}=\{\gamma_{2},\gamma_{5},\gamma_{6}\}$, we get $m_{\eta}=4$.
\item When $\beta_{\min}\cap\eta_{1}=\{\gamma_{1},\gamma_{2},\gamma_{5}\}$ and $\beta_{\max}\cap\eta_{1}=\{\gamma_{3},\gamma_{4},\gamma_{6}\}$, we get $m_{\eta}=4$.
\item In case $\beta_{\min}\cap\eta_{1}=\{\gamma_{1},\gamma_{2},\gamma_{4}\}$ and $\beta_{\max}\cap\eta_{1}=\{\gamma_{3},\gamma_{5},\gamma_{6}\}$, we get $m_{\eta}=5$.
\item In case $\beta_{\min}\cap\eta_{1}=\{\gamma_{1},\gamma_{2},\gamma_{3}\}$ and $\beta_{\max}\cap\eta_{1}=\{\gamma_{4},\gamma_{5},\gamma_{6}\}$, we get $m_{\eta}=5$.
\end{enumerate} \label{lketa23}
\end{prop}
In fact, by writing $\beta_{\min}\cap\eta_{1}=\{\gamma_{a},\gamma_{b},\gamma_{c}\}$ and $\beta_{\max}\cap\eta_{1}=\{\gamma_{d},\gamma_{e},\gamma_{f}\}$, we can write the multiplicities from Proposition \ref{lketa23} in a unified manner, which is similar to that from Theorem \ref{diff2}, as $f-a-\varepsilon(b-d)-\delta(c-e)$, with $\varepsilon$ being 1 when $b>d$ and 0 in case $b<d$, and in which $\delta$ is 1 in case $c>e$ and 0 if $c<e$.

\begin{proof}
We must have $\gamma_{1}\in\beta_{\min}$ and $\gamma_{6}\in\beta_{\max}$, with the index of the intermediate element of $\beta_{\max}\cap\eta_{1}$ being at least 4 (as it must be larger than one element of that set and two of $\beta_{\min}\cap\eta_{1}$), with that of $\beta_{\min}\cap\eta_{1}$ being bounded by 3 (as one more element there, and two of $\beta_{\max}\cap\eta_{1}$, must be larger). These considerations produce only our five cases, and in all of them we indeed have $\beta_{\min}\leq\beta_{\max}$. We also know that $\beta_{\max}$ determines $\beta_{\min}$ by Definition \ref{minmaxelts}.

We apply the same arguments from the proof of Theorems \ref{diff1} and \ref{diff2}, where the case $w=\operatorname{Id}$ is true as an empty statement, and we assume the result for $w$ and aim to prove it for $s_{i}w$ for some $i$ with $\ell(s_{i}w)>\ell(w)$. If $\eta$ is $\eta^{00}$, $\eta^{10}$, $\eta^{20}$, $\eta^{21}$, or $\eta^{22}$, then part $(i)$ of Lemma \ref{siwsets} shows that the sequences arising from $s_{i}w$ are the same as those arising from $w$, and the assertion follows from the induction hypothesis and the equality $m_{\eta}(s_{i}w)=m_{\eta}(w)$ from Lemma \ref{multsiw}. When $\eta$ is one of $\eta^{01}$, $\eta^{02}$, and $\eta^{12}$, parts $(ii)$ and $(iii)$ of Lemma \ref{siwsets} imply that the case associated with $s_{i}w$ and $\eta$ is the same as the one corresponding to $w$ and $s_{i}\eta$, so that the equality $m_{\eta}(s_{i}w)=m_{s_{i}\eta}(w)$ from Lemma \ref{multsiw} and the induction hypothesis conclude this case as well.

We shall thus henceforth assume that $\eta=\eta^{11}$, and that it is in $B_{k,l}(s_{i}w)$ (hence belongs to one of our five cases). If this element is not in $B_{k,l}(w)$, then we have $m_{\eta^{11}}(s_{i}w)=m_{\eta^{20}}(w)+m_{\eta^{02}}(N_{w,i}^{(2)})$ by Corollary \ref{meta11siw}, with Remark \ref{inddecom} and our assumption giving $|\eta^{20}_{2}|=|\eta^{11}_{2}|+1=k-2=l-2$, which combines with Proposition \ref{Bkldiff} and the fact that $\eta^{20}\in\tilde{B}_{k,l}(w)$ by Lemma \ref{Nwipos} to show that $\eta^{20} \in B_{k,l}(w)$ and thus $m_{\eta^{20}}(w)=1$ by Theorem \ref{diff1}.

Now, when $\eta=\eta^{11}$ is associated with one of our cases for $s_{i}w$ but not for $w$, we must have $\beta_{\max}^{k,s_{i}w}(\eta)\neq\beta_{\max}^{k,w}(\eta)$, which by parts $(iv)$ and $(v)$ of Lemma \ref{siwsets} happens if and only if $\beta_{\max}^{k,w}(\eta) \in A_{k,i}(w)$ and $\beta_{\max}^{k,s_{i}w}(\eta)=s_{i}\beta_{\max}^{k,w}(\eta)$. Hence $\beta_{\max}^{k,s_{i}w}(\eta)$, which corresponds to one of our cases, contains $i+1=\gamma_{t+1}$ but not $i=\gamma_{t}$, and applying $s_{i}$, to replace $\gamma_{t+1}$ by $\gamma_{t}$ in $\beta_{\max}$, no longer produces any of our cases. This happens in case $(i)$ with $t$ being 1, 3, or 5, in case $(ii)$ with $t=1$, and in case $(iii)$ with $t=5$.

Corollary \ref{setofpairs} and Lemma \ref{Nwipos} then evaluate the multiplicity $m_{\eta^{02}}(N_{w,i}^{(2)})$ as the number of elements $\beta \in A_{k}^{\eta^{20}}(w)$ that satisfy $\beta_{\min}^{k,w}(\eta^{20})\leq\beta\leq\beta_{\max}^{k,w}(\eta^{20})$ (all of which will be in $A_{k,i}(w)$), and divide it by 2 since each sum $\alpha+\beta=\eta^{20}$ is obtained both from $\alpha$ and from $\beta$ (and they are distinct since $|\eta^{20}_{2}|<k$). For case $(i)$ and $t=1$, such $\beta$ is obtained by adding to $\eta_{20}^{2}$ the sets $\{\gamma_{3},\gamma_{5}\}$, its complement $\{\gamma_{4},\gamma_{6}\}$, or $\{\gamma_{3},\gamma_{6}\}$ and its complement $\{\gamma_{4},\gamma_{5}\}$. When $t=3$ the added sets are $\{\gamma_{1},\gamma_{5}\}$ with the complement $\{\gamma_{2},\gamma_{6}\}$, as well as $\{\gamma_{1},\gamma_{6}\}$ or its complement $\{\gamma_{2},\gamma_{5}\}$. If $t=5$ we can add $\{\gamma_{1},\gamma_{3}\}$, the complement $\{\gamma_{2},\gamma_{4}\}$, $\{\gamma_{1},\gamma_{4}\}$, or the complement $\{\gamma_{2},\gamma_{3}\}$. When we are in case $(ii)$ and $t=1$, the sets from $t=1$ and case $(i)$ are allowed, as well as $\{\gamma_{3},\gamma_{4}\}$ and the complement $\{\gamma_{5},\gamma_{6}\}$, and when $t=3$ and the case is 3 we get the sets showing up in case $(i)$ (with $t=5$), as well as $\{\gamma_{1},\gamma_{2}\}$ and its complement $\{\gamma_{3},\gamma_{4}\}$. Substituting the resulting value of $m_{\eta^{02}}(N_{w,i}^{(2)})$ and adding $m_{\eta^{20}}(w)=1$ yields indeed the asserted value for $m_{\eta^{11}}(s_{i}w)$.

It remains to consider the case with $\eta=\eta^{11}$ lying in $B_{k,l}(w)$, where the multiplicity $m_{\eta^{11}}(s_{i}w)$ is $m_{\eta^{11}}(w)+\big(m_{\eta^{20}}(w)-m_{\eta^{02}}(w)\big)$ by Corollary \ref{meta11siw}. Remark \ref{inddecom}, the assumption on $|\eta^{11}_{2}|$, and Theorem \ref{diff1} again reduce the latter expression to $m_{\eta^{11}}(w)+1$ when $\eta^{20}$ is in $\tilde{B}_{k,l}(w)$ (or equivalently in $B_{k,l}(w)$) and $\eta^{02}$ is not there, and to $m_{\eta^{11}}(w)$ when either both these sequences are in $\tilde{B}_{k,l}(w)$ (or equivalently $B_{k,l}(w)$), or neither are there. The fact that when neither sequence is in that set the sequences $\beta_{\max}$ and $\beta_{\min}$ are the same for $w$ and $s_{i}w$ (with our $\eta$) is again obtained from Part $(i)$ of Lemma \ref{siweta11}, yielding the assertion in this case by the induction hypothesis.

Using part $(ii)$ of Lemma \ref{siweta11} when $\eta^{20} \in B_{k,l}(w)$ and $\eta^{02} \not\in B_{k,l}(w)$ is easier in the case $k=l$, since then $\alpha_{\max}=\beta_{\max}$ and there is only one case, in which $\beta_{\max} \in A_{k,i}(w)$ (so that $i=\gamma_{t}\in\beta_{\max}$ and $i+1=\gamma_{t+1}$ is in its complement $\beta_{\min}$ by Lemma \ref{nontrivAli} and Definition \ref{minmaxelts}), and the cardinality condition from Lemma \ref{eta02Bkl} holds. This happens when our $\eta$ and $w$ produce case $(i)$ with $t=2$ or with $t=4$, or in case $(ii)$ with $t=2$, or in case $(iii)$ with $t=4$. Since the resulting cases for $s_{i}w$ are then $(ii)$, $(iii)$, $(iv)$, and $(iv)$ respectively, the equality $m_{\eta^{11}}(s_{i}w)=m_{\eta^{11}}(w)+1$ holds with the right hand side given by the induction hypothesis if $m_{\eta^{11}}(s_{i}w)$ is the asserted value.

The case in which $\eta^{02}\in\tilde{B}_{k,l}(w)$ and $\beta_{\max} \in A_{k}^{(i)}(w)$ produces again, using part $(i)$ of Lemma \ref{siweta11} and the induction hypothesis, that $m_{\eta^{11}}(s_{i}w)=m_{\eta^{11}}(w)$ is given by the desired value. Finally, if $\eta^{02}\in\tilde{B}_{k,l}(w)$ and $\beta_{\max} \in A_{k,i}(w)$ (and $\eta^{11} \in B_{k,l}(w)$), then part $(iii)$ of Lemma \ref{siweta11} implies that the cardinality conditions from Lemma \ref{eta02Bkl} cannot hold. This can happen only in case $(iv)$ for $\eta$ and $w$, with $t=3$, in which we have $m_{\eta^{11}}(s_{i}w)=m_{\eta^{11}}(w)=5$ by the induction hypothesis. As the result for $s_{i}w$ is case $(v)$, this is indeed the desired value also here. This completes the proof of the proposition.
\end{proof}

We deduce the following consequence.
\begin{cor}
If $l>k=|\eta_{2}|+2$ then $m^{k,l}_{\eta}(w)$ equals 0 when $\eta \not\in B_{k,l}(w)$ and lies between 3 and $l-k+4$ for $\eta \in B_{k,l}(w)$, with each value in this range being attained for some $w$. When $l=k=|\eta_{2}|+3$ the respective value and range are 0 and between 3 and 5. Moreover, wherever $\ell(s_{i}w)>\ell(w)$ we have $m^{k,l}_{\eta}(s_{i}w) \geq m^{k,l}_{\eta}(w)$ for any such $\eta$. \label{lowbdr2}
\end{cor}

\begin{proof}
The value 0 is a consequence of Theorem \ref{mposBkl}, and since in the notation from Theorem \ref{diff2} we have $d-b\geq1$ and $c-a\geq1$ (as well as $d-a\geq3$ because $b$ and $c$ both lie between them), this establishes the lower bound 3 when $l>k=|\eta_{2}|+2$. We now note that $\varepsilon=1$ only when $b>c$, so that the value $d-a+1-(b-c)$ arising from this case is smaller than $d-a+1$ occurring when $\varepsilon=0$. This value thus $\beta_{\max}$ increases and as $\beta_{\min}$ decreases, and since $A_{l}^{\eta}(w) \subseteq A_{l}^{\eta}(s_{i}w)$ when $\ell(s_{i}w)>\ell(w)$ (and the same with $k$), we get $\beta_{\max}^{k,s_{i}w}(\eta)\geq\beta_{\max}^{k,w}(\eta)$ and $\alpha_{\max}^{l,s_{i}w}(\eta)\geq\alpha_{\max}^{l,w}(\eta)$ hence $\beta_{\min}^{k,s_{i}w}(\eta)\leq\beta_{\min}^{k,w}(\eta)$, yielding the asserted inequality. Thus, the maximal value for $m_{\eta}$ will be obtained when $w^{k}$ and $w^{l}$ are large enough not to impose any restrictions on sequences that are contained in $\eta$, which will give $b<c$ (hence $\varepsilon=0$), $a=1$, $d=l-k+4$, and the asserted upper bound (which is also sharp). For any smaller value, say $d-1$ for $4 \leq d \leq l-k+4$, take a permutation $w$ for which $w^{k}=\operatorname{ord}(\eta_{2}\cup\{\gamma_{2},\gamma_{d}\})$ and $w^{l}=\operatorname{ord}(\eta_{1}\cup\eta_{2}\setminus\{\gamma_{1},\gamma_{3}\}) \supseteq w^{k}$, with which $\beta_{\max}=w^{k}$ (with our $d$ and with $b=2$), $\alpha_{\max}=w^{l}$, hence the complement $\beta_{\min}$ is with $a=1$ and $c=3>b$ (thus $\varepsilon=1$), and the multiplicity is indeed $d-1$. The result for $l=k=|\eta_{2}|+3$ follows directly from Proposition \ref{lketa23}, with the inequality holding by the same argument. This proves the corollary
\end{proof}

\section{Questions for Further Research \label{QuesRes}}

Theorems \ref{diff1} and \ref{diff2} and Proposition \ref{lketa23} imply that for $r\leq2$, the multiplicity $m^{k,l}_{\eta}(w)$ depends only on the set of sums $\alpha+\beta=\eta$ with $\alpha \in A_{l}(w)$ and $\beta \in A_{k}(w)$, which by Lemma \ref{conteta} and Corollary \ref{setofpairs} are the same as those $\alpha \in A_{l}^{\eta}(w)$ satisfying $\alpha_{\min}^{l,w}(\eta)\leq\beta\leq\alpha_{\max}^{l,w}(\eta)$ or those $\beta \in A_{k}^{\eta}(w)$ for which $\beta_{\min}^{k,w}(\eta)\leq\beta\leq\beta_{\max}^{k,w}(\eta)$, up to interchanging $\alpha$ and $\beta$ in case $k=l$. This set can be partially ordered, say using the order on $\beta$ (where for $k=l$ we may choose some element of $\eta_{1}$ and use it in order to determine which of our summands is $\alpha$ and which is $\beta$, and then apply the usual ordering), and the multiplicity depends only on the type of this set as a poset. Moreover, if the poset of one multi-set $\eta$ and permutation $w$ can be embedded, as a poset, in another such poset, then the multiplicity of the larger poset is larger than or equal to that of $\eta$ and $w$.

From checking several additional cases, we pose the following conjecture, stating that this is the case for all values of $r$.
\begin{conj}
The multiplicity $m^{k,l}_{\eta}(w)$ depends only on the poset type of the set of sums $\alpha+\beta=\eta$, with $\alpha \in A_{l}(w)$ and $\beta \in A_{k}(w)$. Moreover, when one poset can be embedded into another, the associated multiplicity either remains the same or increases. \label{poset}
\end{conj}

We will pose a weaker conjecture, as assuming it will suffice to getting the
\begin{conj}
Take $k$, $l$, $w$, and $i$ with $\ell(s_{i}w)>\ell(w)$. Then for $\eta$ of the sort $\eta^{01}$, $\eta^{02}$, or $\eta^{12}$ we have $m^{k,l}_{\eta}(w) \leq m^{k,l}_{s_{i}\eta}(w)$. \label{siinc}
\end{conj}
Note that for $\eta$ as in Conjecture \ref{siinc}, part $(ii)$ of Lemma \ref{siwsets} and its proof imply that the set of sums associated with $w$ and $s_{i}\eta$ is isomorphic, via $s_{i}$, to that of $s_{i}w$ and $\eta$ as posets, and the latter clearly contains the set of sums corresponding to $w$ and $\eta$. Hence Conjecture \ref{poset} implies Conjecture \ref{siinc}.

The first question for further research is to prove Conjecture \ref{poset}, or at least Conjecture \ref{siinc}, since the latter produces the following consequence.
\begin{thm}
If Conjecture \ref{siinc} holds, then the polynomial $P_{w,2}$ from Proposition \ref{expPwinT} is given by \[P_{w,2}(x,t)=\sum_{k \leq l}\sum_{\eta \in B_{k,l}(w)}m^{k,l}_{\eta}(w)x^{\eta}T_{k}T_{l},\quad\mathrm{with\ }m^{k,l}_{\eta}(w)>0\mathrm{\ with\ }\eta \in B_{k,l}(w).\] In addition, for each $k$, $l$, and $\eta$ we have $m^{k,l}_{\eta}(s_{i}w) \geq m^{k,l}_{\eta}(w)$ for any $i$ with $\ell(s_{i}w)>\ell(w)$, and if for $\eta \in B_{k,l}(w)$ we set $r=k-\delta_{k,l}-|\eta_{2}|$, then we have $m^{k,l}_{\eta}(w)\geq2^{r}-1$. \label{formPw2}
\end{thm}

\begin{proof}
We begin by proving the inequality involving $s_{i}w$ for such $i$, using which we determine the type of $\eta$. When $\eta$ is $\eta^{00}$, $\eta^{10}$, $\eta^{20}$, $\eta^{21}$, or $\eta^{22}$, the asserted inequality holds as an equality by Lemma \ref{multsiw}. If $\eta$ is either $\eta^{01}$, $\eta^{02}$, or $\eta^{12}$, then the asserted inequality is a direct consequence of Lemma \ref{multsiw} and Conjecture \ref{siinc}. When $\eta=\eta^{11}$, the left hand side of the inequality consists of the right hand side, of the term involving $N_{w,i}^{(2)}$ (which is always non-negative by Lemma \ref{Nwipos}), and the a difference which is also non-negative by Conjecture \ref{siinc}. This proves the desired inequality in all cases (note that if we assume the stronger Conjecture \ref{poset}, then the natural embedding of the poset associated with $w$ and $\eta$ into that of $s_{i}w$ and $\eta$ yields this inequality immediately).

Now, Theorem \ref{mposBkl} implies that $P_{w,2}$ can be described as such a sum over $\eta \in B_{k,l}(w)$, and we wish to prove that if $r=k-\delta_{k,l}-|\eta_{2}|$ for such $\eta$, which satisfies $r\geq1$ by Proposition \ref{Bkldiff}, then $m_{\eta}(w)\geq2^{r}-1$, and in particular $m_{\eta}(w)>0$. As in the proofs of Theorems \ref{diff1} and \ref{diff2} and Proposition \ref{lketa23}, there is nothing to prove when $w=\operatorname{Id}$, so we assume that the result is true for some $w$ and take an index $i$ for which $\ell(s_{i}w)>\ell(w)$. We saw in the proof of Theorems \ref{diff1} that if $\eta \in B_{k,l}(s_{i}w)$ and the type of $\eta$ is $\eta^{00}$, $\eta^{10}$, $\eta^{20}$, $\eta^{21}$, or $\eta^{22}$ then $\eta \in B_{k,l}(w)$ already and thus Lemma \ref{multsiw} and the induction hypothesis give $m_{\eta}(s_{i}w)=m_{\eta}(w)=2^{r}-1>0$. The same proof also showed that if $\eta$ is either $\eta^{01}$, $\eta^{02}$, or $\eta^{12}$ then $s_{i}\eta \in B_{k,l}(w)$, and as the value of $r$ is clearly the same for $\eta$ and $s_{i}\eta$, we get $m_{\eta}(s_{i}w)=m_{s_{i}\eta}(w)=2^{r}-1>0$ via Lemma \ref{multsiw} and the induction hypothesis.

For $\eta=\eta^{11} \in B_{k,l}(s_{i}w)$, recall from Corollary \ref{meta11siw} that $m_{\eta^{11}}(s_{i}w)$ equals $m_{\eta^{11}}(w)+\big(m_{\eta^{20}}(w)-m_{\eta^{02}}(w)\big)$ when $\eta$ was already in $B_{k,l}(w)$, and is given by $m_{\eta^{20}}(w)+m_{\eta^{02}}(N_{w,i}^{(2)})$ if it was not. In the former case the first summand is at least $2^{r}-1>0$ by the induction hypothesis, and the difference is non-negative by Conjecture \ref{siinc} again. In the latter case, Remark \ref{inddecom} implies that the value associated with $\eta^{20}$ is $r-1$, and as Lemma \ref{Nwipos} implies that the latter sequence is in $\tilde{B}_{k,l}(w)$, the first summand is at least $2^{r-1}-1\geq0$. The second one is positive (by Lemma \ref{Nwipos}), so that the positivity of $m_{\eta}$ for all $\eta \in B_{k,l}(w)$ is now established, but for the bound it remains to prove that it equals at least $2^{r-1}$.

But Corollary \ref{setofpairs} implies that evaluates this number as the number of sequences $\beta$ that satisfy $\eta^{20}_{2}\subseteq\beta\subseteq\eta^{20}$ and $\beta_{\min}^{k,w}(\eta^{20})\leq\beta\leq\beta_{\max}^{k,w}(\eta^{20})$ (divided by 2 in case $k=l$ due to the interchanges), and we recall from Proposition \ref{contmax} that these two bounding sequence consist of $\eta^{20}_{2}$ (of size $k-\delta_{k,l}-r+1$) and two disjoint subsets of $\eta^{20}_{1}$, of both of size $\tilde{r}=r-1+\delta_{k,l}$. If we denote the additional elements of $\beta_{\max}^{k,w}(\eta^{20})$ by $(\gamma_{1}^{+},\ldots,\gamma_{\tilde{r}}^{+})$ and those of $\beta_{\min}^{k,w}(\eta^{20})$ by $(\gamma_{1}^{-},\ldots,\gamma_{\tilde{r}}^{-})$, then we get $\gamma_{j}^{-}<\gamma_{j}^{+}$ for every $1 \leq j\leq\tilde{r}$. Hence for every subset $I$ of the numbers between 1 and $\tilde{r}$, the element $\beta_{I}$ consisting of $\eta^{20}_{2}$, $\gamma_{j}^{+}$ for $j \in I$, and $\gamma_{j}^{-}$ for all $j \not\in I$, is in the desired set. This gives $2^{\tilde{r}}$ distinct elements of our set of pairs, and note that if $k=l$ then $I$ and its complement yield the same sum up to interchanges, so we take half of that sum. In total this produces $2^{\tilde{r}}/2^{\delta_{k,l}}=2^{r-1}$ elements contributing to $m_{\eta^{02}}(N_{w,i}^{(2)})$, yielding the desired bound. This completes the proof of the theorem.
\end{proof}
Note that the bounds for $r=1$ and $r=2$ in Theorem \ref{formPw2} are given in Theorem \ref{diff1} and Corollary \ref{lowbdr2}, and in fact, the case $r=0$ follows from Proposition \ref{Bkldiff} and Theorem \ref{mposBkl}. We also saw in the proof of Theorem \ref{diff1} and Proposition \ref{lketa23} that the bound $2^{r-1}$ for the multiplicity $m_{\eta^{02}}(N_{w,i}^{(2)})$ (when it is positive) is just a bound, and the value itself can be larger (indeed, in some $\beta$'s we can have a $j$th entry, appropriately located, that is strictly between $\gamma_{j}^{-}$ and $\gamma_{j}^{+}$ in case their difference is larger than 1).

\smallskip

We consider briefly the next degree part of $P_{w}$ from Definition \ref{Pwdef} and Theorem \ref{formofKw}, by writing the $O(T^{3})$ part from Proposition \ref{expPwinT} as $P_{w,3}+O(T^{4})$, with $P_{w,3}$ coming with a positive sign. For $P_{w,3}$ we have to consider multi-sets $\tau$ of the form $\tau_{1}+2\tau_{2}+3\tau_{3}$ (as a multi-set which can be obtained as the sum of three sets contains elements with multiplicities at most 3), and given three integers $p \leq k \leq l$, we define \[\tilde{C}_{p,k,l}(w)\!:=\!\{\tau|\exists\alpha \in A_{l}(w), \beta \in A_{k}(w),\mathrm{\ and\ }\gamma \in A_{p}(w)\mathrm{\ such\ that\ }\tau=\alpha+\beta+\gamma\}\] to be our ambient set. The set of interest $C_{p,k,l}(w)$ now consists of those $\tau\in\tilde{C}_{p,k,l}(w)$ for which there exists a presentation $\alpha+\beta+\gamma$ for which $\alpha+\beta \in B_{k,l}(w)$, there is one for which $\alpha+\gamma \in B_{p,l}(w)$, and in some presentation we have $\beta+\gamma \in B_{p,k}(w)$.

The analogue of Theorem \ref{mposBkl} and the conditional Theorem \ref{formPw2} is as follows.
\begin{conj}
The polynomial $P_{w,3}$ can be written as in the formula \[P_{w,3}(x,t)=\sum_{p \leq k \leq l}\sum_{\tau \in C_{p,k,l}(w)}m^{p,k,l}_{\tau}(w)x^{\tau}T_{p}T_{k}T_{l},\] namely the multiplicity $m^{p,k,l}_{\tau}(w)$ can be non-zero only for $\tau \in C_{p,k,l}(w)$. Moreover, this multiplicity is positive for every such $\tau$. \label{formPw3}
\end{conj}
To give an example, take $w=4123 \in S_{4}$, for which $P_{w,3}(x,t)$ equals \[x_{1}x_{2}x_{3}x_{4}T_{1}^{2}T_{2}+x_{1}^{2}x_{2}x_{3}x_{4}T_{1}T_{2}^{2}+(x_{1}^{2}x_{2}^{2}x_{3}x_{4}+x_{1}^{2}x_{2}x_{3}^{2}x_{4}+x_{1}^{2}x_{2}x_{3}x_{4}^{2})T_{1}T_{2}T_{3}.\] Note that if $\tau$ represents the last term, then with $\alpha=124 \in A_{3}(w)$ we get $13+4=14+3 \in B_{1,2}(w)$, if $\beta=14 \in A_{2}(w)$ then $123+4=124+3 \in B_{1,2}(w)$, and $\gamma=4 \in A_{1}(w)$ yields $123+14=124+13 \in B_{1,2}(w)$ (thus indeed this $\tau$ is in $C_{1,2,3}(w)$), but there is no presentation $\alpha+\beta+\gamma$ for which all sums of two terms is in the appropriate set from Definition \ref{Bklwdef} (indeed, $\beta$ must be 14, $\gamma$ has to be 4, and the remaining $\alpha=123$ does not satisfy this condition). A similar phenomenon is visible in $P_{w,3}$ part of $P_{w}$ give in Equation \eqref{P31425} for $w=3142 \in S_{4}$.

As evidence for Conjecture \ref{formPw3}, we consider the usual inductive argument, where for $w=\operatorname{Id}$ there is nothing to prove again, and if $\ell(s_{i}w)>\ell(w)$ then the analogue of Equation \eqref{Pw2ind} for $P_{w,3}$ is
\begin{equation}
P_{s_{i}w,3}(x,t)=\pi_{i}\big(P_{w,3}(x,t)\big)-\pi_{i}\big(P_{w,2}(x,t)N_{w,i}^{(1)}(x,t)\big)+\pi_{i}\big(N_{w,i}^{(3)}(x,t)\big). \label{Pw3ind}
\end{equation}
We decompose $P_{w,3}$ as in Equation \eqref{decomPw2} but including terms having also 3 in their superscripts, and then the analogue of Lemma \ref{typesind} shows that for $s_{i}w$, the polynomials with superscripts 01, 02, 03, 13, 23, and 12 vanish, those having superscripts 00 and 33 are the same as those of $w$, the ones for which the superscripts are 10, 20, 30, 31, and 32 are the sums of those for $w$ and the ones for $w$ having the interchanged signs, and we have \[P_{s_{i}w,3}^{11}=P_{w,3}^{11}+P_{w,2}^{10}\tilde{N}_{w,i}^{(1)},\qquad P_{s_{i}w,3}^{22}=P_{w,3}^{22}+P_{w,2}^{21}\tilde{N}_{w,i}^{(1)},\] and \[P_{s_{i}w,3}^{21}=P_{w,3}^{21}+P_{w,3}^{12}+P_{w,2}^{20}\tilde{N}_{w,i}^{(1)}+\tilde{N}_{w,i}^{(3)}\] (note that in the second term in Equation \eqref{Pw3ind}, the product $x_{i}N_{w,i}^{(1)}(x,t)$ is $s_{i}$-invariant, and hence multiplies $\partial_{i}P_{w,2}(x,t)$, to which only the parts $P_{w,2}^{10}$, $P_{w,2}^{20}$, and $P_{w,2}^{21}$ from Equation \eqref{decomPw2} contribute).

By expressing, for such $w$ and $i$, the types of $\tau$ as $\tau^{ab}$ with $a$ and $b$ between 0 and 3 as we did for the $\eta$'s, and by omitting $p$, $k$, and $l$ from the notation $m^{p,k,l}_{\tau}(w)$ when they are clear from the context, this produces the following analogue of Lemma \ref{multsiw}.
\begin{lem}
Fix $p$, $k$, $l$ (with $p \leq k \leq l$), $w$, and $i$ with $\ell(s_{i}w)>\ell(w)$. If $\tau$ is $\tau^{00}$, $\tau^{10}$, $\tau^{20}$, $\tau^{30}$, $\tau^{31}$, $\tau^{32}$, or $\tau^{33}$, then $m_{\tau}(s_{i}w)=m_{\tau}(w)$. When $\tau$ is one of $\tau^{01}$, $\tau^{02}$, $\tau^{03}$, $\tau^{13}$, and $\tau^{23}$, we have $m_{\tau}(s_{i}w)=m_{s_{i}\tau}(w)$. The remaining four multiplicities are given by \[m_{\tau^{11}}(s_{i}w)=m_{\tau^{11}}(w)+\big(m_{\tau^{20}}(w)-m_{\tau^{02}}(w)\big)+m_{\tau^{11}}(x_{i}P_{w,2}^{10}N_{w,i}^{(1)}),\] \[m_{\tau^{22}}(s_{i}w)=m_{\tau^{22}}(w)+\big(m_{\tau^{31}}(w)-m_{\tau^{13}}(w)\big)+m_{\tau^{22}}(x_{i}^{2}x_{i+1}P_{w,2}^{21}N_{w,i}^{(1)}),\] and $m_{\tau^{21}}(s_{i}w)$ and $m_{\tau^{12}}(s_{i}w)$ are both equal to \[m_{\tau^{12}}(w)+\big(m_{\tau^{30}}(w)-m_{\tau^{03}}(w)\big)+m_{\tau^{21}}(x_{i}^{2}P_{w,2}^{21}N_{w,i}^{(1)})+m_{\eta^{03}}(N_{w,i}^{(3)}).\] \label{siwmults}
\end{lem}
Using Lemma \ref{siwmults}, the weak part of Conjecture \ref{formPw3} can be covered in most cases of the inductive argument. It is clear that the sets $C_{p,k,l}(w)$ (as well as $\tilde{C}_{p,k,l}(w)$) can only increase when $w$ is replaced by $s_{i}w$ of longer length, with the $s_{i}$-images also joining (as in Proposition \ref{Bsiw}). Moreover, the proof of Lemma \ref{siwsets} shows that if $\tau$ is one of the first 7 types from Lemma \ref{siwmults}, then it lies in $C_{p,k,l}(s_{i}w)$ if and only if it is in $C_{p,k,l}(w)$, and when the type of $\tau$ is one of the following 5 ones, it is in $C_{p,k,l}(s_{i}w)$ precisely when $s_{i}\tau \in C_{p,k,l}(w)$. Arguments as above also show that if $\tau^{20}$ or $\tau^{02}$ are in $C_{p,k,l}(w)$ then $\tau^{11} \in C_{p,k,l}(s_{i}w)$, when $\tau^{31}$ or $\tau^{13}$ lie in $C_{p,k,l}(w)$ we have $\tau^{22} \in C_{p,k,l}(s_{i}w)$, and in case either $\tau^{30}$ or $\tau^{03}$ belong to $C_{p,k,l}(w)$, both $\tau^{12}$ and $\tau^{21}$ are elements of $C_{p,k,l}(s_{i}w)$. In fact, it is also not hard to verify that when $\eta^{03}$ shows up in $N_{w,i}^{(3)}$, both $\tau^{12}$ and $\tau^{21}$ must be in $C_{p,k,l}(s_{i}w)$, and get the same consequence in case $x_{i}^{2}P_{w,2}^{21}N_{w,i}^{(1)}$ contains the term associated with $\tau^{21}$ (with our $p$, $k$, and $l$). The claims that when the term corresponding to $\tau^{11}$ (resp. $\tau^{22}$) shows up in $x_{i}P_{w,2}^{10}N_{w,i}^{(1)}$ (resp. $x_{i}^{2}x_{i+1}P_{w,2}^{21}N_{w,i}^{(1)}$) then this $\tau$ lies in $C_{p,k,l}(s_{i}w)$ is true in all the cases we checked, but seems to be more delicate (and depend on Conjecture \ref{poset} or at least Conjecture \ref{siinc}).

\smallskip

The formula of $K_{w}(x,t)$ from Theorem \ref{formofKw} suggests a relation between our generating series and the series \[\mathcal{F}_{w}(x,t):=\frac{1}{\prod_{l=1}^{\infty}\prod_{\alpha \in A_{l}(w)}(1-x^{\alpha}T_{l})}\] (or with the superscript $n$). If we consider that series also as a generating function, namely write $\mathcal{F}_{w}(x,t)$ as $\sum_{\lambda}F_{\lambda,w}(x)t^{\lambda}$, then the $F_{\lambda,w}$'s have the following combinatorial description.
\begin{prop}
For a partition $\lambda$, $F_{\lambda,w}$ is a polynomial in $x$, and set $h_{l}:=\lambda_{l}-\lambda_{l+1}$ for every $l\geq1$ as in the proof of Proposition \ref{KId}. Then, for every multi-set $\mu$, the coefficient with which a monomial $x^{\mu}$ shows up in $F_{\lambda,w}$ is the number of ways to take, for every $l$, an unordered collection of $h_{l}$ elements of $A_{l}(w)$, with multiplicities allowed, such that the sum of all those elements, over all $l$, is $\mu$. \label{Fwcoeffs}
\end{prop}

\begin{proof}
As we saw in the proof of Proposition \ref{KId}, the series $\mathcal{F}_{w}(x,t)$ is a product of geometric series, which we can write as $\prod_{l=1}^{\infty}\prod_{\alpha \in A_{l}(w)}\sum_{g_{\alpha}=0}^{\infty}(x^{\alpha}T_{l})^{g_{\alpha}}$. The contributions to $F_{\lambda,w}(x)t^{\lambda}$ are precisely from those choices of $g_{\alpha}$'s such that $\prod_{l=1}^{\infty}\prod_{\alpha \in A_{l}(w)}T_{l}^{g_{\alpha}}=t^{\lambda}$, which means, via the proof of Proposition \ref{KId}, that $\sum_{\alpha \in A_{l}(w)}g_{\alpha}=h_{l}$ for every $l\geq1$. Noting that $h_{l}=0$ for every $l\geq\ell(\lambda)$, we have $g_{\alpha}=0$ for each $\alpha$ with such $l$, meaning that there are finitely many non-zero $g_{\alpha}$'s with bounded sums, proving that indeed $F_{\lambda,w}$ is a polynomial.

For each such choice of $g_{\alpha}$'s, the resulting monomial in $F_{\lambda,w}$ is $x^{\mu}$, where $\mu=\sum_{l=1}^{\infty}\sum_{\alpha \in A_{l}(w)}g_{\alpha}\alpha$. Viewing any choice of $\{g_{\alpha}\}_{\alpha}$ as taking an unordered multi-set of elements of $A(w)$, the equalities involving $h_{l}$ above mean that for each $l$ we take precisely $h_{l}$ elements from $A_{l}(w)$ (with multiplicities), and the last equality means that the sum of all the elements of the multi-set is $\mu$. Hence when $\mu$ is fixed, we get one contribution of $x^{\mu}$ for every such multi-set, as desired. This proves the proposition.
\end{proof}

\begin{rmk}
We can view the coefficients from Proposition \ref{Fwcoeffs} as counting integral points on polytopes. Indeed, for every $l$ the number of options to take $h_{l}$ unordered elements from $A_{l}(w)$ is the number of integral points on a simplex of dimension $|A_{l}(w)|-1$ (these are the solutions to the equation $\sum_{\alpha \in A_{l}(w)}g_{\alpha}=h_{l}$), and by doing it over all $l$ we obtain a product of simplices (finitely many, as Remark \ref{forId} shows that $|A_{l}(w)|=1$ for any large enough $l$), hence a polytope. Note that some of the trivial multipliers are not the origin point, as $\lambda$ and $w$ are unrelated, and we might have $h_{l}>0$ for some $w$ with $|A_{l}(w)|=1$. By choosing a multi-set $\mu$, the restriction that the sum of the elements is $\mu$ intersects this polytope with a rational hyperplane, and the multiplicity from Proposition \ref{Fwcoeffs} counts the integral points on this intersection polytope. \label{ptspoly}
\end{rmk}
In view of Proposition \ref{Fwcoeffs} and Remark \ref{ptspoly}, we may interpret Theorem \ref{formofKw} and the form of $P_{w}$ as expressing the coefficient in which a monomial $x^{\mu}$ shows up in $K_{\lambda,w}$ as a combination of cardinalities of sets of integral points on polytopes. More precisely, take $\lambda$ and $\mu$ as in Proposition \ref{Fwcoeffs}, and we consider the coefficient in which $x^{\mu}$ shows up in $K_{\lambda,w}$. The term 1 in $P_{w}$ means that the ``zeroth order approximation'' of that coefficient is the one in which it appears in $F_{\lambda,w}$, which is a number of integral points on a polytope as in Remark \ref{ptspoly}, and the fact that there are no linear terms in $P_{w}$ (via Proposition \ref{expPwinT}) implies that this is also the ``first order approximation''.

Consider now $k$ and $l$ with $\lambda_{k}>\lambda_{k+1}$ and $\lambda_{l}>\lambda_{l+1}$ (or $k=l$ and $\lambda_{l}\geq\lambda_{k+1}+2$), and some $\eta \in B_{k,l}(w)$ such that $\eta\subseteq\mu$ as multi-sets (namely every element showing up in $\eta$ appears also in $\mu$, and with at least the same multiplicity). Then the appearance of the term $m^{k,l}_{\eta}(w)x^{\eta}T_{k}T_{l}$ in $P_{w,2}$ (with the minus sign) implies that to get the multiplicity of $x^{\mu}$ in $K_{\lambda,w}$, we have to subtract $m^{k,l}_{\eta}(w)$ times the multiplicity of $x^{\mu-\eta}$ in $\lambda_{k,l}$, where the latter partition is with we subtract 1 from each $\lambda_{j}$ with $j \leq l$, and another 1 from any $\lambda_{j}$ with $j \leq k$. All those subtractions, for all the terms in $P_{w,2}$, produce the ``second order approximation'' of the multiplicity for $x^{\mu}$ in $K_{\lambda,w}$.

For example, if $\lambda_{1}=4$, $\lambda_{2}=2$, and the higher $\lambda_{j}$'s vanish (meaning that $h_{1}=h_{2}=2$ and the other $h_{l}$'s vanish), and $w=321$, then there are 6 ways to express $\mu$ with $x^{\mu}=x_{1}^{2}x_{2}^{2}x_{3}^{2}$ using two elements of $A_{1}(w)$ and two from $A_{2}(w)$, namely \[2(12)+2(3),\ 12+13+2+3,\ 2(13)+2(2),\ 12+23+1+3,\ 13+23+1+2,\ 2(23)+2(1).\] We have $P_{w}=1-x_{1}x_{2}x_{3}T_{1}T_{2}$, so that from this 6 we have to subtract the number of ways to express $\mu-\eta$ for that $\eta$ with $\lambda_{1,2}$. In that partition we have $h_{1}=h_{2}=1$, with the three ways $12+3=13+2=23+1$ for obtaining $\mu-\eta$, and indeed $K_{\lambda,w}$ with these $w$ and $\lambda$ contains $x_{1}^{2}x_{2}^{2}x_{3}^{2}$ with multiplicity 3.

The ``higher order approximations'' are obtained by taking into considerations the terms arising from the higher order terms in $P_{w}$. Note that if Conjecture \ref{formPw3} holds (even in its weaker form), then elements of $P_{w,3}$ cancel out negative contributions that come several times from $P_{w,2}$, exhibiting a kind of ``inclusion-exclusion'' behavior. It would be interesting to see whether there is some relation between these considerations and the constructions, described in, e.g., Section 2 of \cite{[P]} (but based on \cite{[Ma1]}, \cite{[Ma2]}, and others) involving semi-standard augmented fillings of augmented skyline diagrams. More precisely, if is a definition producing the $F_{\lambda,w}$'s using less restricted fillings, such that the additional restrictions that produce the $K_{\lambda,w}$'s are obtained by subtracting combinations arising from the terms of $P_{w,2}$, with those from $P_{w,3}$ compensating in an inclusion-exclusion manner, and to find, in this case, what is the effect of the higher degree parts of $P_{w}$. One may also ask whether our results can shed light on the connection to Gelfand--Tsetlin polytopes as in \cite{[KST]}, or to the reduced cases of the set-valued tableaux from \cite{[Y]}.

\smallskip

We conclude with the generalization to the Lascoux polynomials. Take another variable $\xi$, and recall that the operators \[\pi_{i}^{(\xi)}(f):=\pi_{i}\big((1+\xi x_{i+1})f\big)=\partial_{i}\big((x_{i}+\xi x_{i}x_{i+1})f\big)\] also satisfy $(\pi_{i}^{(\xi)})^{2}=\pi_{i}^{(\xi)}$ and the braid relations. Then the \emph{$\xi$-Lascoux polynomial} $L_{\lambda,w}^{\xi}(x)$ is defined to be $\pi_{w}^{(\xi)}(x^{\lambda})$. They are (at least for $\xi=-1$) the $K$-theoretic analogues of the key polynomials, and they reproduce the key polynomials when one substitute $\xi=0$. Like the key polynomials, they are homogeneous of degree $\sum_{i=1}^{\infty}\lambda_{i}$, if $\xi$ has homogeneity degree $-1$.

As in Equation \eqref{genser}, we can consider the generating functions \[\mathcal{L}^{(n)}(x,t):=\sum_{\ell(\lambda) \leq n}\sum_{w \in S_{n}}L_{\lambda,w}(x)t^{\lambda}\varepsilon_{ww_{0}}=\sum_{w \in S_{n}}\mathcal{L}_{w}^{(n)}(x,t)\varepsilon_{ww_{0}}\] of the $\xi$-Lascoux polynomials, and Lemma \ref{piiKw} and Proposition \ref{propgen} extend, in a similar manner to $\mathcal{L}^{(n)}(x,t)$ and to its components $\mathcal{L}_{w}^{(n)}(x,t)$ (or their analogues $\mathcal{L}_{w}(x,t)$ without the superscript) respectively. Since $\mathcal{L}_{\operatorname{Id}}^{(n)}(x,t)=\mathcal{K}_{\operatorname{Id}}^{(n)}(x,t)$ and $\mathcal{L}_{\operatorname{Id}}(x,t)=\mathcal{K}_{\operatorname{Id}}(x,t)$ by definition, the proof of Theorem \ref{formofKw} shows that \[\mathcal{L}_{w}^{(n)}(x,t)=\frac{P_{w}^{(\xi)}(x,t)}{\prod_{l=1}^{n}\prod_{\alpha \in A_{l}(w)}(1-x^{\alpha}T_{l})},\ \mathcal{L}_{w}(x,t)=\frac{P_{w}^{(\xi)}(x,t)}{\prod_{l=1}^{\infty}\prod_{\alpha \in A_{l}(w)}(1-x^{\alpha}T_{l})},\] where $P_{\operatorname{Id}}^{(\xi)}:=1$ and $P_{s_{i}w}^{(\xi)}:=\pi_{i}^{(\xi)}(P_{w}^{(\xi)}N_{w,i})$ whenever $\ell(s_{i}w)>\ell(w)$ by the appropriately modified Definition \ref{Pwdef}. Following the proof of Proposition \ref{expPwinT}, one gets \[P_{w}^{(\xi)}(x,t)=1+P_{w,1}^{(\xi)}(x,t)+O(T^{2}),\] with $P_{w,1}^{(\xi)}$ non-zero in general (but divisible by $\xi$, of course). The latter polynomial is defined inductively as $P_{\operatorname{Id},1}^{(\xi)}$ as the basis, and when $\ell(s_{i}w)>\ell(w)$ we have $P_{s_{i}w,1}^{(\xi)}:=\pi_{i}^{(\xi)}(P_{w,1}^{(\xi)})+\xi x_{i}N_{w,i}^{(1)}$. The first part of that polynomial is given by the following result.
\begin{prop}
For any permutation $w$ and any $l$, let $A_{l}^{(+1)}(w)$ be the set of increasing sequences of length $l+1$ that can be obtained by adding to an element of $A_{l}(w)$ a number that is not already contained in it. We then have \[P_{w,1}^{(\xi)}(x,t)=\xi\sum_{l=1}^{\infty}\sum_{\alpha \in A_{l}^{(+1)}(w)}m^{l}_{\alpha}(w)x^{\alpha}T_{l}+O(\xi^{2}),\] where $m^{l}_{\alpha}(w)$ equals $\big|\big\{j\in\alpha|\alpha\setminus\{j\} \in A_{l}(w)\big\}\big|+1$ for any such $l$ and $\alpha$. \label{Pxiw1}
\end{prop}
Note that the multiplicity $m^{l}_{\alpha}(w)$ can be positive only when $A_{l}(w)$ contains more than one element (as it contains $\alpha\setminus\{j\}$ for more than one value of $j$), which by Remark \ref{forId} happens only for finitely many values of $l$, and thus finitely many values of $\alpha$ (those that are obtained as the union of two elements of $A_{l}(w)$ that intersect at $l-1$ elements). Hence the expression from Proposition \ref{Pxiw1} is indeed a polynomial.

\begin{proof}
As always, we apply the usual induction, where for $w=\operatorname{Id}$ the value of $P_{w,1}^{(\xi)}$ is 0, and indeed there are no $l$ and $\alpha$ with non-zero multiplicity. So assume that the result holds for $w$, take $i$ with $\ell(s_{i}w)>\ell(w)$, and then we can decompose the asserted value of $P_{w,1}^{(\xi)}$ as the sum of $P_{w,1}^{(\xi),00}$, $(x_{i}+x_{i+1})P_{w,1}^{(\xi),01}$, $x_{i}P_{w,1}^{(\xi),10}$, and $x_{i}x_{i+1}P_{w,1}^{(\xi),11}$, analogously to Equation \eqref{decomPw2}.

Consider some increasing sequence $\alpha$ of length $l+1$, whose type, using notation similar to the one from above, can be $\alpha^{00}$, $\alpha^{10}$, $\alpha^{01}$, or $\alpha^{11}$. Noting that $N_{w,i}^{(1)}$ only involves elements containing $x_{i+1}$ and hence $\xi x_{i}N_{w,i}^{(1)}$ only affects elements of type $\alpha^{11}$, adding 1 to the multiplicity when $\alpha\setminus\{i+1\} \in A_{l,i}(w)$ and 0 otherwise. The argument proving Lemma \ref{multsiw} shows that if $\alpha$ is $\alpha^{00}$ or $\alpha^{10}$ then $m^{l}_{\alpha}(s_{i}w)=m^{l}_{\alpha}(w)$, for $\alpha=\alpha^{01}$ we have $m^{l}_{\alpha}(s_{i}w)=m^{l}_{s_{i}\alpha}(w)$, and in case $\alpha=\alpha^{11}$, the value of $m^{l}_{\alpha}(s_{i}w)$ is $m^{l}_{\alpha}(w)+1$ when $\alpha\setminus\{i+1\} \in A_{l,i}(w)$, and $m^{l}_{\alpha}(w)$ otherwise.

We need to verify that our formula for the multiplicity satisfies the same rules. For $\alpha$ which is $\alpha^{00}$ or $\alpha^{10}$, and any $j\in\alpha$, we saw, as in the proof of part $(i)$ of Lemma \ref{siwsets}, that $\alpha\setminus\{j\}$ is in $A_{l}(s_{i}w)$ if and only if it is in $A_{l}(w)$. When $\alpha=\alpha^{01}$, the argument proving part $(ii)$ of Lemma \ref{siwsets} shows that the picture for $\alpha$ is the $s_{i}$-image of that of $s_{i}\alpha$, and the latter is again the same for $w$ and for $s_{i}w$. Finally, when $\alpha=\alpha^{11}$, all the sequences $\alpha\setminus\{j\}$ for $i \neq j\in\alpha$ again lie in $A_{l}(s_{i}w)$ if and only if it is in $A_{l}(w)$, and the $\alpha\setminus\{i\}$ lies in $A_{l}(s_{i}w)$ but not in $A_{l}(w)$ precisely when $\alpha\setminus\{i+1\} \in A_{l,i}(w)$ (by Proposition \ref{Asiw} as always). Hence the desired equality holds for all $\alpha$. This proves the proposition.
\end{proof}
It would be interesting, in that direction of research, to find the combinatorics behind the higher orders of $\xi$ in the expansion from Proposition \ref{Pxiw1}, as well as the additional terms that show up in $P_{w,2}^{(\xi)}$ and the other parts of $P_{w}^{(\xi)}$.

\medskip

\noindent\textsc{Department of Mathematics, Williams College, Williamstown, MA 01267, USA}

\noindent E-mail address: nvc2@williams.edu

\noindent\textsc{Einstein Institute of Mathematics, the Hebrew University of Jerusalem, Edmund Safra Campus, Jerusalem 91904, Israel}

\noindent E-mail address: zemels@math.huji.ac.il

\end{document}